\documentclass[a4paper, reqno]{amsart}

\usepackage{latexsym}
\usepackage{amsmath}
\usepackage{amsfonts}
\usepackage{amssymb}
\usepackage{amsxtra}
\usepackage{mathrsfs}
\usepackage{eucal}
\usepackage{tikz} \usetikzlibrary{calc} \tikzset{>=latex} \usetikzlibrary{backgrounds}
\usepackage{tikz-cd}
\usepackage[all]{xy}
\usepackage{color}
\usepackage{braket}
\usepackage{graphics, setspace}
\usepackage{etoolbox, textcomp}
\usepackage{comment}
\usepackage{changepage}
\usepackage{xcolor}
\definecolor{rouge}{rgb}{0.85,0.1,.4}
\definecolor{bleu}{rgb}{0.1,0.2,0.9}
\definecolor{violet}{rgb}{0.7,0,0.8}
\usepackage[colorlinks=true,linkcolor=bleu,urlcolor=violet,citecolor=rouge]{hyperref}
\usepackage{cleveref}
\usepackage{lscape}
\usepackage{subcaption}
\usepackage{hyperref}
\usepackage{mathtools}

\RequirePackage[numbers,sort&compress]{natbib}

\newtheorem{theorem}{Theorem}[section]
\newtheorem{lemma}[theorem]{Lemma}
\newtheorem{setup}[theorem]{Setup}
\newtheorem{proposition}[theorem]{Proposition}
\newtheorem{corollary}[theorem]{Corollary}

\theoremstyle{definition}
\newtheorem{definition}[theorem]{Definition}
\newtheorem{example}[theorem]{Example}
\newtheorem{remark}[theorem]{Remark}

\def\CC{\mathbb{C}}

\def\KK{\mathbb{K}}

\def\NN{\mathbb{Z}_{\ge 0}}

\def\RR{\mathbb{R}}

\def\ZZ{\mathbb{Z}}

\newcommand\cA{\mathcal{A}}
\newcommand\cB{\mathcal{B}}
\newcommand\cC{\mathcal{C}}
\newcommand\cD{\mathcal{D}}
\newcommand\cE{\mathcal{E}}

\newcommand\cM{\mathcal{M}}
\newcommand\cN{\mathcal{N}}

\newcommand\cP{\mathcal{P}}

\newcommand\cS{\mathcal{S}}

\newcommand\cW{\mathcal{W}}

\newcommand\cY{\mathcal{Y}}
\newcommand\cZ{\mathcal{Z}}

\newcommand{\cEW}{\mathcal{EW}}

\newcommand\tilA{\widetilde{A}}

\newcommand\tilL{\widetilde{L}}

\newcommand\tile{\widetilde{e}}
\newcommand\tilf{\widetilde{f}}

\newcommand{\vac}{\mathbf{1}}

\newcommand{\Gr}{\textup{Gr}}

\newcommand\id{\textup{id}}
\newcommand{\Id}{\mathrm{id}}
\newcommand{\op}{\mathrm{op}}

\newcommand\Rep{\textup{Rep}}

\newcommand{\loc}{\textup{loc}}
\newcommand{\radj}{\textup{ra}}
\newcommand{\ladj}{\textup{la}}

\newcommand\FP{\textup{FPdim}}

\newcommand\Hom{\textup{Hom}}
\newcommand\End{\textup{End}}

\newcommand\coev{\textup{coev}}
\newcommand\ev{\textup{ev}}
\newcommand\Rex{\textup{Rex}}

\newcommand\Fun{\textup{Fun}}

\newcommand\quash[1]{}

\newcommand\one{\mathbf{1}}

\newcommand{\trl}{\triangleleft}
\newcommand{\yes}{\textcolor{blue}{\checkmark}}
\newcommand{\no}{\textcolor{red}{X}}
\newcommand{\open}{\textcolor{violet}{?}}

\def \<{\langle}
\def \>{\rangle}

\newcommand{\bea}{\begin{eqnarray}}
\newcommand{\eea}{\end{eqnarray}}
\newcommand{\be}{\begin {equation}}
\newcommand{\ee}{\end{equation}}
\newcommand{\g}{\mathfrak g}

\begin{document}

\newcommand{\ssl}{\mathfrak{sl}_{2|1}}
\newcommand{\vir}{\mathcal{M}}
\newcommand{\virp}{\vir_{>0}}
\newcommand{\virb}{\vir_{\ge0}}
\newcommand{\virz}{\vir_0}
\newcommand{\virn}{\vir_{<0}}

\newcommand{\uea}[1]{\mathcal{U}(#1)}

\newcommand{\hw}{highest-weight}
\newcommand{\foh}[1]{\textcolor{blue}{#1}}  
\newcommand{\jy}[1]{\textcolor{red}{#1}}
\newcommand{\oc}{\mathcal{O}_c}
\newcommand{\ocfin}{\oc^{\textup{fin}}}
\newcommand{\ocleft}{\oc^{\textup{L}}}
\newcommand{\ocright}{\oc^{\textup{R}}}
\newcommand{\cofcat}{\mathcal{C}_1}

\title[Commutative algebras, rigidity, and VOAs]{
Commutative algebras in Grothendieck-Verdier categories, rigidity, and vertex operator algebras
}
\subjclass[2020]{Primary 18M15, 17B69, 18M20, 81R10}
\author{Thomas Creutzig}
\address{Department Mathematik, Friedrich-Alexander Universit\"at Erlangen-N\"urnberg, 
91058 Erlangen, Germany}
\email{thomas.creutzig@fau.de}

\author{Robert McRae}
\address{Yau Mathematical Sciences Center, Tsinghua University, Beijing 100084, China}
\email{rhmcrae@tsinghua.edu.cn}

\author{Kenichi Shimizu}
\address{Department of Mathematical Sciences, Shibaura Institute of Technology, Saitama 337-8570, Japan}
\email{kshimizu@shibaura-it.ac.jp}

\author{Harshit Yadav}
\address{Department of Mathematical and Statistical Sciences, University of Alberta, Edmonton, Alberta, Canada T6G 2G1}
\email{hyadav3@ualberta.ca}




\begin{abstract}
Let $A$ be a commutative algebra in a braided monoidal category $\mathcal{C}$. For example, $A$ could be a vertex operator algebra (VOA) extension of a VOA $V$ in a category $\mathcal{C}$ of $V$-modules.
We first find conditions for the category $\mathcal{C}_A$ of $A$-modules in $\mathcal{C}$ and its subcategory $\mathcal{C}_A^{\loc}$ of local modules to inherit rigidity from $\mathcal{C}$. Second and more importantly, we prove a converse result, finding conditions under which $\cC$ and $\cC_A$  inherit rigidity from $\cC_A^\loc$. 

For our first results, we assume that $\mathcal{C}$ is a braided finite tensor category and identify mild conditions under which $\mathcal{C}_A$ and $\mathcal{C}_A^{\loc}$ are also rigid. These conditions are based on criteria due to Etingof and Ostrik for $A$ to be an exact algebra in $\mathcal{C}$. As an application, we show that if $A$ is a simple $\mathbb{Z}_{\geq 0}$-graded VOA containing a strongly rational vertex operator subalgebra $V$, then $A$ is also strongly rational, without requiring the dimension of $A$ in the modular tensor category of $V$-modules to be non-zero. We also identify conditions under which the category of $A$-modules inherits rigidity from the module category of a $C_2$-cofinite non-rational subalgebra $V$. 

For our converse result, we assume that $\mathcal{C}$ is a Grothendieck-Verdier category, which means that $\mathcal{C}$ admits a weaker duality structure than rigidity. We first show that $\mathcal{C}_A$ is also a Grothendieck-Verdier category. Using this, we then prove that if $\mathcal{C}_A^{\loc}$ is rigid, then so is $\mathcal{C}$ under conditions that include a mild non-degeneracy assumption on $\mathcal{C}$, as well as assumptions that every simple object of $\mathcal{C}_A$ is local and that induction $F_A:\mathcal{C}\rightarrow\mathcal{lC}_A$ commutes with duality. These conditions
are motivated by free field-like VOA extensions $V\hookrightarrow A$ where $A$ is often an indecomposable $V$-module, and thus our result will make it more feasible to prove rigidity for many vertex algebraic braided monoidal categories. In a follow-up work, our results are used to prove rigidity of the category of weight modules for the simple affine VOA of $\mathfrak{sl}_2$ at any admissible level, which embeds by Adamovi\'{c}'s inverse quantum Hamiltonian reduction into a rational Virasoro VOA tensored with a half-lattice VOA.
\end{abstract}


\maketitle

\numberwithin{equation}{section}

\baselineskip=14pt
\newenvironment{demo}[1]{\vskip-\lastskip\medskip\noindent{\em#1.}\enspace
}{\qed\par\medskip}

\tableofcontents

\section{Introduction}

Braided monoidal categories are by now key structures linking mathematics and mathematical physics, with previous research largely focusing on modular tensor categories. These are finite, semisimple, and non-degenerate ribbon categories that emerged naturally from the basic principles of two-dimensional rational conformal field theory \cite{Moore:1988qv} and found immediate use in providing invariants of knots, links and $3$-manifolds \cite{reshetikhin1991invariants, turaev1992modular}. More recently, braided monoidal categories that may be infinite or non-semisimple are important in both low-dimensional topology and conformal field theory \cite{Lyu, CGPM, de20223, HLZ1, creutzig_gannon}. Many structural results require these categories to be ribbon, or at least rigid. Thus it is crucial to identify when new examples of monoidal categories are rigid.

Given a braided monoidal category $\cC$ and a commutative algebra $A$ in $\cC$, one gets two new monoidal categories: the category $\cC_A$ of left $A$-modules in $\cC$ and its subcategory $\cC_A^{\loc}$ of local $A$-modules \cite{pareigis1995braiding}. While $\cC_A$ is usually not braided, $\cC_A^{\loc}$ is. In this paper, we provide new criteria under which rigidity of $\cC$ implies rigidity of $\cC_A$ and $\cC_A^{\loc}$, and vice versa. If $\cC$ is not rigid but still admits a weaker duality structure called Grothendieck-Verdier duality \cite{boyarchenko2013duality}, we also show that $\cC_A$ is a Grothendieck-Verdier category.

A major source of braided monoidal categories is representations of vertex operator algebras (VOAs); see \cite{HLZ1} and references therein. Such categories also admit Grothendieck-Verdier duality structure under mild conditions \cite{ALSW}. If $\cC$ is a braided monoidal category of modules for a VOA $V$ and $A$ is a VOA which contains $V$ as a vertex operator subalgebra such that $A$ is an object of $\cC$ as a $V$-module, then $A$ is a commutative algebra in $\cC$ \cite{Huang:2014ixa}. Furthermore, the (braided) monoidal category structures on $\cC_A$ and $\cC_A^{\loc}$ are vertex algebraically natural \cite{creutzig2017tensor}. Thus a main goal of our work is to use our general categorical results to transfer rigidity from categories of $V$-modules to categories of $A$-modules, or vice versa (assuming one of these categories is already known to be rigid but the other is not).

Indeed, although it is expected that there should be reasonable general assumptions under which vertex algebraic monoidal categories are rigid, at present most of the general rigidity results that are available assume some degree of semisimplicity. Most notably, Huang showed in \cite{huang2005vertex, Huang:Rig_Mod} that the representation category of a strongly rational VOA (which is finite and semisimple by assumption) is a modular tensor category. 
In another important recent development, after the first version of this paper was completed, Etingof and Penneys \cite{etingof2024rigidity} showed that non-negligible objects of moderate growth in braided monoidal categories are rigid. These non-negligible and moderate growth conditions are immediate in essentially all finite semisimple VOA representation categories of interest, and thus such categories are now rigid. The present paper, however, is motivated by categories that are at least potentially non-semisimple, in which negligible objects in particular cannot be ruled out (and in fact often occur).


For non-finite and non-semisimple module categories for VOAs, rigidity is usually checked by tedious case-by-case analysis that requires finding and solving regular singular point differential equations for correlation functions (see \cite{Tsuchiya:2012ru, Creutzig:2020qvs, Creutzig:2020zom, McRae:2023ado}, among other works). It is not realistic to hope that such differential equations can be found or solved explicitly in general. Thus in this paper, we provide methods for establishing rigidity of braided monoidal categories from VOAs in better generality with much less work. Currently the main open rigidity problem is the rigidity of the category of weight modules for the simple affine VOA $L_k(\mathfrak{sl}_2)$ at any admissible level.
In an accompanying paper our methods are used to show that the category of weight modules for the simple affine VOA $L_k(\mathfrak{sl}_2)$ at any admissible level is indeed rigid \cite{CMY24}.

Much of our work is motivated by logarithmic Kazhdan-Lusztig correspondences, which are braided monoidal equivalences between non-semisimple module categories for VOAs and corresponding quasi-Hopf algebras. In \cite{CLR}, such correspondences are proved by showing the VOA module category $\cC$ of interest is equivalent to a relative Drinfeld center of the module category $\cC_A$ for a suitable commutative algebra $A$ in $\cC$. Under certain conditions, this relative Drinfeld center can then be identified with the category of modules for the corresponding quasi-Hopf algebra.
To show that the functor from $\cC$ to the relative Drinfeld center is fully faithful, it is assumed in \cite{CLR} that $\cC$ is rigid; now, by our work here (especially Theorem \ref{thm:faithfulness-of-I} below), this full faithfulness holds under weaker assumptions (compare with \cite[Problem 3.11]{CLR}).

\subsection{Rigidity of \texorpdfstring{$\cC_A$}{CA} and \texorpdfstring{$\cC_A^\loc$}{CA-loc} from \texorpdfstring{$\cC$}{C}}

Assume that $\cC$ is rigid and $A$ is a commutative algebra in $\cC$. Naively, one would expect $\cC_A$ and $\cC_A^\loc$ to be rigid also, perhaps under mild assumptions. Indeed, Kirillov and Ostrik proved this assuming that $\cC$ has a ribbon twist $\theta$ such that $\theta_A = \id_A$, and assuming that $A$ is what they called a rigid $\cC$-algebra (also called an \'{e}tale $\cC$-algebra elsewhere) \cite[Theorem 1.15]{kirillov2002q}. Part of the definition of rigid $\cC$-algebra is that the categorical dimension $\dim_\cC(A)$ must be non-zero. Unfortunately, this condition is difficult to verify outside of special cases (such as when $\cC$ is unitary or pseudo-unitary, in which case the categorical dimensions of non-zero objects are positive real numbers).

 In Section \ref{sec:rig-of-CA}, we identify criteria for rigidity of $\cC_A$ and $\cC_A^{\loc}$ that do not require non-vanishing of $\dim_\cC(A)$, in the case that $\cC$ is finite. The key idea is to replace rigid (or \'{e}tale) algebras in ribbon categories with exact algebras in finite tensor categories. An algebra $A$ in $\cC$ is exact if $\cC_A$ is an exact $\cC$-module category, which means that for any object $M\in\cC_A$ and projective object $P\in\cC$, the left $A$-module $M\otimes P$ is projective in $\cC_A$. It is shown in the recent work \cite{shimizu2024exact} that if $A$ is an exact commutative haploid algebra in a braided finite tensor category $\cC$, then $\cC_A$ and $\cC_A^{\loc}$ are rigid. Moreover, Etingof and Ostrik have provided good criteria for $A$ to be exact in \cite{etingof2021frobenius}. Specifically, a simple algebra $A$ in a finite tensor category is exact if and only if any of the following equivalent conditions hold:
 \begin{enumerate}
     \item[(a)] There is a left $A$-module embedding ${}^*A\hookrightarrow A\otimes X$ for some $X\in\cC$.
     \item[(b)] There is a right $A$-module embedding $A^*\hookrightarrow X\otimes A$ for some $X\in\cC$.
     \item[(c)] $A^*\otimes_A {}^*A\neq 0$.
 \end{enumerate}
 We give a detailed exposition of this important result of Etingof and Ostrik in Section \ref{subsec:exact-algebras}. Thus if $A$ is a simple commutative algebra in a braided finite tensor category $\cC$ such that any of the above three conditions holds, then $\cC_A$ and $\cC_A^{\loc}$ are finite tensor categories.
 
 As special cases of this rigidity result, we show that if $A$ is a simple commutative algebra in a braided finite tensor category $\cC$, then:
\begin{enumerate}
\item  $\cC_A$ and $\cC_A^{\loc}$ are finite tensor categories in any of the following three cases:
\begin{enumerate}
    \item $\cC$ is integral (Corollary \ref{cor:com-exact-integral}).
    \item $A$ is a Frobenius algebra, that is, $A\cong A^*$ as right $A$-modules (Corollary \ref{cor:com-exact-5}).
    \item $\Hom_{\cC}(A,\one)\neq 0$ (Corollary \ref{cor:com-exact-5}).
\end{enumerate}
\item If $\cC$ is a fusion category, then $\cC_A$ and $\cC_A^\loc$ are fusion categories (Theorem \ref{thm:comm-exact-semisimple}).
\end{enumerate}
These results apply in particular when $\cC$ is a category of modules for a VOA $V$ and $A$ is a VOA that contains $V$ as a vertex operator subalgebra and is an object of $\cC$ when viewed as a $V$-module. In this case $A$ is a commutative algebra in $\cC$ \cite{Huang:2014ixa} and $\cC_A$ and $\cC_A^{\loc}$ have natural vertex algebraic monoidal structure \cite{creutzig2017tensor}. In this case, the duals $X^*$ and ${}^*X$ of a $V$-module $X\in\cC$ are given by the contragredient module $X'$ defined in  \cite{FHL}, provided that $V$ is self-contragredient, that is, $V\cong V'$. However, the larger VOA $A$ need not be self-contragredient even if $V$ is, so duals in $\cC_A$ and $\cC_A^{\loc}$ might not be contragredients of $A$-modules (at least not contragredients with respect to the shared conformal vector of $V$ and $A$).

The nicest VOAs are strongly finite and strongly rational VOAs. A VOA $V$ is \textit{strongly finite} (in the terminology of \cite{creutzig_gannon}) if $V$ is $\NN$-graded by conformal weights, simple, self-contragredient, and $C_2$-cofinite. The technical $C_2$-cofiniteness condition implies in particular that the category $\Rep(V)$ of grading-restricted generalized $V$-modules is a finite abelian category and a braided monoidal category \cite{huangC2}, and thus a braided finite tensor category if $\Rep(V)$ is rigid. A strongly finite VOA $V$ is \textit{strongly rational} if $\Rep(V)$ is semisimple. The main theorem of the theory of rational VOAs is that if $V$ is strongly rational, then $\Rep(V)$ is a modular tensor category \cite{huang2005vertex, Huang:Rig_Mod}. A long-standing open problem in the VOA literature is whether a VOA that contains a strongly rational VOA as a vertex operator subalgebra is strongly rational. In Theorem \ref{thm:VOA_ext_rational}, we solve this problem in the maximum possible generality; see also Theorem \ref{thm:VOSA-ext-rational} for a vertex operator superalgebra generalization:
\begin{theorem}\label{thm:intro-VOA-ext-rational}
If $V$ is a strongly rational VOA and $V\subseteq A$ is a VOA extension such that $A$ is simple and $\NN$-graded, then $A$ is strongly rational. Moreover, the category $\Rep(V)_A$ of non-local $A$-modules in $\Rep(V)$ is a fusion category.
\end{theorem}

One of the first special cases of this result was proved in \cite[Theorem 2.12]{DGH}, where $V$ was taken to be a tensor power of rational Virasoro VOAs at central charge $\frac{1}{2}$.
A more general version appears in \cite[Theorem 3.5]{Huang:2014ixa}, but there it was assumed that all $V$-modules have only non-negative real conformal weights and that $V$ itself is the only simple $V$-module with conformal weight $0$ (with conformal weight $0$ space spanned by the vacuum vector). By  \cite{Huang:Rig_Mod, DJX}, these extra conditions guarantee that $\Rep(V)$ is pseudo-unitary, so that non-zero objects of $\Rep(V)$ have positive real categorical dimensions, and therefore by Kirillov and Ostrik's semisimplicity results for rigid $\Rep(V)$-algebras \cite[Theorem 3.3]{kirillov2002q}, $\Rep(A)$ is semisimple and $A$ is strongly rational. One could just as well replace the pseudo-unitary assumptions of \cite{Huang:2014ixa} with the direct assumption that the dimension of $A$ in $\Rep(V)$ is non-zero, as in \cite[Theorem 1.2(1)]{McRaeSS}. But in practice it is hard to prove that dimensions are non-zero without (pseudo-)unitarity, and it seems that the representation categories of most strongly rational VOAs are not unitary or even pseudo-unitary. This is especially true for the strongly rational exceptional $W$-algebras discovered recently in \cite{Ar15b, AvE, McRae2021}, among other references. 

But now, by replacing rigid algebras with exact algebras, we can prove Theorem \ref{thm:intro-VOA-ext-rational} with no non-zero dimension assumption.
As an application, this theorem has just now been used to establish strong rationality of the coset $\text{Com}(L_{k+1}(\mathfrak{sp}_{2n}),L_{k}(\mathfrak{sp}_{2n}) \otimes F(4n))$ at any admissible level $k$ of $\mathfrak{sp}_{2n}$ \cite[Theorem 9.1]{CKL24}. Here $F(4n)$ is the vertex operator superalgebra of $4n$ free fermions and $L_{k}(\mathfrak{sp}_{2n})$ is the simple affine VOA of $\mathfrak{sp}_{2n}$ at level $k$.  A few more examples are given in Section \ref{sec:ex}.

If $V$ is strongly finite but not strongly rational, then it is conjectured that $\Rep(V)$ is a non-semisimple modular tensor category. So far, however, this is known only if $\Rep(V)$ is rigid \cite{McRae2021}, and rigidity of $\Rep(V)$ in general remains an open problem. In Theorem \ref{thm:exts_of_C2_VOAs}, we show that rigidity is at least preserved under taking extensions, as long as Etingof and Ostrik's exactness criteria from \cite{etingof2021frobenius} hold:
\begin{theorem}\label{thm:intro-simple-VOA}
      Let $V$ be a strongly finite VOA such that $\Rep(V)$ is rigid, and let $V\subseteq A$ be a VOA extension such that $A$ is simple and either of the following equivalent conditions holds:
    \begin{enumerate}
        \item[(a)] There is an $A$-module embedding $A'\hookrightarrow A\boxtimes_V W$ for some $W\in\Rep(V)$,
        \item[(b)] $A'\boxtimes_A A'\neq 0$.
    \end{enumerate}
    Then the category $\Rep(A)$ of grading-restricted generalized $A$-modules is a non-degenerate braided finite tensor category and the category $\Rep(V)_A$ of non-local $A$-modules in $\Rep(V)$ is a finite tensor category. These conclusions hold in particular if $A$ is self-contragredient. Moreover, if $A$ is self-contragredient and $\ZZ$-graded, then $\Rep(A)$ is a (possibly non-semisimple) modular tensor category.
\end{theorem}

We emphasize that our rigidity and semisimplicity results for VOA extensions do not require strong finiteness and rationality. Here we summarize some of the results in Section \ref{subsec:rigidity-for-VOAs} that are immediate corollaries of our general categorical results:
\begin{theorem}\label{thm:intro-gen-VOA-ext}
    Let $V$ be a simple self-contragredient $\ZZ$-graded VOA, $\cC$ a braided finite tensor category of $V$-modules that is closed under contragredients, and $A$ a simple $\ZZ$- or $\frac{1}{2}\ZZ$-graded VOA which contains $V$ as a vertex operator subalgebra such that $A$ is an object of $\cC$.
    \begin{enumerate}
        \item If any of the following conditions hold, then $\cC_A$ and $\cC_A^{\loc}$ are finite tensor categories:
    \begin{enumerate}
        \item There is an $A$-module embedding $A'\hookrightarrow A\boxtimes_V W$ for some $W\in\cC$ (Corollary \ref{thm:com-exact-2-VOA}).
        \item $A'\boxtimes_A A' \neq 0$ (Corollary \ref{thm:com-exact-2-VOA}).
        \item $\cC$ is an integral finite tensor category (Corollary \ref{cor:com-exact-integral-VOA}).
        \item $A\cong A'$ as an $A$-module (Corollary \ref{cor:com-exact-5-VOA}).
        \item $\Hom_\cC(A,V)\neq 0$ (Corollary \ref{cor:com-exact-5-VOA}).
    \end{enumerate}
    \item If $\cC$ is a braided fusion category, then $\cC_A^{\loc}$ and $\cC_A$ are fusion categories (Corollary \ref{thm:comm-exact-semisimple-VOA}).
    \end{enumerate}
\end{theorem}

In the setting of this theorem, $\cC_A^{\loc}$ is precisely the category of (untwisted) $A$-modules which are objects of $\cC$ when considered as $V$-modules \cite[Theorem 3.4]{Huang:2014ixa}, and by \cite[Theorem 3.65]{creutzig2017tensor}, the braided tensor category structure on $\cC_A^{\loc}$ is precisely the vertex algebraic one described in \cite{HLZ8}. For non-strongly finite examples of this theorem, one could consider simple affine VOAs $L_k(\mathfrak{g})$ associated to simple Lie algebras $\mathfrak{g}$ at admissible levels $k$. The category $\cC_k^L(\mathfrak{g})$ of ordinary $L_k(\mathfrak{g})$-modules is a finite semisimple braided monoidal category \cite{Arakawa:2012xrk, Creutzig:2017gpa}, and $\cC_k^L(\mathfrak{g})$ is rigid at least in many cases \cite{Creutzig:2017gpa, Creutzig:2019qje, CVL, Creutzig:2022riy}. If $k\in\NN$, then $L_k(\mathfrak{g})$ is strongly rational and unitary, and hence any simple extension of $L_k(\mathfrak{g})$ is also strongly rational; see \cite{gannon} for recent progress on classifying such extensions. One would also hope for interesting simple extensions $A$ of $L_k(\mathfrak{g})$ for non-integral admissible $k$, and Theorem \ref{thm:intro-gen-VOA-ext}(2) now implies that for any such extension, the category $\cC_k^L(\mathfrak{g})_A^{\loc}$ of ordinary $A$-modules will be a braided fusion category, as long as $\cC_k^L(\mathfrak{g})$ is rigid.

One can also use the above results to study extensions $V\subseteq A$ where $V$ is a $\ZZ$-graded VOA and $A$ is a $\ZZ$- or $\frac{1}{2}\ZZ$-graded vertex operator superalgebra containing $V$ in its even part. See Section \ref{subsec:superalgebras} for details.

As discussed above, Theorems \ref{thm:intro-VOA-ext-rational}, \ref{thm:intro-simple-VOA}, and \ref{thm:intro-gen-VOA-ext} rely on results of Etingof and Ostrik from \cite[\S B1]{etingof2021frobenius}. After the first version of the present paper was completed, Coulembier, Stroi\'nski, and Zorman \cite{coulembier2025simple} proved  \cite[Conjecture B.6]{etingof2021frobenius}, and as a result, certain assumptions in our results have now become unnecessary. Specifically, the assumption that either (a) or (b) holds can be removed from Theorem~\ref{thm:intro-simple-VOA}, and in Theorem~\ref{thm:intro-gen-VOA-ext}(1), $\mathcal{C}_A$ and $\mathcal{C}_A^{\mathrm{loc}}$ are finite tensor categories without assuming conditions (a)--(e). Nevertheless, we have retained the original formulations of our results in the present version of this paper to ensure consistency with works that cite the previous version.


\subsection{Rigidity of \texorpdfstring{$\cC$}{C} and \texorpdfstring{$\cC_A$}{CA} from \texorpdfstring{$\cC_A^\loc$}{CA-loc}}

Now assume that $\cC$ is an abelian braided monoidal category and $A$ is a commutative algebra in $\cC$. In Section \ref{sec:rig-of-C},
we show that if the local $A$-module category $\cC_A^{\loc}$ is rigid, then so are $\cC$ and $\cC_A$ under reasonable assumptions.
It may seem at first unlikely that $\cC_A^{\loc}$ could contain enough information about $\cC$ to prove its rigidity in interesting examples, since for example the monoidal induction functor (given on objects of $\cC$ by $X\mapsto A\otimes X$) relates $\cC$ to $\cC_A$, not $\cC_A^{\loc}$ in general. 
To get around this problem, one of our main assumptions is that every simple object of $\cC_A$ is an object of $\cC_A^{\loc}$. This assumption may seem surprising to readers who are mainly used to algebras in non-symmetric braided fusion categories. However, we have in mind free field-like VOA extensions $V\subseteq A$ in which $A$ is not semisimple but rather indecomposable as a $V$-module. For such extensions, it seems normal that every simple object of $\mathcal{C}_A$ is local: see for example \cite[Lemma 8.5]{CLR} for the case that $V$ is a singlet VOA and $A$ is a rank $1$ Heisenberg VOA, and see \cite[Section 4]{CMY24} for the case that $V$ is a simple affine VOA of $\mathfrak{sl}_2$ at admissible level and $A$ is Adamovi\'{c}'s inverse quantum Hamiltonian reduction \cite{Ad-IQHR}.

Free-field like realizations of VOAs for which all simple objects of $\mathcal{C}_A$ are local have several uses besides the rigidity applications that we will discuss further below.
They are crucially used in \cite{CLR} to prove equivalences between non-semisimple VOA and quantum group module categories. In these logarithmic Kazhdan-Lusztig correspondences, the property that all simple objects of $\cC_A$ are local is used for example to show the existence of a tensor functor $\mathcal{C}_A\rightarrow\cC_A^{\loc}$ that splits the inclusion $\cC_A^\loc\hookrightarrow\cC_A$; see \cite[Lemma 4.4]{CLR}. 
Moreover, a Verlinde formula for the fusion rules of VOAs in logarithmic conformal field theory has just now been proved under a very similar setup; see \cite[Assumption 2.3]{C24}.

To describe our rigidity results here specifically, we need to recall Grothendieck-Verdier duality \cite{boyarchenko2013duality}. If $\cC$ is a monoidal category, then an object $K\in\cC$ is \textit{dualizing} if for every $X\in\cC$ the functor $Y\mapsto\Hom_\cC(Y\otimes X,K)$ is representable, that is, there is an object $DX\in\cC$ and a natural isomorphism
\begin{equation*}
    \Hom_{\cC}(Y\otimes X,K)\cong\Hom_\cC(Y,DX)
\end{equation*}
for all objects $X,Y\in\cC$, such that the resulting contravariant functor $X\mapsto DX$ is an anti-equivalence. A \textit{Grothendieck-Verdier category} is a monoidal category equipped with a dualizing object. If a braided monoidal category of modules for a VOA $V$ is closed under contragredient modules, then by \cite{ALSW} it is a Grothendieck-Verdier category whose dualizing object is the contragredient $V'$ of $V$.

A Grothendieck-Verdier category with dualizing object $K\cong\vac$ is called an r-category in \cite{boyarchenko2013duality}. In an r-category $\cC$, the identity morphism $\id_{DX}$ for any object $X\in\cC$ induces a map $e_X: DX\otimes X\rightarrow\vac$ analogous to the evaluation morphism of a left dual, but there might not be a coevaluation $\vac\rightarrow X\otimes DX$. Conversely, if $\cC$ is rigid, then $\cC$ has an r-category structure such that $D$ is given by left duals and its quasi-inverse $D^{-1}$ by right duals. Thus Grothendieck-Verdier duality is weaker than rigidity.

In this paper, we mainly use Grothendieck-Verdier categories as a tool for proving rigidity, but they are also interesting in their own right since, for example, many braided monoidal categories from non-simple VOAs are Grothendieck-Verdier categories but not rigid. Examples include module categories for the $\cW(p,q)$-triplet VOAs \cite{GRW, Wood, Nakano}, the universal Virasoro VOAs at central charge $1-\frac{6(p-q)^2}{pq}$ for relatively prime $p,q\in\ZZ_{\geq 2}$ \cite{McR-Sop}, and the universal affine VOAs of $\mathfrak{sl}_2$ at admissible levels \cite{McRae:2023ado}. Thus we expect the following result (Theorem \ref{thm:GVinRepA} below) extending Grothendieck-Verdier duality structure from $\cC$ to $\cC_A$ will be of independent interest:
\begin{theorem}\label{thm:intro-CA-GV}
Let $(\cC,K)$ be an abelian braided Grothendieck-Verdier category with bilinear tensor product bifunctor and dualizing functor $D$, and let $A$ be a commutative algebra in $\cC$. Then the category $\cC_A$ of $A$-modules in $\cC$ is an abelian Grothendieck-Verdier category whose dualizing object is a natural $A$-module structure on $DA$.
\end{theorem}

Note that $DA$ might not be isomorphic to $A$ as an $A$-module, even if $\cC$ is an r-category with $K=\vac$. It is not even clear whether $DA$ is a local $A$-module, and thus it is not clear whether $\cC_A^{\loc}$ is a Grothendieck-Verdier category in general. However, we can prove that $DA$ is local if $\cC$ is a ribbon Grothendieck-Verdier category (that is, $\cC$ has a ribbon twist $\theta:\id_\cC\rightarrow\id_\cC$ satisfying the usual balancing equation and such that $\theta_{DX}=D\theta_X$ for all objects $X\in\cC$) such that $\theta_A^2=\id_A$. The following is Theorem \ref{thm:CAloc_ribbon_GV}:
\begin{theorem}\label{thm:intro-CAloc-GV}
    Let $(\cC,K)$ be an abelian ribbon Grothendieck-Verdier category with twist $\theta$ and dualizing functor $D$, and let $A$ be a commutative algebra in $\cC$. If $\theta_A^2=\id_{A}$, then the category $\cC_A^{\loc}$ of local $A$-modules in $\cC$ is a braided Grothendieck-Verdier category with dualizing object $DA$. Moreover, $\cC_A^\loc$ is a ribbon Grothendieck-Verdier category if $\theta_A=\id_A$.
\end{theorem}

In the setting of the previous theorem, if $\cC$ is a category of modules for a VOA $V$ and $V\subseteq A$ is a VOA extension in $\cC$, then $\cC_A^{\loc}$ has two Grothendieck-Verdier category structures: the categorical one of Theorem \ref{thm:intro-CAloc-GV} and the vertex algebraic one given by $A$-module contragredients. In Theorem \ref{thm:VOA_and_tens_cat_contras}, we prove that these two Grothendieck-Verdier category structures are the same. We also describe the Grothendieck-Verdier category structure on $\cC_A$ of Theorem \ref{thm:intro-CA-GV} in vertex algebraic terms. To our knowledge, explicit non-local $A$-module structure on the $V$-module contragredient of a non-local $A$-module has been obtained previously only in the case that $V$ is the fixed-point subalgebra of some automorphism of $A$ \cite{Xu, Huang:Tw-Intw-Ops}.

With the above preparation, we can now state our main result (Theorem \ref{rigidity-of-C-from-CAloc} below) on rigidity of $\cC$ from rigidity of $\cC_A^{\loc}$. Besides assuming that $\cC$ is a ribbon r-category and that every simple object of $\cC_A$ is local, we need to add a mild non-degeneracy assumption on $\cC$. We also need the induction functor $F_A: \cC\rightarrow\cC_A$, given on objects by $X\mapsto A\otimes X$, to commute with duality:
\begin{theorem}\label{thm:intro-C-rigid}
Let $(\cC,D,\theta)$ be a locally finite abelian ribbon $r$-category, and let $A$ be a commutative algebra in $\cC$ such that $\theta_A^2=\id_A$ and the unit morphism $\iota_A:\vac\rightarrow A$ is injective. Also assume the following:
\begin{enumerate}
\item[(a)] $\cC_A^{\loc}$ is rigid and every simple object of $\cC_A$ is an object of $\cC_A^{\loc}$.

\item[(b)] For any simple object $X\in\cC$, the map $e_X: DX\otimes X\rightarrow\vac$ is surjective and any non-zero $\cC_A$-morphism from $F_A(DX)$ to $F_A(X)^*$ is an isomorphism.

\item[(c)] For any $\cC$-subobject $s: S\hookrightarrow A$ such that $S$ is not contained in $\mathrm{Im}\,\iota_A$, there exists an object $Z\in\cC$ such that $c_{Z,S}\neq c_{S,Z}^{-1}$ and the map $s\otimes\id_Z: S\otimes Z\rightarrow A\otimes Z$ is injective.

\end{enumerate}
Then $\cC$ is rigid.
\end{theorem}
The proof of this theorem goes as follows:
\begin{itemize}
    \item First, in the setting of the theorem, $\cC_A^{\loc}$ has two Grothendieck-Verdier duality structures: one given by rigid duals and one from Theorem \ref{thm:intro-CAloc-GV}. By \cite[Proposition 2.3]{boyarchenko2013duality}, the two dualizing objects $A$ and $DA$ must differ by tensoring with an invertible object of $\cC_A^{\loc}$, and thus by \cite[Proposition 2.3]{boyarchenko2013duality} again $A$ is a dualizing object of the larger category $\cC_A$ if and only if $DA$ is. Since $DA$ is a dualizing object of $\cC_A$ by Theorem \ref{thm:intro-CA-GV}, so is $A$. So $\cC_A$ has an r-category structure.
    
    \item Next, we show that if every simple object of a locally finite abelian r-category is rigid, then so is every object by induction on length (Theorem \ref{thm:rigidityfrom simples}). Thus using condition (a) in Theorem \ref{thm:intro-C-rigid}, we conclude $\cC_A$ is rigid (Theorem \ref{thm:GVto rigid}).

    \item Now the Drinfeld center $\mathcal Z(\cC_A)$ is rigid, and by \cite[Corollary 4.5]{schauenburg2001monoidal}, this is equivalent to the local module category $\cZ(\cC)_{(A,\sigma)}^\loc$, where the commutative algebra $(A,\sigma)\in\cZ(\cC)$ is equipped with the half-braiding $\sigma=c_{A,-}^{-1}$. Using induction, we get a braided monoidal functor $I: \cC\rightarrow\cZ(\cC)_{(A,\sigma)}^\loc$.

    \item Next, condition (b) in Theorem \ref{thm:intro-C-rigid}, Theorem \ref{thm:FcommuteswithD}, and Lemma \ref{lem:central-functor-duality} ensure that $I$ commutes with duality, that is, $I(DX)\cong I(X)^*$ for all objects $X\in\cC$. Using this together with the mild non-degeneracy condition (c) and the assumption that $\iota_A:\vac\rightarrow A$ is injective, we show that $I$ is fully faithful (Theorem \ref{thm:faithfulness-of-I}).

    \item Finally, rigidity of $\cC$ follows from Lemma \ref{lem:GVto rigid}, since $I$ is an embedding of $\cC$ into a rigid category and $I$ commutes with duality.
\end{itemize}

Theorem \ref{thm:intro-C-rigid} is tailored for applications to vertex operator algebras such as affine VOAs, affine $W$-algebras, and related VOAs like orbifolds, cosets, and extensions. Affine VOAs and $W$-algebras are non-rational at almost all levels, and when they are non-rational, their categories of weight modules are usually neither finite nor semisimple. Even the existence of braided monoidal structure on suitable module categories for these non-rational VOAs is difficult to establish. 
Once one has a braided monoidal category, rigidity is proved either by relating the category to a known rigid category or by studying analytic properties of explicit correlation functions. 
But these methods can be cumbersome, and except in special situations, they might not generalize much beyond VOAs related to $\mathfrak{sl}_n$.

 In previously studied examples, however, we have found that many affine VOAs and $W$-algebras $V$ admit conformal embeddings $V \hookrightarrow A$ where the (local) representation theory of $A$ is often known, semisimple, and rigid. In these examples, $A$ is often an indecomposable $V$-module, and it seems that this indecomposability is tightly connected to condition (a) in Theorem \ref{thm:intro-C-rigid}.
 For example, \cite[Lemma 8.5]{CLR} proves that if $\cC$ is a suitable module category for the singlet VOA $\cM(p)$, $p\in\ZZ_{\geq 2}$, and $A$ is the rank one Heisenberg VOA, which is an indecomposable $\cM(p)$-module, then every simple object of $\cC_A$ is local. In Example \ref{ex:Wp}, as a proof of concept, we verify the conditions of Theorem \ref{thm:intro-C-rigid} and thus prove rigidity for the strongly finite triplet algebra $\cW(2)$, which embeds into the rational lattice VOA $V_{2\ZZ}$. It was already proved in \cite{Tsuchiya:2012ru} that $\Rep(\cW(2))$ is rigid using analytic properties of explicit solutions of BPZ differential equations; here, we only need a few fusion rules in $\Rep(\cW(2))$ to prove rigidity. 
 
Much more importantly, since the first version of this paper was completed, the follow-up work \cite{CMY24} has verified the assumptions of Theorem \ref{thm:intro-C-rigid} and has thus proved rigidity for the category of weight modules for the simple affine VOA $L_k(\mathfrak{sl}_2)$ at any admissible level $k$. 
The embedding $L_k(\mathfrak{sl}_2)\hookrightarrow A$ is Adamovi\'{c}'s inverse quantum Hamiltonian reduction \cite{Ad-IQHR}, in which $A$ is a rational Virasoro VOA tensored with a half-lattice VOA.
Affine VOAs of $\mathfrak{sl}_2$ at admissible levels have long been regarded as prototypical examples of VOAs for logarithmic conformal field theory, and thus a complete understanding of their weight module categories is a key step towards a comprehensive understanding of non-semisimple vertex algebraic braided monoidal categories. 
Besides rigidity, the remaining open problems for these categories are fusion rules and the Verlinde formula, which are now treated with related ideas in \cite{C24}. The fusion rules and rigidity have also just appeared independently and simultaneously in \cite{Flor24}, via a different approach.

Another prospective application of Theorem \ref{thm:intro-C-rigid} is to prove rigidity and Kazhdan-Lusztig correspondences for module categories of affine Lie superalgebras. For a simple finite-dimensional Lie (super)algebra $\mathfrak{g}$ and a level $k\in\CC$, let $KL^k(\mathfrak{g})$ be the Kazhdan-Lusztig category of finite-length modules for the universal affine vertex operator (super)algebra of $\mathfrak{g}$ at level $k$ that have finite-dimensional conformal weight spaces. For $\mathfrak{g}$ a Lie algebra, and for almost all levels $k$,
 Kazhdan and Lusztig showed that $KL^k(\mathfrak{g})$ is a rigid braided monoidal category which is equivalent to a corresponding category of modules for the quantum group of $\mathfrak{g}$  \cite{kazhdan1993tensor, kazhdan1993tensor2, kazhdan1994tensor, kazhdan1994tensor2}. But for Lie superalgebras, similar results have been achieved so far only for $\mathfrak{gl}_{1\vert 1}$ and a few related solvable Lie superalgebras \cite{Creutzig:2020zom, CLR, Cre-Niu}.

  It might now be within reach to prove that $KL^k(\mathfrak{g})$ is rigid and equivalent to a quantum supergroup category when $\mathfrak{g}$ is a basic classical Lie superalgebra of type I. In this case, the universal affine VOA of $\mathfrak g$ at level $k$ conformally embeds into the affine VOA of the (reductive) even Lie subalgebra $\mathfrak g_0\subseteq\mathfrak g$ (at some levels that are related to $k$) tensored with a fermionic $bc$-system whose rank is the number of odd positive roots of $\mathfrak g$. It should then be possible to use Theorem \ref{thm:intro-C-rigid} to show that $KL^k(\mathfrak{g})$ is rigid for almost all $k$.
  It may then also be possible to generalize the ideas of \cite{CLR} to prove a Kazhdan-Lusztig correspondence for type I Lie superalgebras. As mentioned before, a key step would be to show that a certain functor from $KL^k(\mathfrak{g})$ to a relative Drinfeld center is fully faithful, and by our Theorem \ref{thm:faithfulness-of-I} below, this now requires fewer conditions.

\subsection{State of the art for non-rational VOAs}

Over the last decade there has been a lot of progress on braided monoidal categories from non-rational VOAs. As these results are scattered and also based on many examples, we conclude this introduction with a short summary of recent results.

First, the existence of braided monoidal structure on a given category of modules for a VOA $V$ is in general a difficult question. For categories that may include logarithmic $V$-modules, on which the Virasoro $L_0$ operator acts non-semisimply, this question was addressed by Huang, Lepowsky, and Zhang in the series of papers \cite{HLZ1, HLZ2, HLZ3, HLZ4, HLZ5,HLZ6,HLZ7,HLZ8}. They showed that if a category $\cC$ of $V$-modules satisfies a list of technical conditions, then $V$ is a ``vertex tensor category,'' which in particular means that $\cC$ admits a braided monoidal structure that is particularly natural from the vertex algebraic perspective.
However, it is often difficult to check whether $\cC$ satisfies Huang, Lepowsky, and Zhang's technical conditions, and therefore it is difficult to determine what is the correct category of $V$-modules for braided monoidal structure.
If $V$ is strongly finite, then there is no problem because Huang showed in \cite{huangC2} that the entire category $\Rep(V)$ of grading-restricted generalized $V$-modules is a vertex tensor category.

Beyond the strongly finite case, it is conjectured that for any VOA $V$, the category $\cC_1(V)$ of $C_1$-cofinite $V$-modules should satisfy the conditions for braided monoidal structure. The category of $C_1$-cofinite modules was essentially first introduced in the physics literature by Nahm \cite{Nahm}, who called such modules ``quasi-rational,'' and the vertex algebraic definition is due to Li \cite{Li-fin}. Key results of Huang \cite{Huang:diff-eqns} and Miyamoto \cite{miyamoto:C1-fus-prod} show that $\cC_1(V)$ is closed under vertex algebraic tensor products and that correlation functions built from compositions of vertex algebraic intertwining operators among $C_1$-cofinite $V$-modules satisfy analytic conditions needed for associativity isomorphisms in $\cC_1(V)$.
Building on these results, a recent series of works has specified simpler conditions that guarantee $\cC_1(V)$ satisfies the remaining conditions of Huang, Lepowsky, and Zhang for braided monoidal structure;
see in particular \cite[Theorem~6.6]{Creutzig:2017gpa}, \cite[Theorem~4.2.5]{Creutzig:2020zvv}, \cite[Theorem~3.6]{CJ}, and also \cite[Theorem~2.3]{Mc2}. The most general version so far is the following:
\begin{theorem}\textup{\cite{CORY24}}
 If the category $\cC_1(V)$ of $C_1$-cofinite grading-restricted generalized modules for a vertex operator superalgebra $V$ is closed under contragredients, then it admits the braided monoidal category structure specified in \cite{HLZ8}. In particular, $\cC_1(V)$ is a braided monoidal category if the following two conditions hold:
\begin{enumerate}
    \item[(a)] The contragredient $W'$ of any simple $C_1$-cofinite $V$-module $W$ is $C_1$-cofinite.
    \item[(b)] $\cC_1(V)$ equals the category of finite-length $C_1$-cofinite grading-restricted generalized $V$-modules with $C_1$-cofinite composition factors.
\end{enumerate}
\end{theorem}

If $V$ is not $C_2$-cofinite, then $\cC_1(V)$ may not be the only category of $V$-modules to consider. For example, if $V=L_k(\mathfrak{g})$ is a simple affine VOA at an admissible level, then $\cC_1(L_k(\mathfrak{g}))$ is finite and semisimple \cite{Arakawa:2012xrk}, but the larger category of weight $L_k(\mathfrak{g})$-modules has much richer structure. So far, braided monoidal structure on non-$C_1$-cofinite weight module categories for affine VOAs and $W$-algebras is known in only a few examples, including the $\beta\gamma$-ghost VOA \cite{Allen:2020kkt}, a two-parameter family of extensions of Heisenberg and singlet VOAs \cite{Creutzig:2022ugv}, and $L_k(\mathfrak{sl}_2)$ at admissible levels \cite{Creutzig:2023rlw}. Developing methods for obtaining more braided monoidal categories for VOAs beyond $\cC_1(V)$ is a major problem for future research.

We now present a (probably not completely comprehensive) list of VOA representation categories beyond rational VOAs that are currently known to admit braided monoidal category structure.
For a finite-dimensional Lie (super)algebra $\g$, we let $L_k(\g)$, respectively $V^k(\g)$, denote the simple, respectively universal, affine VOA of $\g$ at level $k \in \mathbb C$.
In the following table, $\yes$ means that the listed property holds, $\open$ means that it is not yet proven, and $\yes/\open$ means that for some specific examples it holds while for other examples it remains open. Similarly, $\yes / \no$ means that the listed property holds in some examples but does not hold in others, and $\no$ means that the property does not hold. The precise details can be found in the references. 
The categories in the table are the categories of generalized modules with some additional properties:
\begin{enumerate}
    \item {\bf $C_1$-cofinite:} This is the category of $C_1$-cofinite grading-restricted generalized modules.
    
    \item {\bf Weight:} If the conformal weight $1$ subspace $V_{(1)}$ of $V$ is non-zero, this is the category of finitely-generated generalized modules on which a Cartan subalgebra $\mathfrak{h}\subseteq V_1$ acts semisimply, such that the intersection of each $\mathfrak{h}$-weight space with each conformal weight space is finite dimensional, and such that the conformal weights of each $\mathfrak{h}$-weight space are bounded below. If $V$ is a Virasoro or singlet VOA, then $V_1=0$, but the category of $C_1$-cofinite $V$-modules has a subcategory with properties similar to the category of weight modules for an affine vertex (super)algebra or $W$-algebra, so here we call this the subcategory of ``weight'' $V$-modules.
    
    \item {\bf Generalized weight:} If the weight $1$ subspace $V_1$ of $V$ is non-zero, this is the same as the category of weight $V$-modules, except that we allow the Cartan subalgebra $\mathfrak{h}\subseteq V_1$ to act by generalized eigenvalues (with finite Jordan blocks). 
\end{enumerate}
In the ``Projectives'' column, we indicate whether the category in question has enough projectives. For Virasoro VOAs, we parameterize the central charge as
$c(t) = 13 - 6(t+t^{-1})$ for $t\in\mathbb{C}^\times$. Finally, we organize the table thematically, roughly listing first affine vertex (super)algebras and then $W$-algebras and their extensions, rather than listing chronologically according to the order in which the indicated properties were proved:
\begin{center}
{\small
\begin{tabular}{ |c|c|c|c|c|c|c| }
\hline
Vertex operator algebra & Category & Rigid & Finite & Semisimple  & Projectives & References \\[1mm]
\hline 
Heisenberg VOAs & $C_1$-cofinite & \yes & \no & \no & \no & Exercise \\[1mm]
 & weight & \yes & \no & \yes & \yes & \cite{Creutzig:2016ehb} \\[1mm]
$L_k(\g)$, $k$ admissible & $C_1$-cofinite & \yes / \open & \yes & \yes & \yes& \cite{Arakawa:2012xrk, Creutzig:2017gpa, Creutzig:2019qje} \\[1mm]
$L_k(\mathfrak{sl}_2)$, $k$ admissible & weight & \yes  & \no & \no & \yes & \cite{Arakawa:2023msa,Creutzig:2023rlw, CMY24} \\[1mm]
$V^k(\mathfrak{sl}_2)$, $k$ admissible & $C_1$-cofinite & \no  & \no & \no & \yes & \cite{McRae:2023ado} \\[1mm]
$L_k(\mathfrak{gl}_{1|1})$, $k\neq 0$ & $C_1$-cofinite  & \yes & \no & \no & \no & \cite{Creutzig:2020zom} \\[1mm]
 & weight & \yes & \no & \no & \yes & \cite{Creutzig:2020zom} \\[1mm]
$L_k(\mathfrak{osp}_{1|2})$, $k$ admissible & weight  & \yes  & \no & \no & \yes & \cite{CR24} \\[1mm]
$L_k(\mathfrak{osp}_{1|2n})$, $k$ admissible & $C_1$-cofinite  & \yes & \yes & \yes & \yes & \cite{Creutzig:2022riy} \\[1mm]
universal Virasoro, $t \not\in \mathbb Q$ & $C_1$-cofinite  & \yes  & \no & \yes  & \yes  & \cite{Creutzig:2020zvv} \\[1mm]
 $t \in \{ \pm 1\} $ & $C_1$-cofinite  & \yes & \no &  \no  & \no  & \cite{MY-c25} \\[1mm]
  & ``weight''  & \yes & \no &  \yes  & \yes  & \cite{MY-c25} \\[1mm]
 $t \in \mathbb Z_{>1}$ & $C_1$-cofinite  & \yes & \no &  \no  & \no & \cite{MY-cp1} \\[1mm]
 & ``weight''  & \yes & \no &  \no  & \yes  & \cite{MY-cp1} \\[1mm]
 $t \in (\mathbb Q\setminus\mathbb{Z})_{>0}$ 
& $C_1$-cofinite &  \no & \no &  \no & \no &  \cite{McR-Sop}\\[1mm]
singlet VOAs $\mathcal{M}(p)$ & $C_1$-cofinite  & \yes & \no  &  \no & \no & \cite{Creutzig:2022lep} \\[1mm]
 & ``weight''  & \yes & \no  &  \no & \yes & \cite{Creutzig:2020qvs,Creutzig:2022lep} \\[1mm]
 triplet VOAs $\mathcal{W}(p)$ & $C_1$-cofinite & \yes & \yes & \no & \yes & \cite{Adamovic:2007er,Tsuchiya:2012ru} \\[1mm]
 symplectic fermions & $C_1$-cofinite & \yes & \yes & \no & \yes & \cite{Abe, mcrae:deligne}\\ [1mm]
$\beta\gamma$-ghost (symplectic bos.) & weight  & \yes & \no & \no & \yes & \cite{Allen:2020kkt} \\[1mm]
 & gen. weight  & \yes & \no & \no & \no & \cite{Ballin:2022rto} \\[1mm]
$\mathcal B_p$ and $\mathcal S_p$ VOAs & weight  & \yes & \no & \no & \yes & \cite{Creutzig:2022ugv} \\[1mm]
$N=1$ super Virasoro & $C_1$-cofinite  & \yes / \open & \no & \yes / \no  & \yes / \open & \cite{CORY24} \\[1mm]
$N=2$  minimal models & $C_1$-cofinite  & \yes  & \no  & \no & \yes & \cite{Creutzig:2023rlw, CMY24} \\[1mm]
\hline
\end{tabular}
}
\end{center}

\smallskip

We emphasize again that proving rigidity in these examples often involved detailed studies of analytic properties of correlation functions, which could be quite cumbersome. However, the rigidity statements for weight modules of $L_k(\mathfrak{sl}_2)$ at admissible levels and for $C_1$-cofinite modules of the $N=2$  minimal models will be proven more easily in \cite{CMY24} using our Theorem \ref{thm:intro-C-rigid}. The rigidity of weight modules for $L_k(\mathfrak{osp}_{1|2n})$ at admissible levels then follows as $L_k(\mathfrak{osp}_{1|2n})$ is an extension of $L_k(\mathfrak{sl}_2)$ tensored with a rational Virasoro VOA; this will appear in \cite{CR24}.

\medskip

\noindent\textbf{Acknowledgments.}
RM is partially supported by a startup grant from Tsinghua University and by a research fellowship from the Alexander von Humboldt Foundation. RM also thanks Universit\"{a}t Hamburg for its hospitality during the visit in which this work was finished.
KS is supported by JSPS KAKENHI Grant Number 24K06676.
HY is partially supported by a start-up grant from the University of Alberta and an NSERC Discovery Grant. Part of the work was finished while HY was in residence at the Mathematical Sciences Research Institute in Berkeley, California, during the Quantum Symmetries Reunion in 2024; this was supported by NSF grant DMS-1928930.


\section{Rigidity of \texorpdfstring{$\cC_A$}{CA} and \texorpdfstring{$\cC_A^{\loc}$}{CAloc} from rigidity of \texorpdfstring{$\cC$}{C}}\label{sec:rig-of-CA}

In this section, we start with background and useful results on (commutative) algebras in (braided) monoidal categories, and then proceed to the main results on rigidity of module categories for algebras in finite tensor categories.

\subsection{Rigid (braided) monoidal categories}\label{subsec:rig-mon-cats}

Given a category $\cC$, we use $\id_{\cC}$ to denote the identity functor from $\cC$ to $\cC$. Given a functor $F:\cC\rightarrow\cD$, we denote by $F(\cC)$ the \textit{essential image} of $F$, that is, the full subcategory of $\cD$ whose objects are all $Y\in\cD$ such that $Y\cong F(X)$ for some $X\in\cC$. If a functor $F$ admits a right (or left) adjoint, we will denote it by $F^{\radj}$ (or $ F^{\ladj}$).

See for example \cite[\S2]{etingof2016tensor} or \cite{walton2024symmetries} for basic definitions related to monoidal categories. Because of Mac Lane's strictness theorem \cite{mac2013categories}, we will generally suppress associativity and unit isomorphisms in calculations with monoidal categories.
Recall that an object $X$ in a monoidal category $(\cC,\otimes,\one)$ is \textit{rigid} if it has left and right duals, that is, there exist objects $X^*, {}^*X \in \cC$ with (co)evaluation maps
\begin{align*}
\ev_X: X^* \otimes X \rightarrow \one, & \hspace{1cm} \coev_X: \one \rightarrow X \otimes X^*,  \\
\ev'_X: X \otimes  {}^*X \rightarrow \one, & \hspace{1cm} \coev'_X: \one \rightarrow {}^*X \otimes X, 
\end{align*} 
such that 
\begin{align*}
    (\id_X \otimes \ev_X)\circ(\coev_X \otimes \id_X) =\id_X, & \hspace{1cm} (\ev_X \otimes \id_{X^*})\circ(\id_{X^*} \otimes \coev_X) =\id_{X^*},\\
    (\ev'_X \otimes \id_X)\circ(\id_X \otimes \coev'_X) =\id_X, & \hspace{1cm} (\id_{{}^*X} \otimes \ev'_X)\circ(\coev'_X \otimes \id_{{}^*X}) =\id_{{}^*X}.
\end{align*}
We say $\cC$ is {\it rigid} if all of its objects are rigid. If $\cC$ is rigid, then left duals define a contravariant functor $\cC\rightarrow\cC$ that maps $X\mapsto X^*$ for any object $X\in\cC$ and $f\mapsto f^*$ for any morphism $:X\rightarrow Y$ in $\cC$, where
\begin{equation*}
f^* =(\ev_Y\otimes\id_{X^*})\circ(\id_{Y^*}\otimes f\otimes\id_{X^*})\circ(\id_{Y^*}\otimes\coev_X).
\end{equation*}
Right duals define a contravariant functor $\cC\rightarrow\cC$ similarly.

An object $X\in\cC$ is called \textit{invertible} if there is an object $X^{-1}\in\cC$ such that $X\otimes X^{-1}\cong \vac\cong X^{-1}\otimes X$. If $X$ is invertible, then $X$ is rigid with $X^*={}^*X=X^{-1}$. To see why $X^{-1}=X^*$ for example, fix any isomorphisms
\begin{equation*}
    \ev_X: X^{-1}\otimes X\rightarrow\vac,\qquad \ev_X': X\otimes X^{-1}\rightarrow\vac.
\end{equation*}
Then $f:=(\id_X\otimes\ev_X)\circ((\ev_X')^{-1}\otimes\id_X)$ is an automorphism of $X$, so we can define $\coev_X=(f^{-1}\otimes\id_{X^{-1}})\circ(\ev_X')^{-1}$. It is then immediate that $(\id_X\otimes\ev_X)\circ(\coev_X\otimes\id_X)=\id_X$. To show that $g:=(\ev_X\otimes\id_{X^{-1}})\circ(\id_{X^{-1}}\otimes\coev_X)$ is also the identity, it is easy to calculate
\begin{equation*}
    \ev_X\circ(g\otimes\id_X)=\ev_X,
\end{equation*}
and thus $g\otimes \id_X=\id_{X^{-1}\otimes X}$. Then
\begin{equation*}
g =(\id_{X^{-1}}\otimes\ev_X)\circ(\id_{X^{-1}}\otimes \ev_X^{-1})\circ g = (\id_{X^{-1}}\otimes\ev_X)\circ(g\otimes\id_X\otimes\id_{X^{-1}})\circ(\id_{X^{-1}}\otimes\ev_X^{-1}) =\id_{X^{-1}}
    \end{equation*}
    as required. Similarly, $X^{-1}$ is a right dual of $X$.

Now let $\cC$ be a braided monoidal category (see for example \cite[\S8]{etingof2016tensor}), with natural braiding isomorphism $c_{X,Y}: X\otimes Y\rightarrow Y\otimes X$ for objects $X,Y\in\cC$. A \textit{twist} on $\cC$ is a natural isomorphism $\theta:\id_{\cC}\rightarrow\id_{\cC}$ satisfying the \textit{balancing equation}
\[ \theta_{X\otimes Y} = c_{Y,X} \circ c_{X,Y} \circ (\theta_X \otimes \theta_Y) . \]
Taking $X=Y=\vac$ in the balancing equation, one can show that $\theta_\vac=\id_\vac$.
We call a twist on $\cC$ a \textit{ribbon structure} if $\cC$ is rigid and $\theta_{X^*}=(\theta_X)^*$ for all $X\in\cC$. In this case, there is a natural isomorphism $X^*\xrightarrow{\sim} {}^*X$ for all $X\in\cC$ given by the composition
\begin{equation*}
(\id_{{}^*X}\otimes\ev_X)\circ(c_{{}^*X,X^*}^{-1}\otimes\theta_X)\circ\id_{X^*}\otimes\coev_X').
\end{equation*}
Using this, if $\theta$ is a ribbon structure, then automatically $\theta_{{}^*X} ={}^*\theta_X$ for all $X\in\cC$.

Let $\cC,\cD$ be monoidal categories.
A monoidal functor $(F,F_2,F_0):\cC\rightarrow\cD$ is a functor $F$ together with a natural isomorphism $F_2(X,Y):F(X) \otimes F(Y) \rightarrow F(X\otimes Y)$ for objects $X,Y\in\cC$ and an isomorphism $F_0:\one_{\cD} \rightarrow F(\one_{\cC})$ that satisfy obvious coherence conditions. If $\cC,\cD$ are braided, then $F$ is a braided monoidal functor if the diagram
\begin{equation*}
\begin{tikzcd}[column sep=4pc]
F(X) \otimes F(Y) \ar[d, "F_2({X,Y})"] \ar[r, "c_{F(X),F(Y)}"] & F(Y)\otimes F(X) \ar[d, "F_2({Y,X})"] \\
    F(X\otimes Y) \ar[r, "F(c_{X,Y})"] & F(Y\otimes X)
    \end{tikzcd}
\end{equation*}
commutes for all objects $X,Y\in\cC$. Given a monoidal functor $F:\cC\rightarrow\cD$, its essential image $F(\cC)$ is a monoidal subcategory of $\cD$. If $F$ is braided, then $F(\cC)$ is a braided monoidal subcategory.

Let $\cC$ be a monoidal category. Its \textit{Drinfeld center} $\cZ(\cC)$ is a braided monoidal category whose objects are pairs $(X, \gamma)$ where $X$ is an object of $\cC$ and $\gamma$ is a half-braiding, that is, a natural isomorphism $\gamma: -\otimes X\rightarrow X\otimes -$ satisfying a hexagon identity \cite[\S7.13]{etingof2016tensor}. Morphisms in $\cZ(\cC)$ are given by
\begin{equation*}
    \Hom_{\cZ(\cC)}((X,\gamma),(Y,\delta)) = \lbrace f\in\Hom_\cC(X,Y)\mid (f\otimes\id_Z)\circ\gamma_Z = \delta_Z\circ(\id_Z\otimes f)\;\text{for all}\;Z\in\cC\rbrace.
\end{equation*}
The tensor product on $\cZ(\cC)$ is given by
\begin{equation*}
    (X,\gamma)\otimes(Y,\delta) = (X\otimes Y, (\id_X\otimes\delta)\circ(\gamma\otimes\id_Y))
\end{equation*}
and the braiding is given by $c_{(X,\gamma),(Y,\delta)} = \delta_X$. If $\cC$ is rigid, then $\cZ(\cC)$ is rigid as well. Indeed, if $(X,\gamma)\in\cZ(\cC)$, then $X^*$ admits a half-braiding $\overline{\gamma}$ such that the evaluation $\ev_X: X^*\otimes X\rightarrow\vac$ and coevaluation $\coev_X:\vac\rightarrow X\otimes X^*$ are morphisms in $\cZ(\cC)$. Specifically, for an object $Y\in\cC$, $\overline{\gamma}_Y$ is the composition
\begin{align*}
    \overline{\gamma}_Y: Y\otimes X^* & \xrightarrow{\id_{Y\otimes X^*}\otimes\coev'_{Y}} Y\otimes X^*\otimes {}^*Y\otimes Y \xrightarrow{\id_{Y\otimes X^*\otimes {}^*Y}\otimes\coev_X\otimes\id_Y} Y\otimes X^*\otimes {}^*Y\otimes X\otimes X^*\otimes Y\nonumber\\
&    \xrightarrow{\id_{Y\otimes X^*}\otimes\gamma_{{}^*Y}\otimes\id_{X^*\otimes Y}} Y\otimes X^*\otimes X\otimes {}^*Y\otimes X^*\otimes Y \xrightarrow{\id_Y\otimes\ev_X\otimes\id_{{}^*Y\otimes X^*\otimes Y}} Y\otimes {}^*Y\otimes X^*\otimes Y\nonumber\\
& \xrightarrow{\ev'_Y\otimes\id_{X^*\otimes Y}} X^*\otimes Y.
\end{align*}
The inverse can be defined more simply as the composition
\begin{align*}
    \overline{\gamma}^{-1}_Y: X^*\otimes Y \xrightarrow{\id_{X^*\otimes Y}\otimes\coev_X} X^*\otimes Y\otimes X\otimes X^* \xrightarrow{\id_{X^*}\otimes\gamma_Y\otimes\id_{X^*}} X^*\otimes X\otimes Y\otimes X^* \xrightarrow{\ev_X\otimes\id_{Y\otimes X^*}} Y\otimes X^*.
\end{align*}
Similarly, $\gamma^{-1}$ induces a half-braiding for ${}^*X$ such that $\ev'_X$ and $\coev'_X$ become morphisms in $\cZ(\cC)$.

The forgetful functor $U:\cZ(\cC)\rightarrow\cC$ given on objects by $U(X, \gamma)=X$ is monoidal, and if $\cC$ is braided, there is a braided monoidal functor $i_+:\cC\rightarrow\cZ(\cC)$ given on objects by $i_+(X) = (X,c_{-,X})$. There is also a braided monoidal functor $i_-:\overline{\cC}\rightarrow\cZ(\cC)$ such that $i_-(X) = (X,c_{X,-}^{-1})$, where $\overline{\cC}$ is the monoidal category $\cC$ equipped with the reversed braiding $\overline{c}_{X,Y} = c_{Y,X}^{-1}$. A monoidal functor $F:\cC\rightarrow\cD$, where $\cC$ is braided but $\cD$ may not be, is called \textit{central} if there exists a braided functor $F':\cC\rightarrow\cZ(\cD)$ such that $F=U\circ F'$.


\subsection{Algebras and their module categories}\label{subsec:algebras-and-modules}
Let  $(A,\mu_A,\iota_A)$ be an algebra in a monoidal category $\cC$ (see for example \cite[\S1]{kirillov2002q} or \cite[\S7.8]{etingof2016tensor}). If $\cC$ is braided, then $A$ is \textit{commutative} if $\mu_A\circ c_{A,A} = \mu_A$. In this section, we will use $\cC_A^l$ and $\cC_A^r$ to distinguish between the categories of left and right $A$-modules in $\cC$. In Sections \ref{sec:rig-of-C} and \ref{sec:VOAs}, where we only consider commutative algebras and thus $\cC_A^l$ and $\cC_A^r$ are isomorphic, we will drop the superscript notation and use $\cC_A$ to denote the category of left $A$-modules (equivalently, right $A$-modules) in $\cC$. For a left (respectively right) $A$-module $M$, we will use the notation $\mu_M^l$ (respectively $\mu^r_M$) to denote the action of $A$ on $M$.
 A left $A$-module $(M,\mu^l_M)$ for a commutative algebra $A$ is called \textit{local} if $\mu^l_M = \mu^l_M\circ c_{M,A} \circ c_{A,M}$. We use $\cC_A^{\loc}$ to denote the full subcategory of $\cC_A^l$ consisting of local $A$-modules. We also use ${}_A\cC_A$ to denote the category of $A$-bimodules in $\cC$. 

If $\cC$ has coequalizers, then the tensor product of a right $A$-module $M$ and a left $A$-module $N$ is:
\begin{equation}\label{eq:tensor-over-A}
M\otimes_A N = \mathrm{coequalizer} \big(
\begin{tikzcd}[column sep = 70pt]
  M\otimes A\otimes N
  \arrow[r, yshift = .3em, "{\mu^r_M\otimes \id_N}"]
  \arrow[r, yshift = -.3em, "\id_M\otimes \mu^l_N"']
  & M\otimes N
\end{tikzcd} \big).
\end{equation}
We use $\pi_{M,N}:M\otimes N\rightarrow M\otimes_A N$ to denote the projection map, which satisfies $\pi_{M,N}\circ (\mu^r_M\otimes \id_N)=\pi_{M,N}\circ (\id_M\otimes \mu^l_N)$. Assuming that the tensor product on $\cC$ preserves coequalizers, one can show that $\otimes_A$ makes ${}_A\cC_A$ a monoidal category,  

A \textit{central algebra} in $\cC$ is a pair $(A,\sigma)$ where $A$ is an algebra in $\cC$ and $\sigma:- \otimes A \rightarrow A \otimes -$ is a half-braiding such that $(A,\sigma)$ is an algebra in $\cZ(\cC)$. Given a central algebra $(A,\sigma)$, we can form a full subcategory of ${}_A\cC_A$ consisting of the following bimodules:
\begin{equation*}
    \cC_A^{\sigma} = \{ (M,\mu^l_M,\mu^r_M) \in {}_A\cC_A \mid \mu^r_M =\mu^l_M \circ \sigma_M \}.
\end{equation*}
When $(A,\sigma)$ is commutative in $\cZ(\cC)$, we call it commutative central. In this case, $\cC_A^{\sigma}$ is closed under $\otimes_A$ and thus is a monoidal subcategory of ${}_A\cC_A$. 

If $\cC$ is braided and $A$ is a commutative algebra in $\cC$, then $i_+(A)=(A,c_{-,A})$ and $i_-(A)=(A,c_{A,-}^{-1})$ are commutative central algebras. We also have two embeddings  $H_{\pm}:\cC_A^l\rightarrow {}_A\cC_A$ given by
\begin{equation*}
    H_+(M,\mu^l_M) = (M,\mu^l_M, \mu^l_M\circ c_{M,A})\quad\text{and}\quad H_-(M,\mu^l_M) = (M,\mu^l_M, \mu^l_M\circ c^{-1}_{A,M}),
\end{equation*}
and so
\begin{equation}\label{eq:CA-CAsigma}
    \cC_A^{c_{-,A}} = H_+(\cC_A^l)  \quad \text{and} \quad \cC_A^{c_{A,-}^{-1}} = H_-(\cC_A^l). 
\end{equation}    
This means that $\cC_A^l$ is a monoidal category, since it embeds as a monoidal subcategory of ${}_A\cC_A$ (in two different ways). Also, $\cC_A^{\loc} =\cC_A^{c_{-,A}}\cap\cC_A^{c_{A,-}^{-1}}$, so $\cC_A^{\loc}$ is also a monoidal category; additionally, $\cC_A^{\loc}$ is braided \cite{pareigis1995braiding}. One can similarly realize the category of right modules $\cC_A^r$ as a monoidal subcategory of ${}_A\cC_A$.

If $(A,\sigma)$ is a commutative central algebra in a monoidal category $\cC$,
then the \textit{induction functor} $F_{A,\sigma}:\cC \rightarrow\cC_A^{\sigma}$ is defined on objects by 
\begin{equation*}
F_{A,\sigma}(X) = \left( A\otimes X, \mu_{A\otimes X}^l = \mu_A \otimes \id_X, \mu_{A\otimes X}^r = (\mu_A\otimes\id_X)\circ(\id_A \otimes \sigma_X) \right),
\end{equation*}
and on morphisms by $F_{A,\sigma}(f)=\id_A\otimes f$.
One can show that $F_{A,\sigma}$ is a monoidal functor with a right adjoint $ F_{A,\sigma}^{\radj}:\cC_A^{\sigma}\rightarrow \cC$ given on objects by
\[F_{A,\sigma}^{\radj} (M,\mu_M^l,\mu_M^r=\mu_M^l\circ\sigma_M) = M . \]
In the special case that $\cC$ is braided and $A$ is a commutative algebra in $\cC$, we get induction functors $F_{A,c_{-,A}}, F_{A,c^{-1}_{A,-}}:\cC\rightarrow\cC_A^l$.

\subsection{Algebras in rigid monoidal categories}

Let $(A,\mu_A,\iota_A)$ be an algebra in a rigid monoidal category $\cC$. Then the duals of $A$-modules admit natural $A$-module structures:
\begin{lemma}\label{lem:dual-action}
If $(M,\mu^l_M)\in\cC_A^l$, then $(M^*,\mu^r_{M^*})\in\cC_A^r$ where 
$$\mu^r_{M^*}=(\ev_M\otimes \id_{M^*})\circ(\id_{M^*}\otimes \mu^l_M\otimes \id_{M^*})\circ(\id_{M^*}\otimes \id_A\otimes \coev_M).$$  
If $(M,\mu^r_M)\in\cC_A^r$, then $({}^*M,\mu^l_{{}^*M})\in\cC{}_A$ where 
$$\mu^l_{{}^*M} = (\id_{{}^*M}\otimes \ev'_{M})\circ(\id_{{}^*M}\otimes \mu^r_M\otimes \id_{{}^*M})\circ(\coev'_M\otimes \id_A\otimes \id_{{}^*M}).$$   
\end{lemma}

Taking $M=A$ and $\mu_M^l=\mu_M^r=\mu_A$ in this lemma, we get $({}^*A,\mu^l_{{}^*A})\in\cC_A^l$ and $(A^*,\mu^r_{A^*})\in\cC_A^r$ where
\begin{align}\label{eqn:A*-and_*A}
    \mu^l_{{}^*A} &= (\id_{{}^*A}\otimes \ev'_A)\circ(\id_{{}^*A} \otimes \mu_A \otimes \id_{{}^*A})\circ(\coev'_{A}\otimes \id_A\otimes \id_{{}^*A}): A\otimes {{}^*A}\rightarrow {{}^*A} \\
    \mu^r_{A^*} &= (\ev_A\otimes \id_{A^*})\circ(\id_{A^*} \otimes \mu_A \otimes \id_{A^*})\circ(\id_{A^*}\otimes \id_A \otimes\coev_A): A^*\otimes A\rightarrow A^*.\nonumber
\end{align}

\begin{lemma}\textup{\cite[Lemma~2.4.13]{douglas2018dualizable}}\label{lem:CA-dual-antiequivalence}
The functors $(-)^*, \, {}^*(-):\cC\rightarrow\cC$ restrict to anti-equivalences $(-)^*:\cC_A^l\rightarrow \cC_A^r$ and ${}^*(-):\cC_A^r\rightarrow \cC_A^l$. 
\end{lemma}

In fact, the left and right duality functors in the above lemma are quasi-inverses of each other. Indeed, if $M$ is any object of $\cC$, then it is easy to show that
\begin{align*}
    (\id_M\otimes\ev'_{M^*})\circ(\coev_M\otimes\id_{{}^*(M^*)}):  {}^*(M^*)\rightarrow M\\
    (\ev_{{}^*M}\otimes\id_M)\circ(\id_{({}^*M)^*}\otimes\coev'_M):  ({}^*M)^*\rightarrow M
\end{align*}
are isomorphisms in $\cC$. Then straightforward calculations show that the first of these isomorphisms is a morphism in $\cC_A^l$ if $M\in\cC_A^l$ and the second is a morphism in $\cC_A^r$ if $M\in\cC_A^r$. 

The next result will be used later in the proof of Theorem \ref{thm:EO-simple-exact}:
\begin{lemma}\label{lem:right-module}
    If $(M,\mu^l_M)$ is a left $A$-module, then ${}^*(A^*\otimes_A M)$ is a right $A$-module.
\end{lemma}
\begin{proof}
Recall from \eqref{eq:tensor-over-A} that $A^*\otimes_A M$ is defined as a coequalizer:
\begin{equation*}
    \begin{tikzcd}[column sep = 60pt]
  A^*\otimes A\otimes M
  \arrow[r, yshift = .3em, "{\id_{A^*}\otimes \mu^l_M}"]
  \arrow[r, yshift = -.3em, "\mu_{A^*}^r\otimes\id_M"']
  & A^*\otimes M  \arrow[r, "\pi_{A^*,M}"] & A^*\otimes_A M
\end{tikzcd}.
\end{equation*}
Taking right duals and applying the identities ${}^*(A^*\otimes M)\cong {}^*M\otimes A$ and ${}^*(A^*\otimes A\otimes M)\cong {}^*M\otimes {}^*A\otimes A$  (see for example \cite[Exercise 2.10.7(b)]{etingof2016tensor}) leads to a diagram
\begin{equation*}
    \begin{tikzcd}[column sep = 60pt]
    {}^*(A^*\otimes_A M) \arrow[r, "\alpha"] & {}^*M\otimes A \arrow[r, yshift = .3em, "\rho_1\otimes\id_A"] \arrow[r, yshift = -.3em, "\id_{{}^*M}\otimes\rho_2" '] & {}^*M\otimes {}^*A\otimes A
\end{tikzcd},
\end{equation*}
where one calculates $\rho_1: {}^*M \rightarrow {}^*M \otimes {}^*A$ to be the composition
\begin{align*}
\rho_1: {}^*M  \xrightarrow{\coev'_M\otimes\id_{{}^*M}} & \,{}^*M\otimes M\otimes {}^*M
 \xrightarrow{\id_{{}^*M} \otimes \coev'_A \otimes \id_{M\otimes{}^*M}} {}^*M \otimes {}^*A \otimes A \otimes M \otimes {}^*M \\
& \xrightarrow{\id_{{}^*M\otimes {}^*A}\otimes \mu_M^l \otimes \id_{{}^*M}} {}^*M \otimes {}^*A \otimes M \otimes {}^*M  \xrightarrow{\id_{{}^*M\otimes {}^*M}\otimes \ev'_M} {}^*M \otimes {}^*A,
\end{align*} 
and $\rho_2: A\rightarrow {}^*A \otimes A$ is the composition
\[\rho_2: A \xrightarrow{\coev'_A\otimes\id_A} {}^*A\otimes A\otimes A \xrightarrow{\id_{{}^*A}\otimes \mu_A} {}^*A\otimes A.\]
That is, $X={}^*(A^*\otimes_A M)$ is the equalizer of $\rho_1\otimes\id_A$ and $\id_{{}^*M}\otimes\rho_2$, and in particular  $\alpha: X \rightarrow {}^*M \otimes A$ is a monomorphism such that $(\rho_1\otimes\id_A)\circ \alpha = (\id_{{}^*M}\otimes\rho_2)\circ \alpha$.
Moreover, $\alpha$ is dual to the projection map $\pi_{A^*,M}: A^*\otimes M \rightarrow A^*\otimes_A M \cong X^*$ in the sense that
\begin{equation}\label{eq:right-action-1}
    \ev_A \circ (\id_{A^*} \otimes\ev'_M\otimes\id_A) \circ (\id_{A^*\otimes M}\otimes \alpha) 
    = \ev'_{A^*\otimes_A M}\circ (\pi_{A^*,M}\otimes\id_{X}): 
    A^*\otimes M \otimes X \rightarrow \one.
\end{equation}

Now to define a right action of $A$ on $X$, consider the map
\[ \beta: X\otimes A \xrightarrow{\alpha\otimes\id_A} {}^*M \otimes A \otimes A \xrightarrow{\id_{{}^*M} \otimes \mu_A} {}^*M \otimes A. \]
A calculation using the associativity of $\mu_A$ shows that $(\rho_1\otimes\id_A)\circ \beta = (\id_{{}^*M}\otimes\rho_2)\circ \beta$. Thus, by the universal property of equalizers, we get a map $\mu_X^r: X\otimes A \rightarrow X$ such that
\begin{equation}\label{eq:right-action}
    \alpha \circ \mu^r_X = (\id_{{}^*M} \otimes \mu_A)\circ (\alpha\otimes\id_A).
\end{equation} 
Using this identity, the associativity and unit properties of the algebra $A$, and the injectivity of $\alpha$, one can check that $(X,\mu_X^r)$ is a right $A$-module. 
\end{proof}

Now suppose $\cC$ is a rigid braided monoidal category. Then there are two natural isomorphisms $\varphi^{\pm}: (-)^*\rightarrow {}^*(-)$ defined by
\begin{equation*}
    \varphi^{\pm}_X: X^* \xrightarrow{\id_{X^*}\otimes\coev'_X} X^*\otimes {}^*X\otimes X\xrightarrow{c^{\pm 1}\otimes\id_X} {}^*X\otimes X^*\otimes X\xrightarrow{\id_{{}^*X}\otimes\ev_X} {}^*X
\end{equation*}
for $X\in\cC$ (compare with the Drinfeld morphism $u_X: X\rightarrow X^{**}$ discussed in \cite[\S 8.9]{etingof2016tensor}). The inverse natural isomorphisms are given by
\begin{equation*}
    (\varphi^{\pm}_X)^{-1}: {}^*X\xrightarrow{\coev_X\otimes\id_{{}^*X}} X\otimes X^*\otimes {}^*X\xrightarrow{\id_X\otimes c^{\mp1}} X\otimes {}^*X\otimes X^*\xrightarrow{\ev'_X\otimes\id_{X^*}} X^*.
\end{equation*}
If $A$ is a commutative algebra in $\cC$ and $(M,\mu_M^l)\in\cC_A^l$, then we would like $\varphi_M^\pm$ to become isomorphisms of $A$-modules. In the next lemma, we will show that $\varphi_M^-$ in particular is an isomorphism in $\cC_A^l$ given suitable left $A$-module structures on $M^*$ and ${}^*M$, which we now specify. First, the braiding on $\cC$ and the right $A$-module structure $\mu_{M^*}^r$ of Lemma \ref{lem:dual-action} induce a left $A$-module structure $\mu^l_{M^*} =\mu^r_{M^*}\circ c_{A,M^*}$.
It is straightforward from the definitions that
\begin{equation}\label{eqn:mu-M*-left}
    \ev_M\circ(\mu_{M^*}^l\otimes\id_M) = \ev_M\circ(\id_{M^*}\otimes\mu_M^l)\circ(c_{A,M^*}\otimes\id_M).
\end{equation}
Similarly, $M$ has a right $A$-module structure given by $\mu_{M}^r = \mu^l_M\circ c_{M,A}$,
and this gives ${}^*M$ a left $A$-module structure $\mu^l_{{}^*M}$ as in Lemma \ref{lem:dual-action}. The definitions yield the identity
\begin{equation}\label{eqn:mu-*M-right}
    (\mu_{{}^*M}^l\otimes\id_{M})\circ(\id_A\otimes\coev'_M) = (\id_{{}^*M}\otimes\mu_M^l)\circ(\id_{{}^*M}\otimes c_{M,A})\circ(\coev'_M\otimes\id_A).
\end{equation}
Using these identities, we can prove:

\begin{lemma}\label{lem:ident-M*-and-*M}
    Let $A$ be a commutative algebra in a rigid braided monoidal category. If $(M,\mu_M^l)$ is a left $A$-module, then $\varphi_M^-:(M^*,\mu_{M^*}^l)\rightarrow ({}^*M,\mu_{{}^*M}^l)$ is an isomorphism of left $A$-modules.
\end{lemma}

\allowdisplaybreaks

\begin{proof}
    We need to show $\varphi_M^-\circ\mu_{M^*}^l=\mu^l_{{}^*M}\circ(\id_A\otimes\varphi^-_M)$. We prove this with graphical calculus in which
    \begin{equation*}
    \begin{tikzpicture}[scale = 1, baseline = {(current bounding box.center)}, line width=0.75pt]
    \draw (1,0) .. controls (1,.7) and (0,.3) .. (0,1);
\draw[white, double=black, line width = 3pt ] (0,0) .. controls (0,.7) and (1,.3) .. (1,1);
    \end{tikzpicture}\qquad\text{and}\qquad
    \begin{tikzpicture}[scale = 1, baseline = {(current bounding box.center)}, line width=0.75pt]
    \draw (0,0) .. controls (0,.7) and (1,.3) .. (1,1);
    \draw[white, double=black, line width = 3pt ] (1,0) .. controls (1,.7) and (0,.3) .. (0,1);
    \end{tikzpicture}
    \end{equation*}
    represent the braiding $c$ and inverse braiding $c^{-1}$, respectively:
    \begin{align*}
        \varphi_M^-\circ\mu^l_{M^*} & =  \begin{tikzpicture}[scale = .8, baseline = {(current bounding box.center)}, line width=0.75pt]
\draw (0,0) -- (.5,.5) -- (.5,2) .. controls (.5,2.7) and (1.5, 2.3) .. (1.5,3) .. controls (1.5,4) and (2.5,4) .. (2.5,3);
\draw[white, double=black, line width = 3pt ] (2.5,3) -- (2.5,2) .. controls (2.5,1) and (1.5,1) .. (1.5,2) .. controls (1.5,2.7) and (.5,2.3) .. (.5,3) -- (.5,4);
\draw (1,0) -- (.5,.5);
\node at (.5,.5) [draw,thick, fill=white] {\tiny{$\mu_{M^*}^l$}};
\node at (2,4) {\tiny{$\ev_M$}};
\node at (2,1) {\tiny{$\coev_M'$}};
\node at (0,-.25) {$A$};
\node at (1,-.25) {$M^*$};
\node at (.5, 4.25) {${}^*M$};
\end{tikzpicture}
=
\begin{tikzpicture}[scale = .8, baseline = {(current bounding box.center)}, line width=0.75pt]
\draw (0,0) -- (0,1) .. controls (0,1.7) and (1,1.3) .. (1,2) -- (1.5,2.5) -- (2,2) .. controls (2,1.3) and (1,1.7) .. (1,1) -- (1,0);
\draw[white, double=black, line width = 3pt ] (1.5,2.5) -- (1.5,3) .. controls (1.5,4) and (3,4) .. (3,3) -- (3,1) .. controls (3,0) and (2,0) .. (2,1) .. controls (2,1.7) and (0,1.3) .. (0,2) -- (0,4);
\node at (1.5,2.5) [draw,thick, fill=white] {\tiny{$\mu_{M^*}^l$}};
\node at (2.25,4) {\tiny{$\ev_M$}};
\node at (2.5,0) {\tiny{$\coev_M'$}};
\node at (0,-.25) {$A$};
\node at (1,-.25) {$M^*$};
\node at (0, 4.25) {${}^*M$};
\end{tikzpicture}
\stackrel{\eqref{eqn:mu-M*-left}}{=}
\begin{tikzpicture}[scale = .8, baseline = {(current bounding box.center)}, line width=0.75pt]
\draw (1,0) -- (1,1) .. controls (1,1.7) and (2,1.3) .. (2,2) .. controls (2,2.7) and (1,2.3) .. (1,3) -- (1,4) .. controls (1,5) and (2.5,5) .. (2.5,4) -- (2.5,3.5);
\draw[white, double=black, line width = 3pt ] (0,0) -- (0,1) .. controls (0,1.7) and (1,1.3) .. (1,2) .. controls (1,2.7) and (2,2.3) .. (2,3) -- (2.5,3.5) -- (3,3) -- (3,1) .. controls (3,0) and (2,0) .. (2,1);
\draw[white, double=black, line width = 3pt ] (2,1) .. controls (2,1.7) and (0,1.3) .. (0,2) -- (0,5);
\node at (2.5,3.5) [draw,thick, fill=white] {\tiny{$\mu_{M}^l$}};
\node at (1.75,5) {\tiny{$\ev_M$}};
\node at (2.5,0) {\tiny{$\coev_M'$}};
\node at (0,-.25) {$A$};
\node at (1,-.25) {$M^*$};
\node at (0, 5.25) {${}^*M$};
\end{tikzpicture}
=
\begin{tikzpicture}[scale = .8, baseline = {(current bounding box.center)}, line width=0.75pt]
\draw (1,0) .. controls (1,.7) and (0,.3) .. (0,1) -- (0,3) .. controls (0,3.7) and (1,3.3) .. (1,4) .. controls (1,5) and (2.5,5) .. (2.5, 4) -- (2.5,3.5);
\draw[white, double=black, line width = 3pt ] (0,0) .. controls (0,.7) and (1,.3) .. (1,1) -- (1,2) .. controls (1,2.7) and (2,2.3) .. (2,3) -- (2.5,3.5) -- (3,3) -- (3,2) .. controls (3,1) and (2,1) .. (2,2);
\draw[white, double=black, line width = 3pt ] (2,2) .. controls (2,2.7) and (1,2.3) .. (1,3) .. controls (1,3.7) and (0,3.3) .. (0,4) -- (0,5);
\node at (2.5,3.5) [draw,thick, fill=white] {\tiny{$\mu_{M}^l$}};
\node at (1.75,5) {\tiny{$\ev_M$}};
\node at (2.5,1) {\tiny{$\coev_M'$}};
\node at (0,-.25) {$A$};
\node at (1,-.25) {$M^*$};
\node at (0, 5.25) {${}^*M$};
\end{tikzpicture}\nonumber\\
& = 
\begin{tikzpicture}[scale = .8, baseline = {(current bounding box.center)}, line width=0.75pt]
\draw (3,0) .. controls (3,.7) and (0,.3) .. (0,1) -- (0,3) .. controls (0,3.7) and (1,3.3) .. (1,4) .. controls (1,5) and (2.5,5) .. (2.5, 4) -- (2.5,3.5);
\draw[white, double=black, line width = 3pt ] (0,0) .. controls (0,.7) and (3,.3) .. (3,1) -- (3,2) .. controls (3,2.7) and (2,2.3) .. (2,3) -- (2.5,3.5) -- (3,3);
\draw[white, double=black, line width = 3pt ] (3,3) .. controls (3,2.3) and (2,2.7) .. (2,2) .. controls (2,1) and (1,1) .. (1,2) -- (1,3) .. controls (1,3.7) and (0,3.3) .. (0,4) -- (0,5);
\node at (2.5,3.5) [draw,thick, fill=white] {\tiny{$\mu_{M}^l$}};
\node at (1.75,5) {\tiny{$\ev_M$}};
\node at (1.5,1) {\tiny{$\coev_M'$}};
\node at (0,-.25) {$A$};
\node at (3,-.25) {$M^*$};
\node at (0, 5.25) {${}^*M$};
\end{tikzpicture}
\stackrel{\eqref{eqn:mu-*M-right}}{=}
\begin{tikzpicture}[scale = .8, baseline = {(current bounding box.center)}, line width=0.75pt]
\draw (1,0) .. controls (1,.7) and (0,.3) .. (0,1) -- (0,3) .. controls (0,3.7) and (1.5,3.3) .. (1.5,4) .. controls (1.5,5) and (3,5) .. (3, 4) -- (3,3);
\draw[white, double=black, line width = 3pt ] (0,0) .. controls (0,.7) and (1,.3) .. (1,1) -- (1,2) -- (1.5,2.5) -- (2,2) .. controls (2,1) and (3,1) .. (3,2) -- (3,3);
\draw[white, double=black, line width = 3pt ] (1.5,2.5) -- (1.5,3) .. controls (1.5,3.7) and (0,3.3) .. (0,4) -- (0,5);
\node at (1.5,2.5) [draw,thick, fill=white] {\tiny{$\mu_{{}^*M}^l$}};
\node at (2.25,5) {\tiny{$\ev_M$}};
\node at (2.5,1) {\tiny{$\coev_M'$}};
\node at (0,-.25) {$A$};
\node at (1,-.25) {$M^*$};
\node at (0, 5.25) {${}^*M$};
\end{tikzpicture}
=
\begin{tikzpicture}[scale = .8, baseline = {(current bounding box.center)}, line width=0.75pt]
\draw (1,0) -- (1,1) .. controls (1,1.7) and (2,1.3) .. (2,2);
\draw[white, double=black, line width = 3pt ] (0,0) -- (0,2) -- (.5,2.5) -- (1,2) .. controls (1,1.3) and (2,1.7) .. (2,1) .. controls (2,0) and (3,0) .. (3,1) -- (3,2) .. controls (3,3) and (2,3) .. (2,2);
\draw (.5, 2.5) -- (.5,3.25);
\node at (0.5,2.5) [draw,thick, fill=white] {\tiny{$\mu_{{}^*M}^l$}};
\node at (2.5,3) {\tiny{$\ev_M$}};
\node at (2.5,0) {\tiny{$\coev_M'$}};
\node at (0,-.25) {$A$};
\node at (1,-.25) {$M^*$};
\node at (.5, 3.5) {${}^*M$};
\end{tikzpicture}
=\mu_{{}^*M}^l\circ(\id_A\otimes\varphi_M^-),
    \end{align*}
    as required.
\end{proof}

If $M=A$, then the left $A$-action $\mu_{{}^*A}^l$ in Lemma \ref{lem:ident-M*-and-*M} agrees with \eqref{eqn:A*-and_*A} as $A$ is commutative, so we get:
\begin{corollary}\label{cor:ident-A*-and-*A}
    If $A$ is a commutative algebra in a braided monoidal category, then $\varphi_A^-: (A^*,\mu_{A^*}^r\circ c_{A,A^*})\rightarrow ({}^*A,\mu_{{}^*A}^l)$ is an isomorphism of left $A$-modules.
\end{corollary}


\subsection{Abelian rigid monoidal categories}

Here we collect properties of algebras and their modules in abelian rigid monoidal categories. We always assume the tensor product in such a category is bilinear.

\begin{lemma}\textup{\cite[Exercise~1.6.4]{etingof2016tensor}}\label{lem:adj-imply-exact}
    Any functor between abelian categories that has a left, respectively right, adjoint is left, respectively right, exact.
\end{lemma}

In particular, if $X$ is a rigid object of a monoidal category, then $X\otimes -$ has left adjoint $X^*\otimes -$ and right adjoint ${}^*X\otimes -$; similarly, $-\otimes X$ has left adjoint $-\otimes{}^*X$ and right adjoint $-\otimes X^*$ (see for example \cite[Proposition 2.10.8]{etingof2016tensor}). Thus we get:
\begin{corollary}
The tensor product functor of any abelian rigid monoidal category is biexact.
\end{corollary}

Using this corollary, the following results are straightforward \cite{etingof2004finite}: 
\begin{lemma}
Let $P$ be a projective object in an abelian rigid monoidal category $\cC$. Then:
\begin{enumerate}\label{lem:projective-facts}
    \item $P\otimes X$ and $X\otimes P$ are projective for any $X\in\cC$. 
    \item $P^*$ and ${}^*P$ are projective. 
\end{enumerate}
\end{lemma}

Now we consider projective and injective modules for algebras in abelian rigid monoidal categories:
\begin{lemma}\label{lem:algebra-properties} 
Let $A$ be an algebra in an abelian rigid monoidal category $\cC$.
\begin{enumerate}
\item If $P\in\cC$ is projective, then $A\otimes P\in\cC_A^l$ and $P\otimes A\in\cC_A^r$ are projective.


\item $(M,\mu^l_M)\in\cC_A^l$ is projective if and only if $(M^*,\mu^r_{M^*})\in\cC_A^r$ is injective.

\item $(N,\mu^r_N)\in \cC_A^r$ is projective if and only if $({}^*N,\mu^l_{{}^*N})\in \cC_A^l$ is injective.

\item If $(M,\mu^l_M)\in \cC_A^l$, then $f = (\mu^l_M\otimes\id_{M^*})\circ(\id_A\otimes\coev_M) :A\rightarrow M\otimes M^*$ is a morphism in ${}_A\cC_A$. 
\end{enumerate}
\end{lemma}

\begin{proof}
(1) The induction functor $\cC\rightarrow\cC_A^l$ given on objects by $X\mapsto(A\otimes X,\mu_A\otimes\id_X)$ is left adjoint to the forgetful functor $U: \cC_A^l\rightarrow\cC$, that is, $\Hom_{\cC}(X,U(M))\cong \Hom_{\cC_A^l}(A\otimes X,M)$ for any $X\in\cC$ and $M\in\cC_A^l$. Thus if $P\in\cC$ is projective, the functor
\begin{equation*}
    \Hom_{\cC_A^l}(A\otimes P, -)  \cong \Hom_{\cC}(P,-)\circ U
\end{equation*}
is exact since $\Hom_{\cC}(P,-)$ and $U$ are exact. This shows that $A\otimes P$ is projective in $\cC_A^l$, and $P\otimes A$ is projective in $\cC_A^r$ similarly.


(2) For $M\in\cC_A^l$ and $N\in \cC_A^r$, the anti-equivalence ${}^*(-):\cC_A^r\rightarrow \cC_A^l$ from Lemma~\ref{lem:CA-dual-antiequivalence} and the $\cC_A^l$-isomorphism ${}^*(M^*)\cong M$ imply
\begin{equation}\label{eqn:iso-for-3-4}
   \Hom_{\cC_A^r}(N,M^*) \cong \Hom_{\cC_A^l}(M,{}^*N).
\end{equation}
Thus $\Hom_{\cC_A^r}(-,M^*)$ is exact if and only if $\Hom_{\cC_A^l}(M,{}^*(-))$ is exact. As $(-)^*$ and ${}^*(-)$ are quasi-inverse anti-equivalences (and thus exact), $\Hom_{\cC_A^r}(-,M^*)$ is exact if and only if $\Hom_{\cC_A^l}(M,-)$ is exact, as required.

(3) This also follows from \eqref{eqn:iso-for-3-4}. 

(4) This is a straightforward calculation.
\end{proof}

An algebra $A$ in an abelian monoidal category is called {\it indecomposable} if it is not isomorphic to a direct sum of two algebras.
A \textit{two-sided ideal} of $A$ is an $A$-bimodule $I$ together with an $A$-bimodule injection $I\hookrightarrow A$.
An algebra $A$ is \textit{simple} if it is non-zero and its only two-sided ideals are $0$ and $A$. The next lemma will be used later in the proof of Theorem \ref{thm:EO-simple-exact}:

\begin{lemma}\textup{\cite[Lemma~B.2]{etingof2021frobenius}}\label{lem:proj-summand}
Let $A$ be an algebra in an abelian rigid monoidal category $\cC$ with a simple unit object, and suppose there exists an $\cC_A^r$-injection $A^*\hookrightarrow X\otimes A$ for some $X\in\cC$. Then:
\begin{enumerate}
    \item For any projective $P\in\cC$, $A\otimes P\in\cC_A^l$ is injective.
    
    \item If $A$ is simple and $\cC$ contains a non-zero projective object, then for any non-zero $N\in\cC_A^l$, there is a projective $P\in\cC$ such that $N\otimes P$ contains a nonzero projective object of $\cC_A^l$ as a direct summand.  
\end{enumerate}
\end{lemma}
\begin{proof}
(1) By Lemma~\ref{lem:algebra-properties}(1), $A\otimes P\in\cC_A^l$ is projective, so by Lemma~\ref{lem:algebra-properties}(2), $(A\otimes P)^*\in\cC_A^r$ is injective. Now, $(A\otimes P)^*\cong P^*\otimes A^*$ as an object of $\cC$ (see for example \cite[Exercise 2.10.7(b)]{etingof2016tensor}), and a calculation shows that the isomorphism is also a morphism in $\cC_A^r$. Thus $P^*\otimes A^*$ is injective in $\cC_A^r$, and from the inclusion $A^* \hookrightarrow X\otimes A$ we get $P^* \otimes A^* \hookrightarrow P^*\otimes X \otimes A$.  Hence $P^*\otimes A^*$ is a direct summand of $P^*\otimes X\otimes A$. Now since $P$ is projective in $\cC$, so is $P^*$ by Lemma~\ref{lem:projective-facts}(2), and then so is $P^*\otimes X$ by Lemma~\ref{lem:projective-facts}(1). Thus Lemma~\ref{lem:algebra-properties}(1) implies that $P^*\otimes X\otimes A$ is projective in $\cC_A^r$, and then so is its direct summand $P^*\otimes A^*$. Now by Lemma \ref{lem:algebra-properties}(3), ${}^*(P^*\otimes A^*)\cong{}^*((A\otimes P)^*)\cong A\otimes P$ is injective in $\cC_A^l$. 



(2) The map $f: A\rightarrow N\otimes N^*$ from Lemma \ref{lem:algebra-properties}(4) is a morphism in ${}_A\cC_A$ such that $f\circ\iota_A=\coev_N$. Since $N\neq 0$, this implies $f\neq 0$, and thus $f$ is injective because $A$ is simple. Tensoring on the right with a non-zero projective $Q\in\cC$  yields
\[ A\otimes Q \hookrightarrow N \otimes N^* \otimes Q \; \in \; \cC_A^l. \]
By (1), $A\otimes Q$ is injective in $\cC_A^l$, so it is a direct summand of $N\otimes N^*\otimes Q$.  Now, $P = N^*\otimes Q$ is a projective object in $\cC$ satisfying the claim. Indeed, $A\otimes Q$ is projective in $\cC_A^l$ by Lemma~\ref{lem:algebra-properties}(1). To show that $A\otimes Q\neq 0$, note that $\coev'_Q: Q\otimes{}^*Q\rightarrow\vac$ is non-zero because $Q\neq 0$ and thus is surjective because $\vac$ is simple. Thus $A\otimes Q =0$ would imply a surjection
\begin{equation*}
    0=A\otimes Q\otimes{}^*Q\xrightarrow{\id_A\otimes\coev'_Q} A\otimes\vac\xrightarrow{\sim} A,
\end{equation*}
which is impossible. Consequently, $A\otimes Q$ is a non-zero projective $\cC_A^l$-direct summand of $N\otimes P$.
\end{proof}


\subsection{Tensor categories and right module categories}

Let $\KK$ be a field. A $\KK$-linear abelian category is called \textit{locally finite} if all morphism spaces are finite dimensional and every object has finite length. If a locally finite category has enough projectives and has finitely many isomorphism classes of simple objects, then it is called a finite linear category.

A \textit{tensor category} in the sense of \cite{etingof2004finite} is a locally finite $\KK$-linear abelian rigid monoidal category with a $\KK$-bilinear tensor product and simple unit object. A \textit{finite tensor category} is a tensor category with enough projectives and finitely many simple objects. A \textit{fusion category} is a semisimple finite tensor category.


\begin{lemma}\textup{\cite[Exercise~7.8.16]{etingof2016tensor}}\label{lem:loc-finite-CA}
If $A$ is an algebra in a tensor category $\cC$, then the category $\cC_A^l$ of left $A$-modules in $\cC$ is a locally finite $\KK$-linear abelian category. If $\cC$ has enough projectives, respectively finitely many isomorphism classes of simple objects, then so does $\cC_A^l$.
\end{lemma} 


Let $\cC$ be a tensor category. Following \cite{ostrik2003module},
a right $\cC$-module category is a pair $(\cM,\trl)$ consisting of a locally finite $\KK$-linear abelian category $\cM$ and a functor $\trl:\cM\times \cC\rightarrow \cM$ such that:
\begin{itemize}
\item There are natural isomorphisms
\begin{equation*}
    M\trl (X\otimes Y) \xrightarrow{\sim} (M\trl X)\trl Y, \quad M \trl \one\xrightarrow{\sim} M \qquad (X,Y\in\cC, \; M\in\cM)
\end{equation*}
satisfying obvious coherence conditions, and
\item $\trl$ is bilinear on morphisms and biexact. 
\end{itemize} 
The main source of module categories for us will be algebras in $\cC$. Indeed, given an algebra $A$, there is a functor $\trl: \cC_A^l \times \cC \rightarrow \cC_A^l$ given on objects by $(M,\mu^l_M)\trl X:= (M\otimes X,\mu^l_M \otimes \id_X)$.

\begin{lemma}\label{lem:CA-action-exact}
If $A$ is an algebra in a tensor category $\cC$, then the functor $\trl:\cC_A^l\times \cC \rightarrow \cC_A^l$ is biexact. 
\end{lemma}
\begin{proof}
For $X\in\cC$, the functor $-\trl X: \cC_A^l\rightarrow\cC_A^l$ has left and right adjoints, namely $-\trl {}^*X$ and $-\trl X^*$, respectively, so it is exact by Lemma~\ref{lem:adj-imply-exact}. Now we show that $(M,\mu^l_M)\trl -: \cC \rightarrow \cC_A^l$ is exact for every $(M,\mu^l_M)\in\cC_A^l$ (following \cite[Corollary~2.26]{douglas2019balanced}). The forgetful functor $U:\cC_A^l\rightarrow\cC$ is exact and reflects short exact sequences. Thus, $(M,\mu^l_M)\trl -: \cC \rightarrow\cC_A^l$ is exact if and only if $U\circ(M\trl -) = M \otimes -:\cC \rightarrow \cC$
is exact. Since this functor has left and right adjoints $M^*\otimes -$ and ${}^*M\otimes -$, respectively, it is exact by Lemma~\ref{lem:adj-imply-exact}. 
\end{proof}

\begin{lemma}\label{lem:CA-finite}
If $A$ is an algebra in a finite tensor category $\cC$, then $(\cC_A^l,\trl)$ is a finite right $\cC$-module category.
\end{lemma}
\begin{proof}
By Lemma~\ref{lem:loc-finite-CA}, $\cC_A^l$ is a finite linear category. The coherence and bilinearity of the $\cC$-action $\trl$ follow from the corresponding properties of $\otimes$, and $\trl$ is biexact by Lemma~\ref{lem:CA-action-exact}. Thus $\cC_A^l$ is a right $\cC$-module category as well. 
\end{proof}

We will also need the following lemma later:
\begin{lemma}\label{lem:action-preserves-proj}
    If $A$ is an algebra in a tensor category, then for any object $X\in\cC$ and any projective object $P\in\cC_A^l$, $P\trl X$ is projective in $\cC_A^l$.
\end{lemma}
\begin{proof}
    As in the previous lemma, $\Hom_{\cC_A^l}(P\trl X, -)\cong\Hom_{\cC_A^l}(P,-\trl X^*)$. Thus $\Hom_{\cC_A^l}(P\trl X, -)$ is the composition of the exact functors $\Hom_{\cC_A^l}(P,-)$ and $-\trl X^*$, implying $P\trl X$ is projective in $\cC_A^l$.
\end{proof}


Given two right $\cC$-module categories $\cM$ and $\cN$, a right $\cC$-module functor is a pair $(F,s)$ consisting of a functor $F:\cM\rightarrow\cN$ and a natural isomorphism $s_{M,X}: F(M)\trl X \rightarrow F(M\trl X)$ for $X\in\cC, M\in\cM$ which satisfies pentagon and triangle coherence conditions. Given two $\cC$-module functors $(F,s),(G,t):\cM\rightarrow\cN$, a natural transformation $\alpha:F\rightarrow G$ is called a $\cC$-module natural transformation if it satisfies
\begin{equation*}
t_{M,X}\circ (\alpha_{M}\trl \id_X) = \alpha_{M\trl X} \circ s_{M,X} : F(M)\trl X \rightarrow G(M\trl X).
\end{equation*}
Given two right $\cC$-module categories $\cM$ and $\cN$, we introduce two categories of functors:
\begin{itemize}
\item $\Fun_{\cC}(\cM,\cN)$ is the category whose objects are $\KK$-linear right $\cC$-module functors $F:\cM \rightarrow \cN$, and whose morphisms are $\cC$-module natural transformations. We write $\Fun_{\cC}(\cC_A^l):= \Fun_{\cC}(\cC_A^l,\cC_A^l)$.
\item $\Rex_{\cC}(\cM,\cN)$ is the full subcategory of right exact functors in $\Fun_{\cC}(\cM,\cN)$. We write $\Rex_{\cC}(\cC_A^l):= \Rex_{\cC}(\cC_A^l,\cC_A^l)$.
\end{itemize}
The following important result describes the second of these categories:
\begin{proposition}[Eilenberg-Watts equivalence] \textup{\cite[Theorem~4.2]{pareigis1977non}}\label{prop:Eilenberg-Watts}
Let $A$ and $B$ be algebras in a finite tensor category $\cC$. Then there is an equivalence of categories 
\begin{equation*}
    \cEW: \Rex_{\cC}(\cC_A^l,\cC_B^l) \xrightarrow{\sim} {}_B\cC_A, \qquad (F,s) \mapsto (F(A),\mu_{F(A)}^r =F(\mu_A)\circ s_{A,A}),
\end{equation*}
with quasi-inverse given by $M \mapsto (M\otimes_A -, s^M)$, where for $N\in\cC_A^l$ and $X\in\cC$, the left $B$-module isomorphism $s^M_{N,X}: (M\otimes_A N)\trl X\rightarrow M\otimes_A(N\trl X)$ is the unique map such that the diagram
    \begin{equation*}
    \begin{tikzcd}[column sep=3pc]
        (M\otimes N)\otimes X \ar[d, "\pi_{M,N}\otimes\id_X"] \ar[r, "\sim"] & M\otimes(N\otimes X) \ar[d, "\pi_{M,N\trl X}"]\\
        (M\otimes_A N) \trl X \ar[r, "s^M_{N,X}"] & M\otimes_A (N\trl X)
        \end{tikzcd}
\end{equation*}
commutes.
\end{proposition}

\begin{remark}\label{rem:Eilenberg-Watts}
When $A=B$ in the preceding proposition, both $\Rex_{\cC}(\cC_A^l)$ and ${}_A\cC_A$ are monoidal categories, where the tensor product on $\Rex_{\cC}(\cC_A^l)$ is functor composition. Since the unit and associativity isomorphisms in ${}_A\cC_A$ imply that
\begin{equation*}
  A\otimes_A -\cong\id_{\cC_A^l},\qquad  (M\otimes_A N)\otimes_A - \cong (M\otimes_A -)\circ(N\otimes_A -),
\end{equation*}
one can show that the quasi-inverse of $\mathcal{EW}$ is a monoidal functor in this case. Thus $\mathcal{EW}: \Rex_{\cC}(\cC_A^l)\rightarrow {}_A\cC_A$ is a monoidal equivalence.
\end{remark}


\subsection{FPdim and Grothendieck rings}
Let $\cC$ be a finite tensor category, and let $\{X_i\}_{i\in I}$ denote the isomorphism classes of simple objects in $\cC$. For $i\in I$, let $P_i$ denote the projective cover of $X_i$. Conversely, given an indecomposable projective object of $\cC$, its socle is a simple object. This correspondence gives a bijection between isomorphism classes of simple and indecomposable projective objects of $\cC$.

Let $\Gr(\cC)$ denote the Grothendieck group of $\cC$ \cite[\S4.5]{etingof2016tensor}. 
For $X\in\cC$, we denote its equivalence class in $\Gr(\cC)$ as $[X]$. 
Then $[X] = \sum_{i\in I} [X:X_i] [X_i]$, where $[X:X_i]$ denotes the multiplicity of $X_i$ in any Jordan-H\"older series of $X$. Thus the classes $\lbrace [X_i]\rbrace_{i\in I}$ form a basis of $\Gr(\cC)$. Also,
 $\Gr(\cC)$ is a ring with product $[X][Y]:=[X\otimes Y]$ and unit $[\one]$. In the following, we will work with the ring $\Gr_{\RR}(\cC):= \Gr(\cC)\otimes_{\ZZ}\RR$.

For any object $X\in\cC$, consider the $|I|\times |I|$ matrix with entries $[X\otimes X_i:X_j]$. The maximal non-negative eigenvalue of this matrix is a real number denoted $\FP(X)$ and called the Frobenius-Perron dimension of $X$. For each simple object $X_i\in\cC$, $\FP(X_i)\geq 1$ \cite[Proposition 3.3.4(2)]{etingof2005fusion}. If $\FP(X)\in\ZZ$ for all objects $X\in\cC$, then $\cC$ is called an \textit{integral} finite tensor category. It is well known that a finite tensor category is integral if and only if there exists a finite-dimensional quasi-Hopf algebra $H$ such that $\cC$ is tensor equivalent to the category of finite-dimensional $H$-modules. The element 
\[ R = \sum_{i\in I} \FP(X_i) [P_i] \in \Gr_{\RR}(\cC) \]
is called the regular object of $\cC$ \cite[\S6.1]{etingof2016tensor}, and it has the following property:

\begin{proposition} \textup{\cite[Proposition~6.1.11]{etingof2016tensor}}\label{prop:FPdim}
For any $Z\in\cC$, $[Z]R = R[Z] = \FP(Z)R$ in $\Gr_\RR(\cC)$. 
\end{proposition}

Now consider a right $\cC$-module category $\cM$. We define $\Gr'(\cM)$ to be the split Grothendieck group of $\cM$. This is the quotient of the $\ZZ$-module generated by all isomorphism classes $[M]$ (for $M\in\cM$) by the submodule generated by all relations $[M]=[L]+[N]$ (where $M\cong L\oplus N$). We denote $\Gr'_{\RR}(\cM):= \Gr'(\cM)\otimes_{\ZZ} \RR$.
Using the right $\cC$-action on $\cM$, $\Gr'(\cM)$ becomes a right $\Gr(\cC)$-module. Namely, for $X\in\cC$ and $M\in\cM$, we set $[M][X]=[M\trl X]$. In fact, this lifts to an action of $\Gr_{\RR}(\cC)$ on $\Gr'_{\RR}(\cM)$.


\subsection{Exact algebras in finite tensor categories}\label{subsec:exact-algebras}
Let $\cC$ be a finite tensor category and $\cM$ a finite right $\cC$-module category. We call $\cM$ \textit{exact} if for all projective $P\in\cC$ and all $M\in \cM$, the object $M\trl P\in\cM$ is projective. We call an algebra $A$ exact if the $\cC$-module category $\cC_A^l$ is exact.
The following result is a key property of exact module categories that will then allow us to obtain duals in ${}_A\cC_A$: 

\begin{proposition}\label{prop:exact-module-porperty}
\textup{\cite[Proposition~3.11]{etingof2004finite}}
Let $\cC$ be a finite tensor category and $\cM$ an exact right $\cC$-module category. Then any $\KK$-linear $\cC$-module functor from $\cM$ is exact. 
\end{proposition}

\begin{theorem}\label{thm:exact-rigid}
If $A$ is an exact indecomposable algebra in a finite tensor category $\cC$, then the category ${}_A\cC_A$ is rigid and hence a finite tensor category.
\end{theorem}
\begin{proof}
To begin, ${}_A\cC_A$ is a locally finite $\KK$-linear abelian monoidal category. To prove that ${}_A\cC_A$ is rigid:
\begin{enumerate}
\item By Proposition \ref{prop:Eilenberg-Watts} and Remark \ref{rem:Eilenberg-Watts}, there is a monoidal equivalence $\cEW:  \Rex_{\cC}(\cC_A^l)\xrightarrow{\sim} {}_A\cC_A$,
so it suffices to prove that $\Rex_{\cC}(\cC_A^l)$ is rigid.

\item As $A$ is exact, $\cC_A^l$ is by definition an exact right $\cC$-module category. Thus, by Proposition~\ref{prop:exact-module-porperty}, any functor $F\in\Fun_{\cC}(\cC_A^l)$ is exact. Therefore, $\Fun_{\cC}(\cC_A^l) = \Rex_{\cC}(\cC_A^l)$.

\item Since $\cC_A^l$ is a finite linear category, left (or right) exactness of $F:\cC_A^l\rightarrow\cC_A^l$ implies that $F$ has a left (or right) adjoint \cite[Proposition~1.7]{douglas2019balanced}. Hence, any $F\in\Rex_{\cC}(\cC_A^l)$ admits both adjoints.

\item The left and right adjoints of $F$ are also $\cC$-module functors (see \cite[Lemma~2.11]{douglas2019balanced} for a proof). Thus, $F^{\radj},F^{\ladj}\in\Fun_{\cC}(\cC_A^l) = \Rex_{\cC}(\cC_A^l)$. As these are dual objects of $F\in\Rex_{\cC}(\cC_A^l)$, we conclude that $\Rex_{\cC}(\cC_A^l)$ is rigid. 
\end{enumerate}
Lastly, ${}_A\cC_A$ is finite by \cite[Proposition~3.23]{etingof2004finite}. 
\end{proof}
\begin{remark}
It seems known to experts that the converse of this result is also true. Namely, if $\cC$ is a finite tensor category and ${}_A\cC_A$ is rigid, then $A$ is an exact algebra. See \cite[\S5.1]{shimizu2024exact} for a proof. 
\end{remark}


\allowdisplaybreaks

In light of Theorem~\ref{thm:exact-rigid}, it is important to know when an algebra is exact. In \cite[Conjecture B.6]{etingof2021frobenius}, Etingof and Ostrik conjectured that an indecomposable algebra $A$ in a finite tensor category is exact if and only if $A$ is simple. They also proved exactness for simple algebras in finite tensor categories under certain additional technical conditions. We give an exposition of their results here, and in particular we give a proof for criterion (c) in the following theorem:

\begin{theorem}\textup{\cite[\S B1]{etingof2021frobenius}}\label{thm:EO-simple-exact}
If $A$ is an indecomposable algebra in a finite tensor category $\cC$, then the following are equivalent:
\begin{itemize}
    \item[(a)] $A$ is exact.
    \item[(b)] $A$ is simple and there is a $\cC_A^l$-embedding ${}^*A\hookrightarrow A\otimes X$ for some $X\in\cC$.
    \item[(c)] $A$ is simple and $A^*\otimes_A {}^*A\neq 0$.  
    \item[(d)] $A$ is simple and there is an $\cC_A^r$-embedding $A^*\hookrightarrow X\otimes A$ for some $X\in\cC$.
\end{itemize}
\end{theorem}
\begin{proof}
(a)$\implies$(b): By Theorem~\ref{thm:exact-rigid}, ${}_A\cC_A$ is a finite tensor category, so its unit object $A$ is simple \cite[Lemma~3.24]{etingof2004finite} Thus, $A$ is a simple algebra since it is simple as an $A$-bimodule.
To find a $\cC_A^l$-embedding ${}^*A\hookrightarrow A\otimes X$, we take $X=I({}^*A)$, the injective hull of ${}^*A$ in $\cC_A^l$. Exactness of $A$ implies that projectives and injectives coincide in $\cC_A^l$ \cite[Corollary~3.6]{etingof2004finite}, and thus $X$ is also projective in $\cC_A^l$. Moreover, $\mu^l_X: A\otimes X\rightarrow X$ is a $\cC_A^l$-epimorphism, so $X$ is a direct summand of $A\otimes X$. Thus there are embeddings ${}^*A\hookrightarrow I({}^*A)=X \hookrightarrow A\otimes X$.

(b)$\implies$(c): Suppose there is a $\cC_A^l$-embedding ${}^*A\hookrightarrow A\otimes X$ for some $X\in\cC$. Similar to \cite[Lemma 7.8.24]{etingof2016tensor}, there is an isomorphism
\begin{align*}
 \Hom_\cC(A^*\otimes_A {}^*A, X) & \xrightarrow{\sim} \Hom_{\cC_A^l}({}^*A, A\otimes X)\nonumber\\
f & \mapsto (\id_A\otimes f)\circ(\id_A\circ\pi_{A^*,{}^*A})\circ(\coev_A\otimes\id_{{}^*A}).
\end{align*}
Since the hom space on the right is non-zero, so is the space on the left, and thus $A^*\otimes_A{}^*A\neq 0$.

(c)$\implies$(d): Take $X=A^*\otimes_A{}^*A\neq 0$. A calculation shows that the composition
\begin{equation*}
    f : A^* \xrightarrow{\id_{A^*}\otimes\coev'_A} A^*\otimes {}^*A\otimes A \xrightarrow{\pi_{A^*,{}^*A}\otimes\id_A} X\otimes A
\end{equation*}
is an $\cC_A^r$-morphism, and $f\neq 0$ because $\pi_{A^*,{}^*A}\neq 0$. To prove that $f$ is an embedding, it suffices to show that its dual $\cC_A^l$-morphism ${}^*f: {}^*A\otimes {}^*X \rightarrow A$ is an epimorphism. Explicitly, ${}^*f$ is the composition
\begin{equation*}
    {}^*f : {}^*A \otimes {}^*X \xrightarrow{\coev_A \otimes \id_{{}^*A\otimes {}^*X}} 
    A \otimes A^* \otimes {}^*A \otimes {}^*X \xrightarrow{\id_A \otimes \pi_{A^*,{}^*A} \otimes \id_{{}^*X}} 
    A \otimes X \otimes {}^*X \xrightarrow{\id_A \otimes \ev'_X} A.
\end{equation*}
By Lemma~\ref{lem:right-module}, ${}^*X$ is a right $A$-module with action $\mu^r_{{}^*X}$, so ${}^*A\otimes {}^*X$ is also a right $A$-module. We show that ${}^*f$ is a morphism of right $A$-modules with the following diagrammatic calculation:
\begin{align*} 
\mu_A\circ({}^*f & \otimes\id_A)  =
\begin{tikzpicture}[scale = .8, baseline = {(current bounding box.center)}, line width=0.75pt]
\draw (4,0) -- (4,3) -- (2,3.5) -- (0,3) -- (0,1) .. controls (0,0) and (1,0) .. (1,1) -- (1,1.5);
\draw (2,0) -- (2,1.5);
\draw (2,3.5) -- (2,4);
\draw (1.5,1.5) -- (1.5,2) .. controls (1.5,3) and (3,3) .. (3,2) -- (3,0);
\node at (1.5,1.5) [draw,thick, fill=white] {\tiny{$\pi_{A^*,{}^*A}$}};
\node at (.5,0)  {\tiny{$\coev_A$}};
\node at (2.25,3) {\tiny{$\ev'_{X}$}};
\node at (2,3.5) [draw, thick, fill=white] {\tiny{$\mu_A$}};
\node at (2,-.25) {${}^*A$};
\node at (3,-.25) {${}^*X$};
\node at (4,-.25) {$A$};
\node at (2,4.25) {$A$};
\end{tikzpicture}
\stackrel{\eqref{eq:right-action-1}}{=} 
\begin{tikzpicture}[scale = .8, baseline = {(current bounding box.center)}, line width=0.75pt]
\draw (5,0) -- (5,3) -- (2.5, 3.5) -- (0,3) -- (0,1) .. controls (0,0) and (1,0) .. (1,1) -- (1,2) .. controls (1,3) and (4,3) .. (4,2) -- (4,.5);
\draw (2,0) -- (2,1) .. controls (2,2) and (3,2) .. (3,1) -- (3,.5);
\draw (3.5,0) -- (3.5,.5);
\draw (2.5,3.5) -- (2.5,4);
\node at (3.5,.5) [draw,thick, fill=white] {$\;\; \alpha\;\; $};
\node at (.5,0) {\tiny{$\coev_A$}};
\node at (2.5,3) {\tiny{$\ev_A$}};
\node at (2.5,2) {\tiny{$\ev'_{{}^*A}$}};
\node at (2.5,3.5) [draw, thick, fill=white] {\tiny{$\mu_A$}};
\node at (2,-.25) {${}^*A$};
\node at (3.5,-.25) {${}^*X$};
\node at (5,-.25) {$A$};
\node at (2.5,4.25) {$A$};
\end{tikzpicture}
=
\begin{tikzpicture}[scale = .8, baseline = {(current bounding box.center)}, line width=0.75pt]
\draw (0,0) -- (0,1) .. controls (0,2) and (1,2) .. (1,1) -- (1,.5);
\draw (3,0) -- (3,1) -- (2.5,1.5) -- (2,1) -- (2,.5);
\draw (1.5,0) -- (1.5,.5);
\draw (2.5,1.5) -- (2.5,2);
\node at (1.5,.5) [draw,thick, fill=white] {$\;\; \alpha\;\; $};
\node at (.5,2) {\tiny{$\ev'_{{}^*A}$}};
\node at (2.5,1.5) [draw, thick, fill=white] {\tiny{$\mu_A$}};
\node at (0,-.25) {${}^*A$};
\node at (1.5,-.25) {${}^*X$};
\node at (3,-.25) {$A$};
\node at (2.5,2.25) {$A$};
\end{tikzpicture}\\
&\stackrel{\eqref{eq:right-action}}{=}
\begin{tikzpicture}[scale = .8, baseline = {(current bounding box.center)}, line width=0.75pt]
\draw (0,0) -- (0,2) .. controls (0,3) and (1,3) .. (1,2) -- (1,1.5);
\draw (1,0) -- (1.5,.5) -- (2,0);
\draw (1.5,.5) -- (1.5,1.5);
\draw (2,1.5) -- (2,3);
\node at (1.5,1.5) [draw,thick, fill=white] {$\;\; \alpha\;\; $};
\node at (.5,3) {\tiny{$\ev'_{{}^*A}$}};
\node at (1.5,.5) [draw, thick, fill=white] {\tiny{$\mu^r_{{}^*X}$}};
\node at (0,-.25) {${}^*A$};
\node at (1,-.25) {${}^*X$};
\node at (2,-.25) {$A$};
\node at (2,3.25) {$A$};
\end{tikzpicture}
=
\begin{tikzpicture}[scale = .8, baseline = {(current bounding box.center)}, line width=0.75pt]
\draw (2,0) -- (2,2) .. controls (2,3) and (3,3) .. (3,2) -- (3,1.5);
\draw (3,0) -- (3.5,.5) -- (4,0);
\draw (3.5,.5) -- (3.5,1.5);
\draw (4,1.5) -- (4,3) .. controls (4,4) and (1,4) .. (1,3) -- (1,2) .. controls (1,1) and (0,1) .. (0,2) -- (0,4);
\node at (3.5,1.5) [draw,thick, fill=white] {$\;\; \alpha\;\; $};
\node at (.5,1) {\tiny{$\coev_A$}};
\node at (2.5,4) {\tiny{$\ev_A$}};
\node at (2.5,3) {\tiny{$\ev'_{{}^*A}$}};
\node at (3.5,.5) [draw, thick, fill=white] {\tiny{$\mu^r_{{}^*X}$}};
\node at (2,-.25) {${}^*A$};
\node at (3,-.25) {${}^*X$};
\node at (4,-.25) {$A$};
\node at (0,4.25) {$A$};
\end{tikzpicture}
\stackrel{\eqref{eq:right-action-1}}{=} 
\begin{tikzpicture}[scale = .8, baseline = {(current bounding box.center)}, line width=0.75pt]
\draw (0,3) -- (0,1) .. controls (0,0) and (1,0) .. (1,1) -- (1,1.5);
\draw (2,0) -- (2,1.5);
\draw (1.5,1.5) -- (1.5,2) .. controls (1.5,3) and (3.5,3) .. (3.5,2) -- (3.5,.5) -- (3,0);
\draw (4,0) -- (3.5,.5);
\node at (1.5,1.5) [draw,thick, fill=white] {\tiny{$\pi_{A^*,{}^*A}$}};
\node at (.5,0) {\tiny{$\coev_A$}};
\node at (2.5,3) {\tiny{$\ev'_{X}$}};
\node at (3.5,.5) [draw,thick, fill=white] {\tiny{$\mu^r_{{}^*X}$}};
\node at (2,-.25) {${}^*A$};
\node at (3,-.25) {${}^*X$};
\node at (4,-.25) {$A$};
\node at (0,3.25) {$A$};
\end{tikzpicture}
= {}^*f\circ\mu^r_{{}^*A\otimes {}^*X}.
\end{align*}
Thus ${}^*f$ is a morphism in ${}_A\cC_A$, and ${}^*f$ is non-zero because $f$ is. As $A$ is simple in ${}_A\cC_A$, ${}^*f$ has to be an epimorphism, and then $f$ is an embedding, as desired.

(d)$\implies$(a): Our goal is to prove that for any $M\in\cC_A^l$ and projective $P\in\cC$, the object $ M\otimes P \in \cC_A^l$ is projective. It suffices to prove that $M\otimes P_i$ is projective where $\{P_i\}_{i\in I}$ are the indecomposable projectives in $\cC$. We can decompose $M\otimes P_i$ as a direct sum of indecomposable objects of $\cC_A^l$, and then write $P^{(i)}$ for the sum of all the projective direct summands in this decomposition and $N^{(i)}$ for the sum of all the non-projective direct summands. Thus $[M\otimes P_i]=[P^{(i)}]+ [N^{(i)}]$ in the split Grothendieck group $\Gr_\RR'(\cC_A^l)$. We need to show $N^{(i)}=0$ for all $i$, and to do so recall the regular object $R=\sum_{i\in I} \FP(X_i)P_i\in\Gr_\RR(\cC)$. In $\Gr_\RR'(\cC_A^l)$, we have
\begin{equation*}
    [M]R =\sum_{i\in I} \FP(X_i)[M\otimes P_i] = \sum_{i\in I}\FP(X_i)[P^{(i)}] +\sum_{i\in I}\FP(X_i)[N^{(i)}].
\end{equation*}
Writing $\cP$ for the first sum on the right above and $\cN$ for the second sum, it suffices to show $\cN=0$ because $\FP(X_i)\geq 1$ for all $i\in I$ \cite[Proposition 3.3.4(2)]{etingof2016tensor}. Now observe:
\begin{itemize}
    \item[(1)] If $\cN\neq 0$, then by Lemma~\ref{lem:proj-summand}, there are non-zero projective objects $P'\in\cC$ and $Q\in\cC_A^l$, where we can assume $Q$ is indecomposable, such that the expansion of $\cN[P']$ as a linear combination of basis elements of $\Gr_\RR'(\cC_A^l)$ contains $x[Q]$ for some $x\in\RR_{>0}$.
    
    \item[(2)] Using Proposition~\ref{prop:FPdim},
    \[ [M](R[P']) = \FP(P') [M]R = \FP(P') \cP + \FP(P') \cN .\] 
    Thus the $\FP$ of the projective part of $[M](R[P'])\in\Gr_\RR'(\cC_A^l)$ is $\FP(P')\FP(\cP)$. Here, we take $\FP$ of elements in $\Gr_\RR'(\cC_A^l)$ by applying the forgetful functor $U: \cC_A^l\rightarrow\cC$ to indecomposable objects of $\cC_A^l$ and taking $\FP$ in $\cC$.
   
    \item[(3)] On the other hand, 
    \begin{equation*}
        [M](R [P']) = ([M]R)[P'] = (\cP+\cN)[P'] = \cP[P'] + \cN[P'].
    \end{equation*} 
    By Lemma \ref{lem:action-preserves-proj}, $\cP[P']$ is a (positive real) linear combination of equivalence classes of indecomposable projective objects of $\cC_A^l$. So because $\FP$ is non-negative, the projective part of $[M](R[P'])$ has $\FP$ at least $\FP(P')\FP(\cP)+x\FP(Q)$. As $\FP(Q)\geq 1$, this contradicts (2).
\end{itemize}
The above contradiction implies $\cN=0$, and therefore $A$ is exact.
\end{proof}

If $A$ is a commutative algebra in a braided finite tensor category $\cC$, then conditions (b) and (c) of Theorem \ref{thm:EO-simple-exact} can be expressed entirely in terms of left duals. In this setting, Corollary \ref{cor:ident-A*-and-*A} yields an isomorphism $\varphi_A^-:(A^*,\mu_{A^*}^r\circ c_{A,A^*})\rightarrow ({}^*A,\mu_{{}^*A}^l)$. Moreover, as discussed in Section \ref{subsec:algebras-and-modules} (see also \cite{kirillov2002q, creutzig2017tensor}), the tensor product $\otimes_A$ on ${}_A\cC_A$ together with the embedding $\cC_A^l \xrightarrow{H_-} \cC_A^{c_{A,-}^{-1}}\subseteq {}_A\cC_A$ induces a monoidal category structure on $\cC_A^l$. Specifically, using \eqref{eq:tensor-over-A}, for $(M,\mu_M^l), (N,\mu_N^l)\in\cC_A^l$,
\begin{equation}\label{eqn:tensor-over-A-in-CA}
   M\otimes_A N = \mathrm{coequalizer} \big(
\begin{tikzcd}[column sep = 100pt]
  A\otimes M\otimes N
  \arrow[r, yshift = .3em, "{\mu^l_M\otimes \id_N}"]
  \arrow[r, yshift = -.3em, "(\id_M\otimes \mu^l_N)\circ(c_{A,M}\otimes\id_N)"']
  & M\otimes N
\end{tikzcd} \big)
\end{equation}
as an object of $\cC$. Now we can prove:
\begin{corollary}\label{cor:EO-exact-commutative}
    If $A$ is a commutative indecomposable algebra in a braided finite tensor category $\cC$, then $A$ is exact if and only if either of the following holds:
    \begin{enumerate}
        \item[(a)] $A$ is simple and there is a $\cC_A^l$-embedding $(A^*,\mu_{A^*}^r\circ c_{A,A^*})\hookrightarrow (A\otimes X,\mu_A\otimes\id_X)$ for some $X\in\cC$.

        \item[(b)] $A$ is simple and the object $(A^*,\mu_{A^*}^r\circ c_{A,A^*})\otimes_A (A^*,\mu_{A^*}^r\circ c_{A,A^*})$ of $\cC_A^l$ is non-zero.
    \end{enumerate}
\end{corollary}
\begin{proof}
 It is immediate from Corollary \ref{cor:ident-A*-and-*A} and Theorem \ref{thm:EO-simple-exact} that condition (a) holds if and only if $A$ is exact. For condition (b), we have
 \begin{align*}
     (A^*,\mu_{A^*}^r\circ c_{A,A^*})\otimes_A (A^*,\mu_{A^*}^r\circ c_{A,A^*}) & \cong H_-(A^*,\mu_{A^*}^r\circ c_{A,A^*})\otimes_A ({}^*A,\mu_{{}^*A}^l)\cong (A^*,\mu_{A^*}^r)\otimes_A ({}^*A,\mu_{{}^*A}^l)
 \end{align*}
 as objects of $\cC$, where we use Corollary \ref{cor:ident-A*-and-*A} in the first isomorphism and the definition of $H_-$ in the second. Now it is immediate from Theorem \ref{thm:EO-simple-exact} that $A$ is exact if and only if condition (b) holds.
\end{proof}

\subsection{Main results on rigidity of \texorpdfstring{$\cC_A$}{CA}}

From now on, we only consider commutative algebras $A$ in braided tensor categories, in which case the categories $\cC_A^l$ and $\cC_A^r$ of left and right $A$-modules in $\cC$ are isomorphic. Thus we now drop superscripts and write $\cC_A$ for the category of $A$-modules in $\cC$. For applications to vertex operator algebras later, we will generally consider objects of $\cC_A$ as left (rather than right) $A$-modules.

In Theorem \ref{thm:exact-rigid}, we saw that the bimodule category of an exact indecomposable algebra $A$ in a finite tensor category $\cC$ is rigid. If $A$ is also commutative, then the category $\cC_A$ of (left) $A$-modules in $\cC$ is a monoidal category, and the category of local $A$-modules $\cC_A^\loc$ is a braided monoidal category. The next result, from \cite{shimizu2024exact}, shows when $\cC_A$ and $\cC_A^\loc$ are also rigid. We call a commutative algebra $A$ haploid if $\Hom_{\cC}(\one,A) = \KK\iota_A$. 

\begin{theorem}\textup{\cite[Theorem 5.5, Corollary 5.14(b)]{shimizu2024exact}}\label{thm:com-exact-1}
If $A$ is an exact commutative haploid algebra in a braided finite tensor category $\cC$, then the category $\cC_A$ of left $A$-modules in $\cC$ is a finite tensor category and $\cC_A^{\loc}$ is a braided finite tensor category. If $\cC$ is non-degenerate, then so is $\cC_A^{\loc}$.
\end{theorem}

\begin{theorem}\label{thm:com-exact-2}
If $A$ is a commutative algebra in a braided finite tensor category $\cC$ such that any of the equivalent conditions in Theorem~\ref{thm:EO-simple-exact} is satisfied, then $\cC_A$ is a finite tensor category and $\cC_A^{\loc}$ is a braided finite tensor category. If $\cC$ is non-degenerate, then so is $\cC_A^{\loc}$. 
\end{theorem}
\begin{proof}
Since $A$ is commutative and simple, $\Hom_{\cC}(\one,A)\cong\Hom_{\cC_A}(A,A)\cong\Hom_{{}_A\cC_A}(A,A)\cong\KK$. Thus $A$ is haploid and hence an indecomposable algebra. Now $A$ is exact  by Theorem~\ref{thm:EO-simple-exact}, and the conclusions follow from Theorem~\ref{thm:com-exact-1}.
\end{proof}

In \cite[Conjecture~B.6]{etingof2021frobenius}, it is conjectured that Theorem~\ref{thm:com-exact-2} holds assuming only that $A$ is simple. Next we record some conditions on $\cC$ that guarantee this conjecture holds:

\begin{corollary}\label{cor:com-exact-integral}
If $A$ is a  simple commutative algebra in an integral braided finite tensor category $\cC$, then $A$ is exact. Consequently, $\cC_A$ is a finite tensor category and $\cC_A^{\loc}$ is a braided finite tensor category.
\end{corollary}
\begin{proof}
If $\cC$ is integral, then $\cC$ is tensor equivalent to the category of finite-dimensional modules for some quasi-Hopf algebra \cite[Proposition 2.6]{etingof2004finite}. Now $A$ is exact by \cite[Proposition~B.7, Remark~B.8]{etingof2021frobenius}, so the conclusions follow from Theorem~\ref{thm:com-exact-1}.
\end{proof}

\begin{theorem}\label{thm:comm-exact-semisimple}
If $A$ is a simple commutative algebra in a fusion category $\cC$, then $\cC_A$ is a fusion category and $\cC_A^\loc$ is a braided fusion category.
\end{theorem}
\begin{proof}
As $A$ is commutative and simple, it is haploid, and since $\cC$ is semisimple, there is a morphism  $\varepsilon_A: A\rightarrow\one$ such that $\varepsilon_A \circ\iota_A=\id_{\one}$. Now consider
\[ \phi := (\varepsilon_A\otimes\id_{A^*})\circ( \mu_A\otimes \id_{A^*})\circ (\id_A\otimes \coev_A) : A\rightarrow A^*.  \]
As in the proof of \cite[Lemma 1.20]{kirillov2002q}, $\phi$ is a non-zero map of right $A$-modules. Also, since $A$ is commutative and simple as an algebra, it is still simple as a right $A$-module; then $A^*$ is a simple right $A$-module as well. Thus $\phi$ is an isomorphism and we get a right $A$-module embedding $A^* \xrightarrow{\phi^{-1}} A\xrightarrow{\sim} \vac\otimes A$. So $A$ is exact by Theorem \ref{thm:EO-simple-exact}, and then $\cC_A$ and $\cC_A^\loc$ are finite tensor categories by Theorem \ref{thm:com-exact-1}.

Now because $\cC$ is semisimple, the tensor unit  $\one$ in particular is projective in $\cC$. Thus since $\cC_A$ is an exact $\cC$-module category, any object $M\cong M\otimes\vac$ in $\cC_A^l$ is projective. Because $\cC_A$ is locally finite, it follows that $\cC_A$ is semisimple, and then its subcategory $\cC_A^\loc$ is semisimple as well.
\end{proof}

\begin{remark}
    Theorem \ref{thm:comm-exact-semisimple} substantially strengthens \cite[Theorem 3.3]{kirillov2002q} in the case that $\cC$ is finite, since it does not require $\cC$ to be a ribbon category in which $A$ has non-zero dimension. However, \cite{kirillov2002q} does not require $\cC$ to be finite.
\end{remark}

\begin{corollary}\label{cor:com-exact-4}
If $A$ is a separable commutative haploid algebra in a fusion category $\cC$, then $\cC_A$ is a fusion category and $\cC_A^\loc$ is a braided fusion category.
\end{corollary}
\begin{proof}
Since $A$ is commutative and haploid, $\Hom_{{}_A\cC_A}(A,A) \cong\Hom_{\cC_A}(A,A) \cong \Hom_{\cC}(\one,A)\cong\KK$.
As $\cC$ is semisimple and $A$ is separable, ${}_A\cC_A$ is also semisimple \cite[Proposition~2.7]{davydov2013witt}. Thus $\Hom_{{}_A\cC_A}(A,A)\cong \KK$ implies that $A$ is a simple object of ${}_A\cC_A$ and therefore is a simple algebra. Now the conclusions follow from Theorem~\ref{thm:comm-exact-semisimple}.
\end{proof}

In the next result, instead of asking $\cC$ to satisfy certain properties, we extract properties of $A$ from the proof of Theorem \ref{thm:comm-exact-semisimple} that guarantee $A$ is exact. Recall that an algebra $A\in\cC$ is called \textit{Frobenius} if there is an isomorphism $\phi:A\rightarrow A^*$ of right $A$-modules.

\begin{corollary}\label{cor:com-exact-5}
If a simple commutative algebra $A$ in a braided finite tensor category $\cC$ satisfies either of the following conditions, then $A$ is exact:
\begin{enumerate}
\item[(a)] $A$ is a Frobenius algebra.
    \item[(b)] $\Hom_{\cC}(A,\one)\neq 0$.
\end{enumerate}
Consequently, in these cases, $\cC_A$ is a finite tensor category and $\cC_A^{\loc}$ is a braided finite tensor category.
\end{corollary}
\begin{proof}
(a) When $A$ is Frobenius, it is a simple algebra with an isomorphism $\phi:A\rightarrow A^*$ of right $A$-modules. Thus it is exact as in the proof of Theorem \ref{thm:comm-exact-semisimple}. 

(b) We just need to show $A$ is Frobenius. We can define $\phi: A\rightarrow A^*$ as in the proof of Theorem \ref{thm:comm-exact-semisimple}, using any non-zero $\cC$-morphism $\varepsilon_A: A\rightarrow\vac$. As before, $\phi$ is an isomorphism of right $A$-modules, so $A$ is Frobenius (in particular, we do not need $\varepsilon_A\circ\iota_A=\id_\vac$ to guarantee that $\phi$ is non-zero).
\end{proof}


\section{Rigidity of \texorpdfstring{$\cC$}{C} and \texorpdfstring{$\cC_A$}{CA} from rigidity of \texorpdfstring{$\cC_A^{\loc}$}{CA-loc}}\label{sec:rig-of-C}

In the previous section, in particular in Theorem \ref{thm:com-exact-2}, we gave conditions under which rigidity of a braided tensor category $\cC$ induces rigidity of the left $A$-module categories $\cC_A$ and $\cC_A^\loc$ for a commutative algebra $A$ in $\cC$. In this section, we study the converse problem, where we assume $\cC_A^{\loc}$ (but not $\cC_A$!) is rigid and find conditions under which $\cC$ and $\cC_A$ are rigid. To do this, we will need to assume that $\cC$ at least admits a weaker duality structure, called Grothendieck-Verdier duality in \cite{boyarchenko2013duality}.

\subsection{Grothendieck-Verdier categories}\label{subsec:GV-intro}
Let $\cC$ be a monoidal category, and fix an object $K\in\cC$. For an object $X\in\cC$, the contravariant functor $\Hom_\cC( - \otimes X,K)$ is representable if there is a natural isomorphism
\begin{equation*}
\Lambda_{ - ,X}: \Hom_\cC( - ,DX)\longrightarrow\Hom_\cC( - \otimes X,K)
\end{equation*}
for some object $DX\in\cC$. If this is the case, we can define
\begin{equation*}
    e_X =\Lambda_{DX,X}(\id_{DX})\in\Hom_\cC(DX\otimes X, K).
\end{equation*}
Then for any object $Y\in\cC$ and morphism $\varphi: Y\rightarrow DX$, we can apply the commutative diagram
\begin{equation*}
\begin{tikzcd}[column sep=3pc]
\Hom_\cC(DX,DX) \ar[r, "\Lambda_{DX,X}"] \ar[d, "g\mapsto g\circ\varphi"] & \Hom_\cC(DX\otimes X,K) \ar[d, "f\mapsto f\circ(\varphi\otimes\id_X)"]\\
\Hom_\cC(Y,DX) \ar[r, "\Lambda_{Y,X}"] & \Hom_\cC(Y\otimes X,K)
  \end{tikzcd}
\end{equation*}
to $\id_{DX}$ to obtain $\Lambda_{Y,X}(\varphi)=e_X\circ(\varphi\otimes\id_X)$ for all $\varphi\in\Hom_\cC(Y,DX)$. In other words, if $\Hom_\cC( - \otimes X, K)$ is representable, then there is an object $DX\in\cC$ and a morphism $e_X: DX\otimes X\rightarrow K$ which satisfies the following universal property: For any object $Y\in\cC$ and morphism $f:Y\otimes X\rightarrow K$, there is a unique morphism $\varphi: Y\rightarrow DX$ such that the diagram
\begin{equation*}
\begin{tikzcd}[column sep=3pc]
Y\otimes X \ar[d, "\varphi\otimes\id_X"'] \ar[rd, "f"] & \\
DX\otimes X \ar[r, "e_X"'] & K
\end{tikzcd}
\end{equation*}
commutes (namely, $\varphi=\Lambda_{Y,X}^{-1}(f)$). Conversely, if a pair $(DX, e_X)$ satisfying this universal property exists, then $\Hom_\cC( - \otimes X,K)$ is representable with representing object $DX$ and 
\begin{align*}
    \Lambda_{Y,X}: \Hom_\cC(Y,DX) & \rightarrow\Hom_{\cC}(Y\otimes X,K)\\
    \varphi &\mapsto e_X\circ(\varphi\otimes\id_X)
\end{align*}
for all objects $Y\in\cC$.

Now assume that $\Hom_\cC( - \otimes X,K)$ is representable for all $X\in\cC$, that is, a pair $(DX,e_X)$ satisfying the above universal property exists. Then given a morphism $f:X\rightarrow Y$ in $\cC$, there is a unique morphism $Df: DY\rightarrow DX$ such that the diagram
\begin{equation*}
\begin{tikzcd}[column sep=3pc]
  DY\otimes X \ar[d, "Df\otimes\Id_{X}"] \ar[r, "\Id_{DY}\otimes f"] & DY\otimes Y \ar[d, "e_{Y}"]\\
DX\otimes X \ar[r, "e_{X}"'] & K  
\end{tikzcd}
\end{equation*}
commutes. Clearly $D(\id_X)=\id_{DX}$ for all $X\in\cC$, and one readily checks that if $f: X\rightarrow Y$ and $g: Y\rightarrow Z$ are two morphisms in $\cC$, then 
\begin{equation*}
    e_{X}\circ((Df\circ Dg)\otimes\id_{X}) = e_{Z}\circ(\id_{DZ}\otimes(g\circ f)) = e_{X}\circ(D(g\circ f)\otimes\id_{X}),
\end{equation*}
and thus $D(g\circ f) = Df\circ Dg$ by the universal property of $(DX,e_{X})$. So we get a contravariant functor $D: \cC\rightarrow\cC$.

Following \cite{boyarchenko2013duality}, we can now define the notions of dualizing object and Grothendieck-Verdier category:
\begin{definition}
    An object $K$ in a monoidal category $\cC$ is \textit{dualizing} if $\Hom_\cC( - \otimes X,K)$ is representable for all objects $X\in\cC$ and the corresponding contravariant functor $D$ is an anti-equivalence. A \textit{(braided) Grothendieck-Verdier category} is a (braided) monoidal category $\cC$ equipped with a dualizing object $K$. 
\end{definition}


\subsection{Grothendieck-Verdier categories from \texorpdfstring{$A$}{A}-modules}
Let $(\cC,K)$ be a $\mathbb{K}$-linear abelian braided Grothendieck-Verdier category, where $\mathbb{K}$ is a field. As always, we assume the tensor product $\otimes$ is $\mathbb{K}$-bilinear. Moreover, equations (2.7) and (2.8) in \cite{boyarchenko2013duality} show that the functors $X\otimes - $ and $ - \otimes X$ have right adjoints, and thus they are right exact by Lemma \ref{lem:adj-imply-exact} (but not necessarily left exact).

 Let $(A,\mu_A,\iota_A)$ be a commutative algebra in $\cC$. Since we work only with left modules in this section, we drop superscripts from the notation, writing $\cC_A$ for the category of left $A$-modules in $\cC$ (instead of $\cC_A^l$) and $\mu_M$ for the action of $A$ on a left $A$-module (instead of $\mu_M^l$). As in Section \ref{subsec:algebras-and-modules}, because $A$ is commutative, $\cC_A$ is an abelian monoidal category with $\mathbb{K}$-bilinear right exact tensor product $\otimes_A$; see also the discussions in \cite{kirillov2002q, creutzig2017tensor}. Namely, as in \eqref{eqn:tensor-over-A-in-CA}, we identify $\cC_A$ with the subcategory $\cC_A^{c_{A,-}^{-1}}\subseteq {}_A\cC_A$, so that for $(M_1,\mu_{M_1}), (M_2,\mu_{M_2})\in\cC_A$, the tensor product $M_1\otimes_A M_2$ is the coequalizer of the two morphisms
\begin{align*}
& \mu^{(1)}: A\otimes M_1\otimes M_2 \xrightarrow{\mu_{M_1}\otimes\Id_{M_2}} M_1\otimes M_2,\nonumber\\
& \mu^{(2)}: A\otimes M_1\otimes M_2 \xrightarrow{ c_{A,M_1}\otimes\Id_{M_2}} M_1\otimes A\otimes M_2\xrightarrow{\Id_{M_1}\otimes\mu_{M_2}} M_1\otimes M_2.
\end{align*}
We write $\pi_{M_1,M_2}: M_1\otimes M_2\rightarrow M_1\otimes_A M_2$ for the canonical epimorphism such that
\begin{equation}\label{eqn:IYX_coequalizer}
 \pi_{M_1,M_2}\circ\mu^{(1)}= \pi_{M_1,M_2}\circ\mu^{(2)},
\end{equation}
and then $\mu_{M_1\otimes_A M_2}$ is defined so that
\begin{equation}\label{eqn:RepA_tens_action}
\mu_{M_1\otimes_A M_2}\circ (\id_A\otimes  \pi_{M_1,M_2}) =  \pi_{M_1,M_2}\circ\mu^{(i)}
\end{equation}
for $i=1,2$. The tensor product of morphisms in $\cC_A$ is defined as follows: If $f_1: M_1\rightarrow N_1$ and $f_2: M_2\rightarrow N_2$ are morphisms in $\cC_A$, then $f_1\otimes_A f_2: M_1\otimes_A M_2\rightarrow N_1\otimes_A N_2$ is the unique morphism such that
\begin{equation*}
(f_1\otimes_A f_2)\circ  \pi_{M_1,M_2}= \pi_{N_1,N_2}\circ(f_1\otimes f_2).
\end{equation*}
We also recall the full subcategory $\cC_A^{\loc}\subseteq\cC_A$ consisting of objects $(M,\mu_M)$ such that $\mu_M\circ c_{M,A}\circ c_{A,M}=\mu_M$.
The subcategory $\cC_A^{\loc}$ is an abelian braided monoidal category with tensor product $\otimes_A$.

The goal of this subsection is to show that $\cC_A$ is a Grothendieck-Verdier category (we will consider $\cC_A^\loc$ in the next subsection). Thus let $D$ be the dualizing functor of $(\cC,K)$. We want to show that if $(M,\mu_M)\in\cC_A$, then $DM$ also admits a left $A$-module structure. Indeed, we define $\mu_{DM}: A\otimes DM\rightarrow DM$ to be the unique $\cC$-morphism such that the diagram
\begin{equation}\label{eqn:mu-DM-def}
\begin{matrix}
\begin{tikzcd}[column sep=4pc, row sep=2pc]
A\otimes DM\otimes M \ar[d, "\mu_{DM}\otimes\id_M"] \ar[r, "c_{A,DM}\otimes\id_M"] & DM\otimes A\otimes M\ar[r, "\id_{DM}\otimes\mu_M"] & DM\otimes M \ar[d, "e_M"]\\
DM\otimes M \ar[rr, "e_M"] && K
\end{tikzcd}
\end{matrix}
\end{equation}
commutes. We represent the definition of $\mu_{DM}$ diagrammatically as follows:
\begin{equation*}
\begin{tikzpicture}[scale = .8, baseline = {(current bounding box.center)}, line width=0.75pt]
\draw (0,0) -- (.5,.5) -- (.5,1) .. controls (.5, 2) and (2,2) .. (2,1) -- (2,0);
\draw (1,0) -- (.5,.5);
\draw[dashed] (1.25,1.75) -- (1.25, 2.5);
\node at (.5,.5) [draw,thick, fill=white] {$\mu_{DM}$};
\node at (1.25,1.75) [draw,thick, fill=white] {$e_M$};
\node at (0,-.25) {$A$};
\node at (1,-.25) {$DM$};
\node at (2,-.25) {$M$};
\node at (1.25, 2.75) {$K$};
\end{tikzpicture}
=
\begin{tikzpicture}[scale = 1, baseline = {(current bounding box.center)}, line width=0.75pt]
\draw (1,0) .. controls (1,.7) and (0,.3) .. (0,1) -- (0,2) .. controls (0,3) and (1.5,3) .. (1.5,2) -- (1.5,1.5) -- (2,1) -- (2,0);
\draw[white, double=black, line width = 3pt ] (0,0) .. controls (0,.7) and (1,.3) .. (1,1) -- (1.5,1.5);
\draw[dashed] (.75,2.75) -- (.75, 3.5);
\node at (1.5,1.5) [draw,thick, fill=white] {$\mu_{M}$};
\node at (.75,2.75) [draw,thick, fill=white] {$e_M$};
\node at (0,-.25) {$A$};
\node at (1,-.25) {$DM$};
\node at (2,-.25) {$M$};
\node at (.75, 3.75) {$K$};
\end{tikzpicture}
\end{equation*}
\begin{proposition}\label{prop:X'inCA}
    $(DM,\mu_{DM})$ is an object of $\cC_A$.
\end{proposition}
\begin{proof}
We need to check the unit and associativity properties for $(DM,\mu_{DM})$. The unit property 
$\mu_{DM}\circ(\iota_A\otimes\id_{DM}) = l_{DM}$, where $l_{DM}$ is the left unit isomorphism,
follows from the diagrammatic calculation
\begin{align*}
\begin{tikzpicture}[scale = .8 , baseline = {(current bounding box.center)}, line width=0.75pt]
\draw (0,-.5) -- (0,0) -- (.5,.5) -- (.5,1) .. controls (.5, 2) and (2,2) .. (2,1) -- (2,-1);
\draw (1,-1) -- (1,0) -- (.5,.5);
\draw[dashed] (0,-1) -- (0,-.5);
\draw[dashed] (1.25,1.75) -- (1.25, 2.5);
\node at (.5,.5) [draw,thick, fill=white] {$\mu_{DM}$};
\node at (1.25,1.75) [draw,thick, fill=white] {$e_M$};
\node at (0,-.5) [draw,thick, fill=white] {$\iota_A$};
\node at (0,-1.25) {$\vac$};
\node at (1,-1.25) {$DM$};
\node at (2,-1.25) {$M$};
\node at (1.25, 2.75) {$K$};
\end{tikzpicture}
& = \begin{tikzpicture}[scale = .8, baseline = {(current bounding box.center)}, line width=0.75pt]
\draw (1,-1) -- (1,0) .. controls (1,.7) and (0,.3) .. (0,1) -- (0,2) .. controls (0,3) and (1.5,3) .. (1.5,2) -- (1.5,1.5) -- (2,1) -- (2,-1);
\draw[white, double=black, line width = 3pt ] (0,-.5) -- (0,0) .. controls (0,.7) and (1,.3) .. (1,1) -- (1.5,1.5);
\draw[dashed] (.75,2.75) -- (.75, 3.5);
\draw[dashed] (0,-1) -- (0,-.5);
\node at (1.5,1.5) [draw,thick, fill=white] {$\mu_{M}$};
\node at (.75,2.75) [draw,thick, fill=white] {$e_M$};
\node at (0,-.5) [draw,thick, fill=white] {$\iota_A$};
\node at (0,-1.25) {$\vac$};
\node at (1,-1.25) {$DM$};
\node at (2,-1.25) {$M$};
\node at (.75, 3.75) {$K$};
\end{tikzpicture}
= \begin{tikzpicture}[scale = .8, baseline = {(current bounding box.center)}, line width=0.75pt]
\draw (1,-1) .. controls (1,-.3) and (0,-.7) .. (0,0) -- (0,2) .. controls (0,3) and (1.5,3) .. (1.5,2) -- (1.5,1.5) -- (2,1) -- (2,-1);
\draw[white, double=black, line width = 3pt ] (1,.5) -- (1,1) -- (1.5,1.5);
\draw[dashed] (.75,2.75) -- (.75, 3.5);
\draw[dashed] (0,-1) .. controls (0,-.3) and (1,-.7) .. (1,0) -- (1,.5);
\node at (1.5,1.5) [draw,thick, fill=white] {$\mu_{M}$};
\node at (.75,2.75) [draw,thick, fill=white] {$e_M$};
\node at (1,.5) [draw,thick, fill=white] {$\iota_A$};
\node at (0,-1.25) {$\vac$};
\node at (1,-1.25) {$DM$};
\node at (2,-1.25) {$M$};
\node at (.75, 3.75) {$K$};
\end{tikzpicture}\nonumber
 = \begin{tikzpicture}[scale = .8, baseline = {(current bounding box.center)}, line width=0.75pt]
\draw (1,-1) .. controls (1,-.3) and (0,-.7) .. (0,0) -- (0,1) .. controls (0,2) and (2,2) .. (2,1) -- (2,-1);
\draw[dashed] (1,1.75) -- (1, 2.5);
\draw[dashed] (0,-1) .. controls (0,-.3) and (1,-.7) .. (1,0) .. controls (1,.7) .. (2,1);
\node at (1,1.75) [draw,thick, fill=white] {$e_M$};
\node at (0,-1.25) {$\vac$};
\node at (1,-1.25) {$DM$};
\node at (2,-1.25) {$M$};
\node at (1, 2.75) {$K$};
\end{tikzpicture}
= \begin{tikzpicture}[scale = .8, baseline = {(current bounding box.center)}, line width=0.75pt]
\draw (1,-1) .. controls (1,-.3) and (0,-.7) .. (0,0) -- (0,1) .. controls (0,2) and (2,2) .. (2,1) -- (2,-1);
\draw[dashed] (1,1.75) -- (1, 2.5);
\draw[dashed] (0,-1) .. controls (0,-.3) and (1,-.7) .. (1,0) .. controls (1,.7) .. (0,1);
\node at (1,1.75) [draw,thick, fill=white] {$e_M$};
\node at (0,-1.25) {$\vac$};
\node at (1,-1.25) {$DM$};
\node at (2,-1.25) {$M$};
\node at (1, 2.75) {$K$};
\end{tikzpicture}
= \begin{tikzpicture}[scale = .8, baseline = {(current bounding box.center)}, line width=0.75pt]
\draw (1,-1) -- (1,0) .. controls (1,1) and (2,1) .. (2,0) -- (2,-1);
\draw[dashed] (1.5,0.75) -- (1.5, 1.5);
\draw[dashed] (0,-1) .. controls (0,-.3) .. (1,0);
\node at (1.5,.75) [draw,thick, fill=white] {$e_M$};
\node at (0,-1.25) {$\vac$};
\node at (1,-1.25) {$DM$};
\node at (2,-1.25) {$M$};
\node at (1.5, 1.75) {$K$};
\end{tikzpicture}
\end{align*}
together with the universal property of $(M,e_M)$. The associativity property
$\mu_{DM}\circ(\Id_A\otimes\mu_{DM}) =\mu_{DM}\circ(\mu_A\otimes\Id_{DM})\circ\cA_{A,A,DM}$, where $\cA_{A,A,DM}$ is the associativity isomorphism,
follows from
\begin{align*}
\begin{tikzpicture}[scale = .8, baseline = {(current bounding box.center)}, line width=0.75pt]
\draw (0,0) -- (0,1) -- (.75,1.5) -- (.75,2) .. controls (.75,3) and (3,3) .. (3,2) -- (3,0);
\draw (1,0) -- (1.5,.5) -- (1.5,1) -- (.75,1.5);
\draw (2,0) -- (1.5,.5);
\draw[dashed] (1.875,2.75) -- (1.875, 3.5);
\node at (1.5,.5) [draw,thick, fill=white] {$\mu_{DM}$};
\node at (.75,1.5) [draw,thick, fill=white] {$\mu_{DM}$};
\node at (1.875,2.75) [draw,thick, fill=white] {$e_M$};
\node at (0,-.25) {$A$};
\node at (1,-.25) {$A$};
\node at (2,-.25) {$DM$};
\node at (3,-.25) {$M$};
\node at (1.875, 3.75) {$K$};
\end{tikzpicture}
& = \begin{tikzpicture}[scale = .8, baseline = {(current bounding box.center)}, line width=0.75pt]
\draw (1,0) -- (1.5,.5) -- (1.5,1) .. controls (1.5,1.7) and (0,1.3) .. (0,2) -- (0,3);
\draw (2,0) -- (1.5,.5);
\draw (3,0) -- (3,2) -- (2.25,2.5);
\draw[white, double=black, line width = 3pt ] (0,0) -- (0,1) .. controls (0,1.7) and (1.5,1.3) .. (1.5,2) -- (2.25,2.5) -- (2.25,3) .. controls (2.25,4) and (0,4) .. (0,3);
\draw[dashed] (1.125,3.75) -- (1.125, 4.5);
\node at (1.5,.5) [draw,thick, fill=white] {$\mu_{DM}$};
\node at (2.25,2.5) [draw,thick, fill=white] {$\mu_{M}$};
\node at (1.125,3.75) [draw,thick, fill=white] {$e_M$};
\node at (0,-.25) {$A$};
\node at (1,-.25) {$A$};
\node at (2,-.25) {$DM$};
\node at (3,-.25) {$M$};
\node at (1.125, 4.75) {$K$};
\end{tikzpicture}
= \begin{tikzpicture}[scale = .8, baseline = {(current bounding box.center)}, line width=0.75pt]
\draw (1,0) .. controls (1,.7) and (0,.3) .. (0,1) -- (0,2) -- (.5,2.5) -- (.5,3);
\draw (2,0) .. controls (2,.7) and (1,.3) .. (1,1) -- (1,2) -- (.5,2.5);
\draw (3,0) -- (3,1) -- (2.5,1.5);
\draw[white, double=black, line width = 3pt ] (0,0) .. controls (0,.7) and (2,.3) .. (2,1) -- (2.5,1.5) -- (2.5,3) .. controls (2.5,4) and (0.5,4) .. (0.5,3);
\draw[dashed] (1.5,3.75) -- (1.5, 4.5);
\node at (.5,2.5) [draw,thick, fill=white] {$\mu_{DM}$};
\node at (2.5,1.5) [draw,thick, fill=white] {$\mu_{M}$};
\node at (1.5,3.75) [draw,thick, fill=white] {$e_M$};
\node at (0,-.25) {$A$};
\node at (1,-.25) {$A$};
\node at (2,-.25) {$DM$};
\node at (3,-.25) {$M$};
\node at (1.5, 4.75) {$K$};
\end{tikzpicture}
= \begin{tikzpicture}[scale = .8, baseline = {(current bounding box.center)}, line width=0.75pt]
\draw (2,0) .. controls (2,.7) and (1,.3) .. (1,1) .. controls (1,1.7) and (0,1.3) .. (0,2) -- (0,3);
\draw[white, double=black, line width = 3pt ] (1,0) .. controls (1,.7) and (0,.3) .. (0,1) .. controls (0,1.7) and (1,1.3) .. (1,2) -- (1.75,2.5);
\draw (3,0) -- (3,1) -- (2.5,1.5);
\draw[white, double=black, line width = 3pt ] (0,0) .. controls (0,.7) and (2,.3) .. (2,1) -- (2.5,1.5) -- (2.5,2) -- (1.75,2.5) -- (1.75,3)
 .. controls (1.75,4) and (0,4) .. (0,3);
\draw[dashed] (.875,3.75) -- (.875, 4.5);
\node at (1.75,2.5) [draw,thick, fill=white] {$\mu_{M}$};
\node at (2.5,1.5) [draw,thick, fill=white] {$\mu_{M}$};
\node at (.875,3.75) [draw,thick, fill=white] {$e_M$};
\node at (0,-.25) {$A$};
\node at (1,-.25) {$A$};
\node at (2,-.25) {$DM$};
\node at (3,-.25) {$M$};
\node at (.875, 4.75) {$K$};
\end{tikzpicture}\nonumber\\
& = \begin{tikzpicture}[scale = .8, baseline = {(current bounding box.center)}, line width=0.75pt]
\draw (2,0) .. controls (2,.7) and (0,.3) .. (0,1) -- (0,4);
\draw[white, double=black, line width = 3pt ] (1,0) .. controls (1,.7) and (2,.3) .. (2,1) .. controls (2,1.7) and (1,1.3) .. (1,2) -- (1.5,2.5);
\draw (3,0) -- (3,3) -- (2.25,3.5);
\draw[white, double=black, line width = 3pt ] (0,0) .. controls (0,.7) and (1,.3) .. (1,1) .. controls (1,1.7) and (2,1.3) .. (2,2) -- (1.5,2.5) -- (1.5,3) -- (2.25, 3.5) -- (2.25,4)
 .. controls (2.25,5) and (0,5) .. (0,4);
\draw[dashed] (1.125,4.75) -- (1.125, 5.5);
\node at (1.5,2.5) [draw,thick, fill=white] {$\mu_{A}$};
\node at (2.25,3.5) [draw,thick, fill=white] {$\mu_{M}$};
\node at (1.125,4.75) [draw,thick, fill=white] {$e_M$};
\node at (0,-.25) {$A$};
\node at (1,-.25) {$A$};
\node at (2,-.25) {$DM$};
\node at (3,-.25) {$M$};
\node at (1.125, 5.75) {$K$};
\end{tikzpicture}
= \begin{tikzpicture}[scale = .8, baseline = {(current bounding box.center)}, line width=0.75pt]
\draw (2,0) -- (2,1) .. controls (2,1.7) and (0.5,1.3) .. (.5,2) -- (.5,3);
\draw[white, double=black, line width = 3pt ] (1,0) -- (.5,.5);
\draw (3,0) -- (3,2) -- (2.5,2.5);
\draw[white, double=black, line width = 3pt ] (0,0) -- (.5,.5) -- (.5,1) .. controls (.5,1.7) and (2,1.3) .. (2,2) -- (2.5,2.5) -- (2.5, 3) .. controls (2.5,4) and (.5,4) .. (.5,3);
\draw[dashed] (1.5,3.75) -- (1.5, 4.5);
\node at (.5,.5) [draw,thick, fill=white] {$\mu_{A}$};
\node at (2.5,2.5) [draw,thick, fill=white] {$\mu_{M}$};
\node at (1.5,3.75) [draw,thick, fill=white] {$e_M$};
\node at (0,-.25) {$A$};
\node at (1,-.25) {$A$};
\node at (2,-.25) {$DM$};
\node at (3,-.25) {$M$};
\node at (1.5, 4.75) {$K$};
\end{tikzpicture}
= \begin{tikzpicture}[scale = .8, baseline = {(current bounding box.center)}, line width=0.75pt]
\draw (2,0) -- (2,1) -- (1.25,1.5) -- (1.25,2) .. controls (1.25,3) and (3,3) .. (3,2) -- (3,0);
\draw[white, double=black, line width = 3pt ] (1,0) -- (.5,.5);
\draw[white, double=black, line width = 3pt ] (0,0) -- (.5,.5) -- (.5,1) -- (1.25,1.5);
\draw[dashed] (2.125,2.75) -- (2.125, 3.5);
\node at (.5,.5) [draw,thick, fill=white] {$\mu_{A}$};
\node at (1.25,1.5) [draw,thick, fill=white] {$\mu_{DM}$};
\node at (2.125,2.75) [draw,thick, fill=white] {$e_M$};
\node at (0,-.25) {$A$};
\node at (1,-.25) {$A$};
\node at (2,-.25) {$DM$};
\node at (3,-.25) {$M$};
\node at (2.125, 3.75) {$K$};
\end{tikzpicture}
\end{align*}
together with the universal property of $(M,e_M)$. This shows that $(DM,\mu_{DM})$ is an object of $\cC_A$.
\end{proof}

We will show that $\cC_A$ is a Grothendieck-Verdier category with dualizing object $(DA,\mu_{DA})$. The first step is to show that $(DM,\mu_{DM})$ is a representing object for the functor $\Hom_{\cC_A}( - \otimes_A M, DA)$:
\begin{proposition}
    For any $(M,\mu_M)\in\cC_A$, there is a $\cC_A$-morphism $e^A_M: DM\otimes_A M\rightarrow DA$ such that for any $\cC_A$-morphism $f: N\otimes_A M\rightarrow DA$, there is a unique $\cC_A$-morphism $\varphi: N\rightarrow DM$ making the diagram
    \begin{equation}\label{diag:RepA_contra_univ_prop}
   \begin{tikzcd}[column sep=3pc]
        N\otimes_A M \ar[rd, "f"] \ar[d, "\varphi\otimes_A\id_M"'] & \\
        DM\otimes_A M \ar[r, "e^A_M"'] & DA
\end{tikzcd}
    \end{equation}
    commute.
\end{proposition}
\begin{proof}
To begin, define $\tile^A_M: DM\otimes M\rightarrow DA$ to be the unique $\cC$-morphism such that the diagram
\begin{equation*}
\begin{tikzcd}[column sep=4pc, row sep=2.5pc]
DM\otimes M\otimes A \ar[d, "\tile^A_M\otimes\Id_A"] \ar[r, "\Id_{DM}\otimes c_{A,M}^{-1}"] & DM\otimes A\otimes M \ar[r, "\Id_{DM}\otimes\mu_M"] & DM\otimes M \ar[d, "e_M"]\\
 DA\otimes A \ar[rr, "e_A"] && K
\end{tikzcd}
\end{equation*}
commutes. We want to show that $\tile^A_M$ descends to a unique $\cC$-morphism $e^A_M: DM\otimes_A M\rightarrow DA$ such that
\begin{equation*}
e^A_M\circ  \pi_{DM,M}=\tile^A_M.
\end{equation*}
To do so, we need to show that $\tile^A_M\circ\mu^{(1)}=\tile^A_M\circ\mu^{(2)}$, or equivalently,
\begin{equation*}
e_A\circ ((\tile^A_M\circ\mu^{(1)})\otimes\Id_A)= e_A\circ((\tile^A_M\circ\mu^{(2)})\otimes\Id_A).
\end{equation*}
We prove this relation diagrammatically beginning with $e_A\circ ((\tile^A_M\circ\mu^{(2)})\otimes\Id_A)$:
\begin{align*}
 \begin{tikzpicture}[scale = .75, baseline = {(current bounding box.center)}, line width=0.75pt]
\draw (1,0) .. controls (1,.7) and (0,.3) .. (0,1) -- (0,4) .. controls (0,5) and (2.25,5) .. (2.25,4) -- (2.25,3.5) -- (3,3) .. controls (3,2.3) and (1.5, 2.7) .. (1.5,2) -- (1.5,1.5) -- (2,1) -- (2,0);
\draw[white, double=black, line width = 3pt ] (0,0) .. controls (0,.7) and (1,.3) .. (1,1) -- (1.5,1.5);
\draw[white, double=black, line width = 3pt ] (3,0) -- (3,2) .. controls (3,2.7) and (1.5, 2.3) .. (1.5,3) -- (2.25,3.5);
\draw[dashed] (1.125,4.75) -- (1.125, 5.5);
\node at (1.5,1.5) [draw,thick, fill=white] {$\mu_{M}$};
\node at (2.25,3.5) [draw,thick, fill=white] {$\mu_{M}$};
\node at (1.125,4.75) [draw,thick, fill=white] {$e_M$};
\node at (0,-.25) {$A$};
\node at (1,-.25) {$DM$};
\node at (2,-.25) {$M$};
\node at (3,-.25) {$A$};
\node at (1.125, 5.75) {$K$};
\end{tikzpicture}
& =\begin{tikzpicture}[scale = .75, baseline = {(current bounding box.center)}, line width=0.75pt]
\draw (1,0) .. controls (1,.7) and (0,.3) .. (0,1) -- (0,4) .. controls (0,5) and (1.75,5) .. (1.75,4) -- (1.75,3.5) -- (2.5,3) -- (2.5,2.5) -- (3,2) .. controls (3,1.3) and (2, 1.7) .. (2,1) -- (2,0);
\draw[white, double=black, line width = 3pt ] (0,0) .. controls (0,.7) and (1,.3) .. (1,1) .. controls (1,1.7) and (2,1.3) .. (2,2) -- (2.5,2.5);
\draw[white, double=black, line width = 3pt ] (3,0) -- (3,1) .. controls (3,1.7) and (1, 1.3) .. (1,2) -- (1,3) -- (1.5,3.5);
\draw[dashed] (.875,4.75) -- (.875, 5.5);
\node at (1.75,3.5) [draw,thick, fill=white] {$\mu_{M}$};
\node at (2.5,2.5) [draw,thick, fill=white] {$\mu_{M}$};
\node at (.875,4.75) [draw,thick, fill=white] {$e_M$};
\node at (0,-.25) {$A$};
\node at (1,-.25) {$DM$};
\node at (2,-.25) {$M$};
\node at (3,-.25) {$A$};
\node at (.875, 5.75) {$K$};
\end{tikzpicture}
=\begin{tikzpicture}[scale = .75, baseline = {(current bounding box.center)}, line width=0.75pt]
\draw (1,0) .. controls (1,.7) and (0,.3) .. (0,1) -- (0,3) .. controls (0,4) and (2.25,4) .. (2.25,3) -- (2.25,2.5) -- (3,2) -- (3,1) .. controls (3,.3) and (2,.7) .. (2,0);
\draw[white, double=black, line width = 3pt ] (0,0) .. controls (0,.7) and (1,.3) .. (1,1) -- (1.5,1.5) -- (1.5,2) -- (2.25,2.5);
\draw[white, double=black, line width = 3pt ] (3,0) .. controls (3,.7) and (2,.3) .. (2,1) -- (1.5,1.5);
\draw[dashed] (1.125,3.75) -- (1.125, 4.5);
\node at (2.25,2.5) [draw,thick, fill=white] {$\mu_{M}$};
\node at (1.5,1.5) [draw,thick, fill=white] {$\mu_{A}$};
\node at (1.125,3.75) [draw,thick, fill=white] {$e_M$};
\node at (0,-.25) {$A$};
\node at (1,-.25) {$DM$};
\node at (2,-.25) {$M$};
\node at (3,-.25) {$A$};
\node at (1.125, 4.75) {$K$};
\end{tikzpicture}\nonumber
 =\begin{tikzpicture}[scale = .75, baseline = {(current bounding box.center)}, line width=0.75pt]
\draw (1,0) -- (1,1) .. controls (1,1.7) and (0,1.3) .. (0,2) -- (0,3) .. controls (0,4) and (1.75,4) .. (1.75,3) -- (1.75,2.5) -- (2.5,2) -- (2.5,1.5) -- (3,1) .. controls (3,.3) and (2,.7) .. (2,0);
\draw[white, double=black, line width = 3pt ] (0,0) -- (0,1) .. controls (0,1.7) and (1,1.3) .. (1,2) -- (1.75,2.5);
\draw[white, double=black, line width = 3pt ] (3,0) .. controls (3,.7) and (2,.3) .. (2,1) -- (2.5,1.5);
\draw[dashed] (.875,3.75) -- (.875, 4.5);
\node at (1.75,2.5) [draw,thick, fill=white] {$\mu_{M}$};
\node at (2.5,1.5) [draw,thick, fill=white] {$\mu_{M}$};
\node at (.875,3.75) [draw,thick, fill=white] {$e_M$};
\node at (0,-.25) {$A$};
\node at (1,-.25) {$DM$};
\node at (2,-.25) {$M$};
\node at (3,-.25) {$A$};
\node at (.875, 4.75) {$K$};
\end{tikzpicture}
=\begin{tikzpicture}[scale = .75, baseline = {(current bounding box.center)}, line width=0.75pt]
\draw (1,0) -- (1,1) -- (0.5,1.5) -- (.5,2) .. controls (.5,3) and (2.5,3) .. (2.5,2) -- (2.5,1.5) -- (3,1) .. controls (3,.3) and (2,.7) .. (2,0);
\draw[white, double=black, line width = 3pt ] (0,0) -- (0,1) -- (.5,1.5);
\draw[white, double=black, line width = 3pt ] (3,0) .. controls (3,.7) and (2,.3) .. (2,1) -- (2.5,1.5);
\draw[dashed] (1.5,2.75) -- (1.5, 3.5);
\node at (.5,1.5) [draw,thick, fill=white] {$\mu_{DM}$};
\node at (2.5,1.5) [draw,thick, fill=white] {$\mu_{M}$};
\node at (1.5,2.75) [draw,thick, fill=white] {$e_M$};
\node at (0,-.25) {$A$};
\node at (1,-.25) {$DM$};
\node at (2,-.25) {$M$};
\node at (3,-.25) {$A$};
\node at (1.5, 3.75) {$K$};
\end{tikzpicture}
\end{align*}
which is $e_A\circ((\tile^A_M\circ\mu^{(1)})\otimes\Id_A)$. This proves that $e^A_M: DM\otimes_A M\rightarrow DA$ exists.

To show that $e^A_M$ is a morphism in $\cC_A$, we note that $\Id_A\otimes  \pi_{DM,M}$ is surjective because $ \pi_{DM,M}$ is surjective and the tensor product on $\cC$ is right exact. Thus it is enough to show that
\begin{equation*}
\mu_{DA}\circ(\Id_A\otimes e^A_M)\circ(\Id_A\otimes  \pi_{DM,M}) = e^A_M\circ\mu_{DM\otimes_A M}\circ(\Id_A\otimes  \pi_{DM,M}).
\end{equation*}
By \eqref{eqn:RepA_tens_action}, the definition of $e^A_M$, and the universal property of $(DA,e_A)$, this is equivalent to proving
\begin{equation*}
e_A\circ\left((\mu_{DA}\circ(\Id_A\otimes \tile^A_M))\otimes\Id_A\right) = e_A\circ((\tile^A_M\circ\mu^{(i)})\otimes\Id_A)
\end{equation*}
for $i=1,2$. We analyze the left side diagrammatically, beginning with the definition of $\mu_{DA}$:
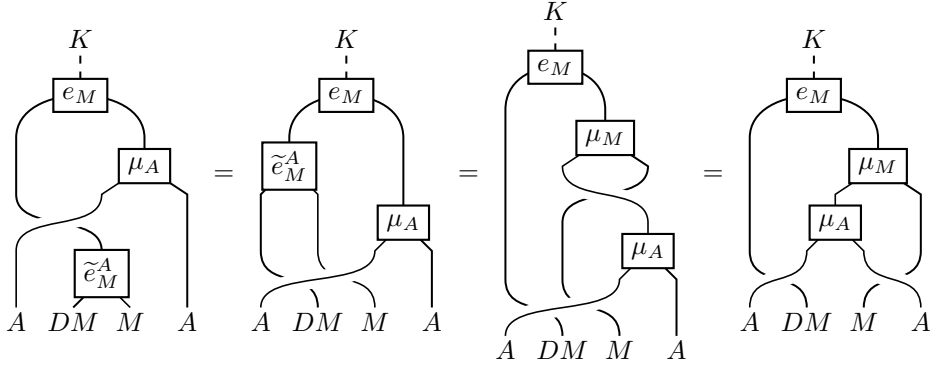
\begin{align*}
\begin{tikzpicture}[scale = .75, baseline = {(current bounding box.center)}, line width=0.75pt]
\draw (1,0) -- (1.5,.5);
\draw (1.5,1) .. controls (1.5,1.7) and (0,1.3) .. (0,2) -- (0,3) .. controls (0,4) and (2.25, 4) .. (2.25,3) -- (2.25,2.5) -- (3,2) -- (3,0); 
\draw[white, double=black, line width = 3pt ] (2,0) -- (1.5,.5);
\draw[white, double=black, line width = 3pt ] (0,0) -- (0,1) .. controls (0,1.7) and (1.5,1.3) .. (1.5,2) -- (2.25,2.5);
\draw[dashed] (1.125,3.75) -- (1.125, 4.5);
\node at (1.5,.6) [draw,thick, fill=white] {$\tile^A_M$};
\node at (2.25,2.5) [draw,thick, fill=white] {$\mu_{A}$};
\node at (1.125,3.75) [draw,thick, fill=white] {$e_M$};
\node at (0,-.25) {$A$};
\node at (1,-.25) {$DM$};
\node at (2,-.25) {$M$};
\node at (3,-.25) {$A$};
\node at (1.125, 4.75) {$K$};
\end{tikzpicture}
= \begin{tikzpicture}[scale = .75, baseline = {(current bounding box.center)}, line width=0.75pt]
\draw (1,0) .. controls (1,.7) and (0,.3) .. (0,1) -- (0,2) -- (.5,2.5);
\draw (.5,2.5) -- (.5,3) .. controls (.5,4) and (2.5, 4) .. (2.5,3) -- (2.5,1.5) -- (3,1) -- (3,0); 
\draw[white, double=black, line width = 3pt ] (2,0) .. controls (2,.7) and (1,.3) .. (1,1) -- (1,2) -- (.5,2.5);
\draw[white, double=black, line width = 3pt ] (0,0) .. controls (0,.7) and (2,.3) .. (2,1) -- (2.5,1.5);
\draw[dashed] (1.5,3.75) -- (1.5, 4.5);
\node at (.5,2.5) [draw,thick, fill=white] {$\tile^A_M$};
\node at (2.5,1.5) [draw,thick, fill=white] {$\mu_{A}$};
\node at (1.5,3.75) [draw,thick, fill=white] {$e_M$};
\node at (0,-.25) {$A$};
\node at (1,-.25) {$DM$};
\node at (2,-.25) {$M$};
\node at (3,-.25) {$A$};
\node at (1.5, 4.75) {$K$};
\end{tikzpicture}
= \begin{tikzpicture}[scale = .75, baseline = {(current bounding box.center)}, line width=0.75pt]
\draw (1,0) .. controls (1,.7) and (0,.3) .. (0,1) -- (0,4) .. controls (0,5) and (1.75,5) .. (1.75, 4) -- (1.75,3.5);
\draw (2.5,1.5) -- (3,1) -- (3,0); 
\draw (2,0) .. controls (2,.7) and (1,.3) .. (1,1) -- (1,2) .. controls (1,2.7) and (2.5,2.3) .. (2.5,3) -- (1.75,3.5);
\draw[white, double=black, line width = 3pt ] (0,0) .. controls (0,.7) and (2,.3) .. (2,1) -- (2.5,1.5) -- (2.5,2) .. controls (2.5,2.7) and (1,2.3) .. (1,3) -- (1.75,3.5);
\draw[dashed] (.875,4.75) -- (.875, 5.5);
\node at (1.75,3.5) [draw,thick, fill=white] {$\mu_M$};
\node at (2.5,1.5) [draw,thick, fill=white] {$\mu_{A}$};
\node at (.875,4.75) [draw,thick, fill=white] {$e_M$};
\node at (0,-.25) {$A$};
\node at (1,-.25) {$DM$};
\node at (2,-.25) {$M$};
\node at (3,-.25) {$A$};
\node at (.875, 5.75) {$K$};
\end{tikzpicture}
=\begin{tikzpicture}[scale = .75, baseline = {(current bounding box.center)}, line width=0.75pt]
\draw (1,0) .. controls (1,.7) and (0,.3) .. (0,1) -- (0,3) .. controls (0,4) and (2.25,4) .. (2.25,3) -- (2.25,2.5) -- (3,2) -- (3,1) .. controls (3,.3) and (2,.7) .. (2,0);
\draw[white, double=black, line width = 3pt ] (0,0) .. controls (0,.7) and (1,.3) .. (1,1) -- (1.5,1.5) -- (1.5,2) -- (2.25,2.5);
\draw[white, double=black, line width = 3pt ] (3,0) .. controls (3,.7) and (2,.3) .. (2,1) -- (1.5,1.5);
\draw[dashed] (1.125,3.75) -- (1.125, 4.5);
\node at (2.25,2.5) [draw,thick, fill=white] {$\mu_{M}$};
\node at (1.5,1.5) [draw,thick, fill=white] {$\mu_{A}$};
\node at (1.125,3.75) [draw,thick, fill=white] {$e_M$};
\node at (0,-.25) {$A$};
\node at (1,-.25) {$DM$};
\node at (2,-.25) {$M$};
\node at (3,-.25) {$A$};
\node at (1.125, 4.75) {$K$};
\end{tikzpicture}
\end{align*}
The earlier diagrammatic calculation now shows that this equals $e_A\circ ((\tile^A_X\circ\mu^{(i)})\otimes\Id_A)$, as required.

We now prove the universal property of $((DM,\mu_{DM}), e^A_M)$ in $\cC_A$. Thus consider $f\in\Hom_{\cC_A}(N\otimes_A M, DA)$. We need to show there is a unique $\varphi\in\Hom_{\cC_A}(N,DM)$ such that \eqref{diag:RepA_contra_univ_prop}
commutes. To prove the uniqueness, suppose such a $\varphi$ exists. We define a $\cC$-morphism $\tilf: N\otimes M\rightarrow K$ as the following composition:
\begin{align*}
\tilf: N\otimes M \xrightarrow{ \pi_{N,M}} N\otimes_A M\xrightarrow{f} DA\xrightarrow{r_{DA}^{-1}} DA\otimes\vac\xrightarrow{\Id_{DA}\otimes\iota_A} DA\otimes A\xrightarrow{e_A} K,
\end{align*} 
and then we analyze $\tilf$ using the assumed existence of $\varphi$:
\begin{align*}
\tilf 
& = e_A\circ(\Id_{DA}\otimes\iota_A)\circ r_{DA}^{-1}\circ e^A_M\circ(\varphi\otimes_A\Id_M)\circ  \pi_{N,M} \nonumber\\
& = e_A\circ(\Id_{DA}\otimes\iota_A)\circ r_{DA}^{-1}\circ e^A_M\circ  \pi_{N,M}\circ(\varphi\otimes\Id_M) \nonumber\\
& =e_A\circ(\Id_{DA}\otimes\iota_A)\circ r_{DA}^{-1}\circ\tile^A_M\circ(\varphi\otimes\Id_M) \nonumber\\
& = e_A\circ(\tile^A_M\otimes\Id_A)\circ(\Id_{DM\otimes M}\otimes\iota_A)\circ r_{DM\otimes M}^{-1}\circ(\varphi\otimes\Id_M)\nonumber\\
& = e_M\circ(\Id_{DM}\otimes\mu_M)\circ(\Id_{DM}\otimes c_{A,M}^{-1})\circ(\Id_{DM\otimes M}\otimes\iota_A)\circ (\Id_{DM}\otimes r_M^{-1})\circ(\varphi\otimes\Id_M)\nonumber\\
& = e_M\circ(\Id_{DM}\otimes\mu_M)\circ(\Id_{DM}\otimes\iota_A\otimes\Id_M)\circ(\Id_{DM}\otimes l_M^{-1})\circ(\varphi\otimes\Id_M)\nonumber\\
& = e_M\circ(\varphi\otimes\Id_M).
\end{align*}
Thus $\varphi$ must be the unique $\cC$-morphism (guaranteed by the universal property of $(DM,e_M)$) such that 
\begin{equation*}
e_M\circ(\varphi\otimes\Id_M)=\tilf.
\end{equation*}
To complete the proof, we must show that this unique $\cC$-morphism $\varphi$ is also a $\cC_A$-morphism and satisfies
\begin{equation}\label{eqn:RepA_contra_univ_prop}
e^A_M\circ(\varphi\otimes_A\Id_M)=f.
\end{equation}

To show that $\varphi$ is a morphism in $\cC_A$, the universal property of $(DM,e_M)$ implies it is enough to show
\begin{equation*}
e_M\circ\left((\mu_{DM}\circ(\Id_A\otimes \varphi))\otimes\Id_M\right) = e_M\circ\left( (\varphi\circ\mu_N)\otimes\Id_M\right).
\end{equation*}
 Then using the definitions as well as \eqref{eqn:IYX_coequalizer}, we get
\begin{align*}
e_M\circ\left((\mu_{DM}\circ(\Id_A\otimes \varphi))\otimes\Id_M\right) & = e_M\circ(\Id_{DM}\otimes\mu_M)\circ( c_{A,DM}\otimes\Id_M)\circ(\Id_A\otimes \varphi\otimes\Id_M)\nonumber\\
& =e_M\circ(\varphi\otimes\Id_M)\circ(\Id_N\otimes\mu_M)\circ( c_{A,N}\otimes\Id_M)\nonumber\\
& = e_A\circ(\Id_{DA}\otimes\iota_A)\circ r_{DA}^{-1}\circ f\circ  \pi_{N,M}\circ\mu^{(2)}\nonumber\\
& =e_A\circ(\Id_{DA}\otimes\iota_A)\circ r_{DA}^{-1}\circ f\circ  \pi_{N,M}\circ\mu^{(1)}\nonumber\\
&  = e_M\circ(\varphi\otimes\Id_M)\circ(\mu_N\otimes\Id_M),
\end{align*}
as required. Finally, to prove \eqref{eqn:RepA_contra_univ_prop}, the surjectivity of $ \pi_{N,M}$ and the university property of $(DA, e_A)$ imply that it is enough to prove
\begin{equation*}
e_A\circ((e^A_M\circ(\varphi\otimes_A\Id_M)\circ  \pi_{N,M})\otimes\Id_A) = e_A\circ((f\circ  \pi_{N,M})\otimes\Id_A).
\end{equation*}
From the definitions, the left side is
\begin{align*}
e_A & \circ(e^A_M\otimes\Id_A)\circ( \pi_{DM,M}\otimes\Id_A)\circ(\varphi\otimes\Id_{M\otimes A})\nonumber\\
& =e_A\circ(\tile^A_M\otimes\Id_A)\circ(\varphi\otimes\Id_{M\otimes A})\nonumber\\
& = e_M\circ(\Id_{DM}\otimes\mu_M)\circ(\Id_{DM}\otimes c_{M,A}^{-1})\circ(\varphi\otimes\Id_{M\otimes A})\nonumber\\
& = e_M\circ(\varphi\otimes\Id_M)\circ(\Id_{N}\otimes\mu_M)\circ(\Id_{N}\otimes c_{M,A}^{-1})\nonumber\\
& =\tilf\circ(\Id_{N}\otimes\mu_M)\circ(\Id_{N}\otimes c_{M,A}^{-1}).
\end{align*}
We continue analyzing this composition diagrammatically, using the definition of $\tilf$; the second equality uses \eqref{eqn:IYX_coequalizer}, \eqref{eqn:RepA_tens_action}, and the fact that $f$ is a morphism in $\cC_A$:
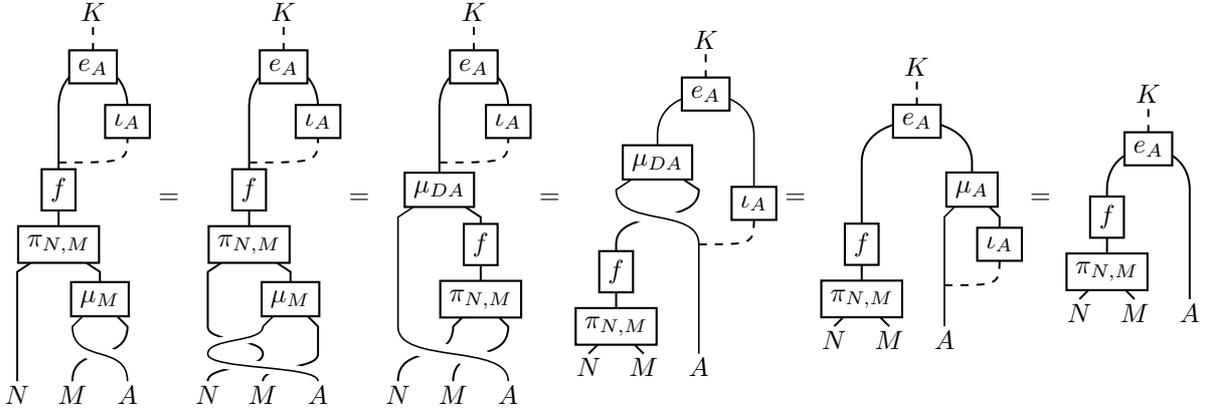
\begin{align*}
\begin{tikzpicture}[scale = .725, baseline = {(current bounding box.center)}, line width=0.75pt]
\draw (0,0) -- (0,2) -- (.75,2.5) -- (.75,5) .. controls (.75,6) and (2,6) .. (2,5);
\draw (1,0) .. controls (1,.7) and (2,.3) .. (2,1) -- (1.5,1.5) -- (1.5,2) -- (.75,2.5);
\draw[white, double=black, line width = 3pt ] (2,0) .. controls (2,.7) and (1,.3) .. (1,1) -- (1.5,1.5);
\draw[dashed] (.75,4) .. controls (2,4) .. (2,4.5);
\draw[dashed] (1.375, 5.75) -- (1.375, 6.5); 
\node at (.75,3.5) [draw,thick, fill=white] {$f$};
\node at (1.5,1.5) [draw,thick, fill=white] {$\mu_{M}$};
\node at (2,4.75) [draw,thick, fill=white] {$\iota_{A}$};
\node at (1.375,5.75) [draw,thick, fill=white] {$e_A$};
\node at (.75,2.5) [draw,thick, fill=white] {$ \pi_{N,M}$};
\node at (0,-.25) {$N$};
\node at (1,-.25) {$M$};
\node at (2,-.25) {$A$};
\node at (1.375, 6.75) {$K$};
\end{tikzpicture}
& = \begin{tikzpicture}[scale = .725, baseline = {(current bounding box.center)}, line width=0.75pt]
\draw (0,0) .. controls (0,.3) and (1,.2) .. (1,.5) .. controls (1,.8) and (0,.7) .. (0,1) -- (0,2) -- (.75,2.5) -- (.75,5) .. controls (.75,6) and (2,6) .. (2,5);
\draw (1,0) .. controls (1,.3) and (2,.2) .. (2,.5) -- (2,1) -- (1.5,1.5) -- (1.5,2) -- (.75,2.5);
\draw[white, double=black, line width = 2pt ] (2,0) .. controls (2,.3) and (0,.2) .. (0,.5) .. controls (0,.8) and (1,.7) .. (1,1) -- (1.5,1.5);
\draw[dashed] (.75,4) .. controls (2,4) .. (2,4.5);
\draw[dashed] (1.375, 5.75) -- (1.375, 6.5); 
\node at (.75,3.5) [draw,thick, fill=white] {$f$};
\node at (1.5,1.5) [draw,thick, fill=white] {$\mu_{M}$};
\node at (2,4.75) [draw,thick, fill=white] {$\iota_{A}$};
\node at (1.375,5.75) [draw,thick, fill=white] {$e_A$};
\node at (.75,2.5) [draw,thick, fill=white] {$ \pi_{N,M}$};
\node at (0,-.25) {$N$};
\node at (1,-.25) {$M$};
\node at (2,-.25) {$A$};
\node at (1.375, 6.75) {$K$};
\end{tikzpicture}
= \begin{tikzpicture}[scale = .725, baseline = {(current bounding box.center)}, line width=0.75pt]
\draw (0,0) .. controls (0,.7) and (1,.3) .. (1,1) -- (1.5,1.5) -- (1.5,3) -- (.75,3.5) -- (.75,5) .. controls (.75,6) and (2,6) .. (2,5);
\draw (1,0) .. controls (1,.7) and (2,.3) .. (2,1) -- (1.5,1.5);
\draw[white, double=black, line width = 3pt ] (2,0) .. controls (2,.7) and (0,.3) .. (0,1) -- (0,3) -- (.75,3.5);
\draw[dashed] (.75,4) .. controls (2,4) .. (2,4.5);
\draw[dashed] (1.375, 5.75) -- (1.375, 6.5); 
\node at (1.5,2.5) [draw,thick, fill=white] {$f$};
\node at (.75,3.5) [draw,thick, fill=white] {$\mu_{DA}$};
\node at (2,4.75) [draw,thick, fill=white] {$\iota_{A}$};
\node at (1.375,5.75) [draw,thick, fill=white] {$e_A$};
\node at (1.5,1.5) [draw,thick, fill=white] {$ \pi_{N,M}$};
\node at (0,-.25) {$N$};
\node at (1,-.25) {$M$};
\node at (2,-.25) {$A$};
\node at (1.375, 6.75) {$K$};
\end{tikzpicture}
= \begin{tikzpicture}[scale = .725, baseline = {(current bounding box.center)}, line width=0.75pt]
\draw (0,0) -- (.5,.5) -- (.5,2) .. controls (.5,2.7) and (2,2.3) .. (2,3) -- (1.25,3.5);
\draw (1,0) -- (.5,.5);
\draw[white, double=black, line width = 3pt ] (2,0) -- (2,2) .. controls (2,2.7) and (.5,2.3) .. (.5,3) -- (1.25,3.5) -- (1.25,4) .. controls (1.25,5) and (3,5) .. (3,4) -- (3,2.5);
\draw[dashed] (2,2) .. controls (3,2) .. (3,2.5);
\draw[dashed] (2.125, 4.75) -- (2.125, 5.5); 
\node at (.5,1.5) [draw,thick, fill=white] {$f$};
\node at (1.25,3.5) [draw,thick, fill=white] {$\mu_{DA}$};
\node at (3,2.75) [draw,thick, fill=white] {$\iota_{A}$};
\node at (2.125,4.75) [draw,thick, fill=white] {$e_A$};
\node at (.5,.5) [draw,thick, fill=white] {$ \pi_{N,M}$};
\node at (0,-.25) {$N$};
\node at (1,-.25) {$M$};
\node at (2,-.25) {$A$};
\node at (2.125, 5.75) {$K$};
\end{tikzpicture}\nonumber
 = \begin{tikzpicture}[scale = .725, baseline = {(current bounding box.center)}, line width=0.75pt]
\draw (0,0) -- (.5,.5) -- (.5,3) .. controls (.5,4) and (2.5,4) .. (2.5,3) -- (2.5,2.5) -- (3,2) -- (3,1.5);
\draw (1,0) -- (.5,.5);
\draw[white, double=black, line width = 3pt ] (2,0) -- (2,2) -- (2.5,2.5);
\draw[dashed] (2,.75) .. controls (3,.75) .. (3,1.25);
\draw[dashed] (1.5, 3.75) -- (1.5, 4.5); 
\node at (.5,1.5) [draw,thick, fill=white] {$f$};
\node at (2.5,2.5) [draw,thick, fill=white] {$\mu_{A}$};
\node at (3,1.5) [draw,thick, fill=white] {$\iota_{A}$};
\node at (1.5,3.75) [draw,thick, fill=white] {$e_A$};
\node at (.5,.5) [draw,thick, fill=white] {$ \pi_{N,M}$};
\node at (0,-.25) {$N$};
\node at (1,-.25) {$M$};
\node at (2,-.25) {$A$};
\node at (1.5, 4.75) {$K$};
\end{tikzpicture}
= \begin{tikzpicture}[scale = .725, baseline = {(current bounding box.center)}, line width=0.75pt]
\draw (0,0) -- (.5,.5) -- (.5,2) .. controls (.5,3) and (2,3) .. (2,2) -- (2,0);
\draw (1,0) -- (.5,.5);
\draw[dashed] (1.25, 2.75) -- (1.25, 3.5); 
\node at (.5,1.5) [draw,thick, fill=white] {$f$};
\node at (1.25,2.75) [draw,thick, fill=white] {$e_A$};
\node at (.5,.5) [draw,thick, fill=white] {$ \pi_{N,M}$};
\node at (0,-.25) {$N$};
\node at (1,-.25) {$M$};
\node at (2,-.25) {$A$};
\node at (1.25, 3.75) {$K$};
\end{tikzpicture}
\end{align*}
which is $e_A\circ((f\circ  \pi_{N,M})\otimes\Id_A)$, as required. This completes the proof of the proposition.
\end{proof}

The preceding proposition shows that for any $(M,\mu_M)\in\cC_A$, the functor $\Hom_{\cC_A}( - \otimes_A M, DA)$ is representable. Thus we get a contravariant functor $D_A: \cC_A\rightarrow\cC_A$ such that $D_A(M,\mu_M)=(DM,\mu_{DM})$, and for a morphism $f:M\rightarrow N$ in $\cC_A$, $D_A(f):DN\rightarrow DM$ is the unique map making the diagram
\begin{equation}\label{diag:DA(f)_def}
\begin{matrix}
\begin{tikzcd}[column sep=4pc]
    DN\otimes_A M \ar[r, "\id_{DN}\otimes_A f"] \ar[d, "D_A(f)\otimes_A\id_M"] & DN\otimes_A N \ar[d, "e^A_N"]\\
    DM\otimes_A M \ar[r, "e^A_M"] & DA
\end{tikzcd}
    \end{matrix}
\end{equation}
commute. In fact, $D_A(f)$ is the same as the $\cC$-morphism $Df$ defined earlier:
\begin{proposition}\label{prop:DA(f)=f'}
    For any morphism $f:M\rightarrow N$ in $\cC_A$, the $\cC_A$-morphism $D_A(f): DN\rightarrow DM$ is the same as the unique $\cC$-morphism $Df$ such that the diagram
    \begin{equation*}
    \begin{tikzcd}[column sep=4pc]
    DN\otimes M \ar[r, "\id_{DN}\otimes f"] \ar[d, "Df\otimes\id_M"] & DN\otimes N \ar[d, "e_N"]\\
    DM\otimes M \ar[r, "e_M"] & K
\end{tikzcd}
    \end{equation*}
    commutes.
\end{proposition}
\begin{proof}
    Using the unit properties of $\mu_M$ and $\mu_N$ and the definitions, including \eqref{diag:DA(f)_def}, we calculate:
    \begin{align*}
        e_M\circ(D_A(f)\otimes\id_M) & = e_M\circ(\id_{DM}\otimes\mu_M)\circ(\id_{DM}\otimes c_{A,M}^{-1})\circ(\id_{DM\otimes M}\otimes\iota_A)\circ(D_A(f)\otimes r_M^{-1})\nonumber\\
        & = e_A\circ(\tile^A_M\otimes\id_A)\circ(\id_{DM\otimes M}\otimes\iota_A)\circ(D_A(f)\otimes r_M^{-1})\nonumber\\
        & = e_A\circ(e^A_M\otimes\id_A)\circ( \pi_{DM,M}\otimes\id_A)\circ(D_A(f)\otimes\id_M\otimes\iota_A)\circ r_{DN\otimes M}^{-1}\nonumber\\
        & = e_A\circ(e^A_M\otimes\id_A)\circ((D_A(f)\otimes_A\id_M)\otimes\id_A)\circ( \pi_{DN,M}\otimes\iota_A)\circ r_{DN\otimes M}^{-1}\nonumber\\
        & = e_A\circ(e^A_N\otimes\id_A)\circ((\id_{DN}\otimes_A f)\otimes\id_A)\circ( \pi_{DN,M}\otimes\iota_A)\circ r_{DN\otimes M}^{-1}\nonumber\\
        & = e_A\circ(\tile^A_N\otimes\id_A)\circ(\id_{DN}\otimes f\otimes\iota_A)\circ r_{DN\otimes M}^{-1}\nonumber\\
        & =e_N\circ(\id_{DN}\otimes\mu_N)\circ(\id_{DN}\otimes c_{A,N}^{-1})\circ(\id_{DN\otimes N}\otimes\iota_{A})\circ r_{DN\otimes N}^{-1}\circ(\id_{DN}\otimes f)\nonumber\\
        & = e_N\circ(\id_{DN}\otimes f),
    \end{align*}
    as required.
\end{proof}

The previous proposition says that if $f: M\rightarrow N$ is a $\cC_A$-morphism, then the $\cC$-morphism $Df:DN\rightarrow DM$ is automatically a $\cC_A$-morphism as well. We also have the converse result:
\begin{lemma}\label{lem:f'hom-->fhom}
    Let $M$ and $N$ be objects of $\cC_A$, and let $f:M\rightarrow N$ be a $\cC$-morphism. If $Df:DN\rightarrow DM$ is a $\cC_A$-morphism, then so is $f$.
\end{lemma}
\begin{proof}
    We need to show that $\mu_N\circ(\id_A\otimes f) =f\circ\mu_M$. Since the contravariant functor $D:\cC\rightarrow\cC$ is an anti-equivalence, it is sufficient to prove $D(\mu_N\circ(\id_A\otimes f)) =D(f\circ\mu_M)$, and this is equivalent to showing
    \begin{equation*}
        e_N\circ(\id_{DN}\otimes(\mu_N\circ(\id_A\otimes f))) = e_N\circ(\id_{DN}\otimes(f\circ\mu_M)).
    \end{equation*}
    In fact, the definitions and the assumption that $Df$ is a morphism in $\cC_A$ imply that
    \begin{align*}
      e_N\circ(\id_{DN}\otimes(\mu_N\circ(\id_A\otimes f))) & =  e_N\circ(\mu_{DN}\otimes\id_N)\circ( c_{A,DN}^{-1}\otimes f)\nonumber\\
      & = e_M\circ(Df\otimes\id_M)\circ(\mu_{DN}\otimes\id_M)\circ( c_{A,DN}^{-1}\otimes\id_M)\nonumber\\
      & = e_M\circ(\mu_{DM}\otimes\id_M)\circ( c_{A,DM}^{-1}\otimes\id_M)\circ(Df\otimes\id_{A\otimes M})\nonumber\\
      & = e_M\circ(\id_{DM}\otimes\mu_M)\circ(Df\otimes\id_{A\otimes M})\nonumber\\
      & = e_N\circ(\id_{DN}\otimes f)\circ(\id_{DN}\otimes\mu_M),
    \end{align*}
    as required.
\end{proof}

Using the preceding lemma, we can prove:
\begin{corollary}
    The contravariant functor $D_A$ is fully faithful.
\end{corollary}
\begin{proof}
    For objects $M,N\in\cC_A$, we need to show that the map
    \begin{equation*}
        D_A: \Hom_{\cC_A}(M,N)\longrightarrow\Hom_{\cC_A}(DN,DM)
    \end{equation*}
    is an isomorphism. By Proposition \ref{prop:DA(f)=f'}, this map $D_A$ is the restriction of
    \begin{equation*}
        D: \Hom_\cC(M,N)\longrightarrow\Hom_\cC(DN,DM)
    \end{equation*}
    to the subspace $\Hom_{\cC_A}(M,N)$, and the map $D$ is an isomorphism because by assumption $(\cC,K)$ is a Grothendieck-Verdier category. Thus $D_A$ is injective, and to show that $D_A$ is surjective, any $g\in\Hom_{\cC_A}(DN,DM)$ is equal to $Df$ for some $f\in\Hom_\cC(M,N)$. By Lemma \ref{lem:f'hom-->fhom}, $f$ is actually a morphism in $\cC_A$, and thus $g=Df=D_A(f)$. This shows that $D_A$ is an isomorphism.
\end{proof}

To show that $\cC_A$ is a Grothendieck-Verdier category, it remains to show that the contravariant functor $D_A$ is essentially surjective. To this end, note that replacing braiding with inverse braiding isomorphisms in Proposition \ref{prop:X'inCA} and its proof shows that if $(M,\mu_M)$ is an object of $\cC_A$, then so is $DM_-=(DM,\mu_{DM}^-)$, where $\mu_{DM}^-:A\otimes DM\rightarrow DM$ is the unique morphism such that the diagram
\begin{equation}\label{eqn:muDM-_def}
\begin{matrix}
\begin{tikzcd}[column sep = 4pc]
A\otimes DM\otimes M \ar[d, "\mu_{DM}^-\otimes\id_M"] \ar[r, "c_{DM,A}^{-1}\otimes\id_M"] & DM\otimes A\otimes M\ar[r, "\id_{DM}\otimes\mu_M"] & DM\otimes M \ar[d, "e_M"]\\
DM\otimes M \ar[rr, "e_M"] && K 
\end{tikzcd}
\end{matrix}
\end{equation}
commutes. Thus $DM$ admits two left $A$-module structures, and we now use $DM_+$ to denote the left $A$-module structure of Proposition \ref{prop:X'inCA}.

For any object $X\in\cC$, we now define two $\cC$-morphisms $\varphi^{\pm}_X: X\rightarrow D^2 X$ to be the unique morphisms such that the diagrams
\begin{equation*}
\begin{tikzcd}[column sep=4pc]
X\otimes DX \ar[r, "c^{\pm 1}"] \ar[d, "\varphi^\pm_X\otimes\id_{DX}"] & DX\otimes X \ar[d, "e_X"]\\
D^2 X\otimes DX \ar[r, "e_{DX}"] & K
\end{tikzcd}
\end{equation*}
commute. The morphisms $\varphi^\pm_X$ define natural transformations $\varphi^\pm:\id_\cC\rightarrow D^2$, since if $f: X\rightarrow Y$ is a morphism in $\cC$, then
\begin{align*}
e_{DY}\circ((\varphi^{\pm}_Y\circ f)\otimes\Id_{DY}) & = e_Y\circ c^{\pm 1}\circ(f\otimes\Id_{DY}) = e_Y\circ(\Id_{DY}\otimes f)\circ c^{\pm1}\nonumber\\
& = e_X\circ(Df\otimes\Id_X)\circ c^{\pm1} = e_X\circ c^{\pm1}\circ(\Id_X\otimes Df)\nonumber\\
& = e_{DX}\circ(\varphi^\pm_X\otimes\Id_{DX})\circ(\Id_X\otimes Df) =e_{DX}\circ(\Id_{D^2 X}\otimes Df)\circ(\varphi^\pm_X\otimes\Id_{DY})\nonumber\\
& =e_{DY}\circ((D^2 f\circ\varphi^\pm_X)\otimes\Id_{DY}).
\end{align*}
The universal property of $(D^2 Y, e_{DY})$ then implies $\varphi^\pm_Y\circ f=D^2 f\circ\varphi^\pm_X$, as required.

If $(M,\mu_M)$ is an object of $\cC_A$, then we can ask whether $\varphi^{\pm}_M$ are morphisms in $\cC_A$. In fact:
\begin{proposition}\label{prop:phi_pm_A-hom}
    For $(M,\mu_M)\in\cC_A$, $\varphi^-_M\in\Hom_{\cC_A}(M,D(DM_-)_+)$ and $\varphi^+_M\in\Hom_{\cC_A}(X,D(DM_+)_-)$.
\end{proposition}
\begin{proof}
    For $\varphi^-_M$, we need to show that $\mu_{D(DM_-)_+}\circ(\id_A\otimes \varphi_M^-) =\varphi_M^-\circ\mu_M$. We prove this diagrammatically:
    \begin{align*}
     \begin{tikzpicture}[scale = .8, baseline = {(current bounding box.center)}, line width=0.75pt]
\draw (1,0) -- (1,1) -- (.5,1.5) -- (.5, 2) .. controls (.5,3) and (2,3) .. (2,2) -- (2,0);
\draw[white, double=black, line width = 3pt ] (0,0) -- (0,1) -- (.5,1.5);
\draw[dashed] (1.25,2.75) -- (1.25, 3.5);
\node at (.5,1.5) [draw,thick, fill=white] {$\mu_{D(DM_-)_+}$};
\node at (1,.5) [draw,thick, fill=white] {$\varphi_M^-$};
\node at (1.25,2.75) [draw,thick, fill=white] {$e_{DM}$};
\node at (0,-.25) {$A$};
\node at (1,-.25) {$M$};
\node at (2,-.25) {$DM$};
\node at (1.25, 3.75) {$K$};
\end{tikzpicture} 
& =\begin{tikzpicture}[scale = .8, baseline = {(current bounding box.center)}, line width=0.75pt]
\draw (1,0) .. controls (1,.7) and (0,.3) .. (0,1) -- (0,2) .. controls (0,3) and (1.5,3) .. (1.5,2) -- (1.5,1.5) -- (2,1) -- (2,0);
\draw[white, double=black, line width = 3pt ] (0,0) .. controls (0,.7) and (1,.3) .. (1,1) -- (1.5,1.5);
\draw[dashed] (.75,2.75) -- (.75, 3.5);
\node at (1.5,1.5) [draw,thick, fill=white] {$\mu_{DM_-}$};
\node at (0,1.5) [draw,thick, fill=white] {$\varphi_{M}^-$};
\node at (.75,2.75) [draw,thick, fill=white] {$e_{DM}$};
\node at (0,-.25) {$A$};
\node at (1,-.25) {$M$};
\node at (2,-.25) {$DM$};
\node at (.75, 3.75) {$K$};
\end{tikzpicture}
 =\begin{tikzpicture}[scale = .8, baseline = {(current bounding box.center)}, line width=0.75pt]
\draw (1,0) .. controls (1,.7) and (0,.3) .. (0,1) -- (0,2) .. controls (0,2.7) and (1.5,2.3) .. (1.5,3);
\draw[white, double=black, line width = 3pt ] (0,0) .. controls (0,.7) and (1,.3) .. (1,1) -- (1.5,1.5) -- (1.5,2) .. controls (1.5,2.7) and (0,2.3) .. (0,3) .. controls (0,4) and (1.5,4) .. (1.5,3);
\draw (2,0) -- (2,1) -- (1.5,1.5);
\draw[dashed] (.75,3.75) -- (.75, 4.5);
\node at (1.5,1.5) [draw,thick, fill=white] {$\mu_{DM_-}$};
\node at (.75,3.75) [draw,thick, fill=white] {$e_{M}$};
\node at (0,-.25) {$A$};
\node at (1,-.25) {$M$};
\node at (2,-.25) {$DM$};
\node at (.75, 4.75) {$K$};
\end{tikzpicture}
=\begin{tikzpicture}[scale = .8, baseline = {(current bounding box.center)}, line width=0.75pt]
\draw (0,0) -- (0,1) -- (.5,1.5);
\draw (1,0) .. controls (1,.7) and (2,.3) .. (2,1);
\draw[white, double=black, line width = 3pt ] (2,0) .. controls (2,.7) and (1,.3) .. (1,1) -- (.5,1.5) -- (.5,2) .. controls (.5,3) and (2,3) .. (2,2) -- (2,1);
\draw[dashed] (1.25,2.75) -- (1.25, 3.5);
\node at (.5,1.5) [draw,thick, fill=white] {$\mu_{DM_-}$};
\node at (1.25,2.75) [draw,thick, fill=white] {$e_{M}$};
\node at (0,-.25) {$A$};
\node at (1,-.25) {$M$};
\node at (2,-.25) {$DM$};
\node at (1.25, 3.75) {$K$};
\end{tikzpicture}
=\begin{tikzpicture}[scale = .8, baseline = {(current bounding box.center)}, line width=0.75pt]
\draw (0,0) -- (0,1) .. controls (0,1.7) and (1,1.3) .. (1,2) -- (1.5,2.5);
\draw (1,0) .. controls (1,.7) and (2,.3) .. (2,1) -- (2,2) -- (1.5,2.5);
\draw[white, double=black, line width = 3pt ] (2,0) .. controls (2,.7) and (1,.3) .. (1,1) .. controls (1,1.7) and (0,1.3) .. (0,2) -- (0,3) .. controls (0,4) and (1.5,4) .. (1.5,3) -- (1.5,2.5);
\draw[dashed] (.75,3.75) -- (.75, 4.5);
\node at (1.5,2.5) [draw,thick, fill=white] {$\mu_{M}$};
\node at (.75,3.75) [draw,thick, fill=white] {$e_{M}$};
\node at (0,-.25) {$A$};
\node at (1,-.25) {$M$};
\node at (2,-.25) {$DM$};
\node at (.75, 4.75) {$K$};
\end{tikzpicture}\nonumber\\
 &=\begin{tikzpicture}[scale = .8, baseline = {(current bounding box.center)}, line width=0.75pt]
\draw (0,0) -- (.5,.5);
\draw (1,0) -- (.5,.5) -- (.5,1) .. controls (.5,1.7) and (2,1.3) .. (2,2);
\draw[white, double=black, line width = 3pt ] (2,0) -- (2,1) .. controls (2,1.7) and (.5,1.3) .. (.5,2) .. controls (.5,3) and (2,3) .. (2,2);
\draw[dashed] (1.25,2.75) -- (1.25, 3.5);
\node at (.5,.5) [draw,thick, fill=white] {$\mu_{M}$};
\node at (1.25,2.75) [draw,thick, fill=white] {$e_{M}$};
\node at (0,-.25) {$A$};
\node at (1,-.25) {$M$};
\node at (2,-.25) {$DM$};
\node at (1.25, 3.75) {$K$};
\end{tikzpicture}
=\begin{tikzpicture}[scale = .8, baseline = {(current bounding box.center)}, line width=0.75pt]
\draw (0,0) -- (.5,.5);
\draw (1,0) -- (.5,.5) -- (.5,2);
\draw[white, double=black, line width = 3pt ] (2,0) -- (2,2) .. controls (2,3) and (.5,3) .. (.5,2);
\draw[dashed] (1.25,2.75) -- (1.25, 3.5);
\node at (.5,.5) [draw,thick, fill=white] {$\mu_{M}$};
\node at (.5,1.5) [draw,thick, fill=white] {$\varphi_{M}^-$};
\node at (1.25,2.75) [draw,thick, fill=white] {$e_{M}$};
\node at (0,-.25) {$A$};
\node at (1,-.25) {$M$};
\node at (2,-.25) {$DM$};
\node at (1.25, 3.75) {$K$};
\end{tikzpicture} .
    \end{align*}
    The proof that $\mu_{D(DM_+)_-}\circ(\id_A\otimes\varphi_M^+) =\varphi_M^+\circ\mu_M$ is similar.
\end{proof}

If $\varphi^-$ is a natural isomorphism, then it will follow that $D_A$ is essentially surjective, since we will then have $(M,\mu_M)\cong D_A(DM_-)$ for all $(M,\mu_M)\in\cC_A$. In fact, it is shown in \cite[\S 7.2]{boyarchenko2013duality} that $\varphi^\pm$ are natural isomorphisms; we include a proof for completeness:
\begin{proposition}
    The natural transformations $\varphi^\pm:\id_\cC\rightarrow D^2$ are natural isomorphisms.
\end{proposition}
\begin{proof}
  To show $\varphi^{\pm}_X$ is injective for $X\in\cC$, suppose $f: Y\rightarrow X$ is a morphism in $\cC$ such that $\varphi^{\pm}_X\circ f=0$.  We need to show that $f=0$. In fact, we have
  \begin{align*}
      0 & = e_{DX}\circ(\varphi^{\pm}_X\otimes\id_{DX})\circ(f\otimes\id_{DX}) = e_X\circ c^{\pm1}\circ(f\otimes\id_{DX}) = e_X\circ(\id_{DX}\otimes f)\circ c^{\pm1},
  \end{align*}
  and thus $e_X\circ(\id_{DX}\otimes f)=0$ since $ c^{\pm1}$ is an isomorphism. The definition of $D$ on morphisms then forces $Df=0$. Since $D$ is an anti-equivalence and thus injective on morphisms, this means $f=0$ as required.

  To show $\varphi^{\pm}$ is surjective, suppose $f: D^2 X\rightarrow Y$ is a morphism in $\cC$ such that $f\circ\varphi^{\pm}_X=0$. We need to show that $f=0$. Since $D$ is an anti-equivalence, there exists an isomorphism $\psi: Y\rightarrow DZ$ for some object $Z\in\cC$, and also a morphism $g: Z\rightarrow DX$ such that $\psi\circ f=Dg$. In particular, $Dg\circ\varphi^{\pm}_X=0$, and thus
  \begin{align*}
      0  & = e_Z\circ(Dg\otimes\id_Z)\circ(\varphi^{\pm}_X\otimes\id_Z) = e_{DX}\circ(\id_{D^2 X}\otimes g)\circ(\varphi_X^{\pm}\otimes\id_Z)\nonumber\\
      &= e_{DX}\circ(\varphi_X^\pm\otimes\id_{DX})\circ(\id_X\otimes g)
       =e_X\circ c^{\pm1}\circ(\id_X\otimes g) =e_X\circ(g\otimes\id_X)\circ c^{\pm1}.
  \end{align*}
Since $ c^{\pm1}$ is an isomorphism, $e_X\circ(g\otimes\id_X)=0$, that is, $g$ is sent to $0$ under the natural isomorphism
\begin{equation*}
    \Hom_\cC(Z,DX) \longrightarrow\Hom_\cC(Z\otimes X,K).
\end{equation*}
Thus $g=0$, and then so is $Dg=\psi\circ f$, and then $f=0$ as required since $\psi$ is an isomorphism.
\end{proof}

In view of the results in this section, we have proved:
\begin{theorem}\label{thm:GVinRepA}
    Let $(\cC,K)$ be an abelian braided Grothendieck-Verdier category with bilinear tensor product bifunctor and dualizing functor $D$, and let $A$ be a commutative algebra in $\cC$. Then the category $\cC_A$ of left $A$-modules in $\cC$ is an abelian Grothendieck-Verdier category with dualizing object $(DA,\mu_{DA})$.
\end{theorem}


\subsection{Ribbon Grothendieck-Verdier categories from local \texorpdfstring{$A$}{A}-modules}

We now consider whether the category $\cC_A^{\loc}$ of local modules for a commutative algebra $A$ in a braided Grothendieck-Verdier category $\cC$ is a braided Grothendieck-Verdier category. The problem is that it is not clear whether the dualizing functor $D_A$ in the previous subsection preserves $\cC_A^{\loc}$ in general.

Recall from Section \ref{subsec:rig-mon-cats} that a twist on a braided monoidal category $\cC$ is a natural isomorphism $\theta: \cC\rightarrow\cC$ such that the balancing equation
\begin{equation*}
    \theta_{X\otimes Y}= c_{Y,X}\circ c_{X,Y}\circ(\theta_X\otimes\theta_Y)
\end{equation*}
holds for all $X,Y\in\cC$, and that $\theta_\vac=\id_\vac$ for any twist on $\cC$. The following is \cite[Definition 9.1]{boyarchenko2013duality}:
\begin{definition}
    A \textit{ribbon Grothendieck-Verdier category} is a braided Grothendieck-Verdier category $(\cC,K)$ with a twist $\theta$ such that $D\theta_X=\theta_{DX}$ for all $X\in\cC$, where $D$ is the dualizing functor of $(\cC,K)$.
\end{definition}


Our goal for this subsection is to prove the following theorem:
\begin{theorem}\label{thm:CAloc_ribbon_GV}
    Let $(\cC,K)$ be an abelian ribbon Grothendieck-Verdier category with twist $\theta$ and dualizing functor $D$, and let $A$ be a commutative algebra in $\cC$. If $\theta_A^2=\id_{A}$, then the category $\cC_A^{\loc}$ of local left $A$-modules in $\cC$ is a braided Grothendieck-Verdier category with dualizing object $(DA,\mu_{DA})$. Moreover, $\cC_A^\loc$ is a ribbon Grothendieck-Verdier category if $\theta_A=\id_A$.
\end{theorem}
\begin{proof}
    We need the dualizing functor $D_A:\cC_A\rightarrow\cC_A$ from the previous subsection to restrict to a dualizing functor on $\cC_A^{\loc}$. From the proof of Theorem \ref{thm:GVinRepA}, especially Proposition \ref{prop:phi_pm_A-hom}, it is enough to show that if $(M,\mu_M)\in\cC_A^\loc$, then both of $DM_\pm=(DM,\mu_{DM}^\pm)$ are objects of $\cC_A^\loc$. For this, it is enough to show that
    \begin{equation*}
e_M\circ\left((\mu_{DM}^\pm\circ c_{DM,A}\circ c_{A,DM})\otimes\id_M\right) = e_M\circ(\mu_{DM}^\pm\otimes\id_M),
    \end{equation*}
 which can be proved diagrammatically using the balancing equation, naturality of $\theta$, $\theta_{DM}=D\theta_M$, the assumption $\theta_A=\theta_A^{-1}$, and the assumption that $(M,\mu_M)\in\cC_A^\loc$:
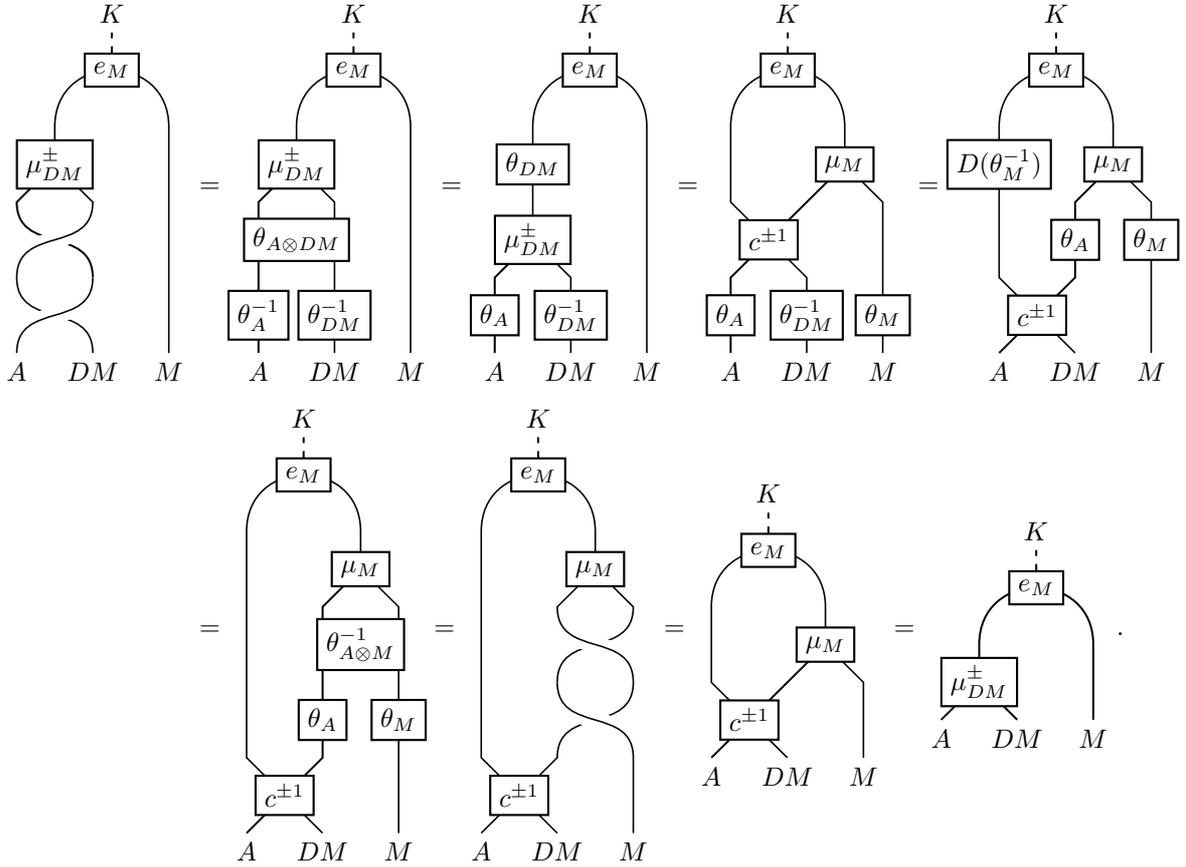
    \begin{align*}
\begin{tikzpicture}[scale = 1, baseline = {(current bounding box.center)}, line width=0.75pt]
\draw (1,1) .. controls (1,1.7) and (0,1.3) .. (0,2) -- (.5,2.5);
\draw[white, double=black, line width = 3pt ] (1,0) .. controls (1,.7) and (0,.3) .. (0,1) .. controls (0,1.7) and (1,1.3) .. (1,2) -- (.5,2.5) -- (.5,3) .. controls (.5,4) and (2,4) .. (2,3) -- (2,0);
\draw[white, double=black, line width = 3pt ] (0,0) .. controls (0,.7) and (1,.3) .. (1,1);
\draw[dashed] (1.25,3.75) -- (1.25, 4.25);
\node at (.5,2.5) [draw,thick, fill=white] {$\mu_{DM}^\pm$};
\node at (1.25,3.75) [draw,thick, fill=white] {$e_M$};
\node at (0,-.25) {$A$};
\node at (1,-.25) {$DM$};
\node at (2,-.25) {$M$};
\node at (1.25, 4.5) {$K$};
\end{tikzpicture}
& = \begin{tikzpicture}[scale = 1, baseline = {(current bounding box.center)}, line width=0.75pt]
\draw (0,0) -- (0,2) -- (.5,2.5);
\draw[white, double=black, line width = 3pt ] (1,0) -- (1,2) -- (.5,2.5) -- (.5,3) .. controls (.5,4) and (2,4) .. (2,3) -- (2,0);
\draw[dashed] (1.25,3.75) -- (1.25, 4.25);
\node at (.5,2.5) [draw,thick, fill=white] {$\mu_{DM}^\pm$};
\node at (0,.5) [draw,thick, fill=white] {$\theta_A^{-1}$};
\node at (1,.5) [draw,thick, fill=white] {$\theta_{DM}^{-1}$};
\node at (.5,1.5) [draw,thick, fill=white] {$\theta_{A\otimes DM}$};
\node at (1.25,3.75) [draw,thick, fill=white] {$e_M$};
\node at (0,-.25) {$A$};
\node at (1,-.25) {$DM$};
\node at (2,-.25) {$M$};
\node at (1.25, 4.5) {$K$};
\end{tikzpicture}
 = \begin{tikzpicture}[scale = 1, baseline = {(current bounding box.center)}, line width=0.75pt]
\draw (0,0) -- (0,1) -- (.5,1.5);
\draw[white, double=black, line width = 3pt ] (1,0) -- (1,1) -- (.5,1.5) -- (.5,3) .. controls (.5,4) and (2,4) .. (2,3) -- (2,0);
\draw[dashed] (1.25,3.75) -- (1.25, 4.25);
\node at (.5,1.5) [draw,thick, fill=white] {$\mu_{DM}^\pm$};
\node at (0,.5) [draw,thick, fill=white] {$\theta_A$};
\node at (1,.5) [draw,thick, fill=white] {$\theta_{DM}^{-1}$};
\node at (.5,2.5) [draw,thick, fill=white] {$\theta_{DM}$};
\node at (1.25,3.75) [draw,thick, fill=white] {$e_M$};
\node at (0,-.25) {$A$};
\node at (1,-.25) {$DM$};
\node at (2,-.25) {$M$};
\node at (1.25, 4.5) {$K$};
\end{tikzpicture}
 = \begin{tikzpicture}[scale = 1, baseline = {(current bounding box.center)}, line width=0.75pt]
\draw (0,0) -- (0,1) -- (1.5,2.5);
\draw[white, double=black, line width = 3pt ] (1,0) -- (1,1) -- (0,2) -- (0,3) .. controls (0,4) and (1.5,4) .. (1.5,3) -- (1.5,2.5) -- (2,2) -- (2,0);
\draw[dashed] (.75,3.75) -- (.75, 4.25);
\node at (1.5,2.5) [draw,thick, fill=white] {$\mu_{M}$};
\node at (0,.5) [draw,thick, fill=white] {$\theta_A$};
\node at (1,.5) [draw,thick, fill=white] {$\theta_{DM}^{-1}$};
\node at (2,.5) [draw,thick, fill=white] {$\theta_{M}$};
\node at (.5,1.5) [draw,thick, fill=white] {$ c^{\pm1}$};
\node at (.75,3.75) [draw,thick, fill=white] {$e_M$};
\node at (0,-.25) {$A$};
\node at (1,-.25) {$DM$};
\node at (2,-.25) {$M$};
\node at (.75, 4.5) {$K$};
\end{tikzpicture}
 = \begin{tikzpicture}[scale = 1, baseline = {(current bounding box.center)}, line width=0.75pt]
\draw (0,0) -- (1,1) -- (1,2) -- (1.5,2.5);
\draw[white, double=black, line width = 3pt ] (1,0) -- (0,1) -- (0,2) -- (0,3) .. controls (0,4) and (1.5,4) .. (1.5,3) -- (1.5,2.5) -- (2,2) -- (2,0);
\draw[dashed] (.75,3.75) -- (.75, 4.25);
\node at (1.5,2.5) [draw,thick, fill=white] {$\mu_{M}$};
\node at (1,1.5) [draw,thick, fill=white] {$\theta_A$};
\node at (0,2.5) [draw,thick, fill=white] {$D(\theta_{M}^{-1})$};
\node at (2,1.5) [draw,thick, fill=white] {$\theta_{M}$};
\node at (.5,.5) [draw,thick, fill=white] {$ c^{\pm1}$};
\node at (.75,3.75) [draw,thick, fill=white] {$e_M$};
\node at (0,-.25) {$A$};
\node at (1,-.25) {$DM$};
\node at (2,-.25) {$M$};
\node at (.75, 4.5) {$K$};
\end{tikzpicture}\nonumber\\
& = \begin{tikzpicture}[scale = 1, baseline = {(current bounding box.center)}, line width=0.75pt]
\draw (0,0) -- (1,1) -- (1,3) -- (1.5,3.5);
\draw[white, double=black, line width = 3pt ] (1,0) -- (0,1) -- (0,2) -- (0,4) .. controls (0,5) and (1.5,5) .. (1.5,4) -- (1.5,3.5) -- (2,3) -- (2,0);
\draw[dashed] (.75,4.75) -- (.75, 5.25);
\node at (1.5,3.5) [draw,thick, fill=white] {$\mu_{M}$};
\node at (1,1.5) [draw,thick, fill=white] {$\theta_A$};
\node at (1.5,2.5) [draw,thick, fill=white] {$\theta_{A\otimes M}^{-1}$};
\node at (2,1.5) [draw,thick, fill=white] {$\theta_{M}$};
\node at (.5,.5) [draw,thick, fill=white] {$ c^{\pm1}$};
\node at (.75,4.75) [draw,thick, fill=white] {$e_M$};
\node at (0,-.25) {$A$};
\node at (1,-.25) {$DM$};
\node at (2,-.25) {$M$};
\node at (.75, 5.5) {$K$};
\end{tikzpicture}
= \begin{tikzpicture}[scale = 1, baseline = {(current bounding box.center)}, line width=0.75pt]
\draw[white, double=black, line width = 3pt ] (1,0) -- (0,1) -- (0,2) -- (0,4) .. controls (0,5) and (1.5,5) .. (1.5,4) -- (1.5,3.5) -- (2,3) .. controls (2,2.3) and (1,2.7) .. (1,2);
\draw[white, double=black, line width = 3pt ] (0,0) -- (1,1) .. controls (1,1.7) and (2,1.3) .. (2,2) .. controls (2,2.7) and (1,2.3) .. (1,3) -- (1.5,3.5);
\draw[white, double=black, line width = 3pt ] (2,0) -- (2,1) .. controls (2,1.7) and (1,1.3) .. (1,2);
\draw[dashed] (.75,4.75) -- (.75, 5.25);
\node at (1.5,3.5) [draw,thick, fill=white] {$\mu_{M}$};
\node at (.5,.5) [draw,thick, fill=white] {$ c^{\pm1}$};
\node at (.75,4.75) [draw,thick, fill=white] {$e_M$};
\node at (0,-.25) {$A$};
\node at (1,-.25) {$DM$};
\node at (2,-.25) {$M$};
\node at (.75, 5.5) {$K$};
\end{tikzpicture}
= \begin{tikzpicture}[scale = 1, baseline = {(current bounding box.center)}, line width=0.75pt]
\draw[white, double=black, line width = 3pt ] (1,0) -- (0,1) -- (0,2) .. controls (0,3) and (1.5,3) .. (1.5,2) -- (1.5,1.5) -- (2,1) -- (2,0);
\draw (0,0) -- (1.5,1.5);
\draw[dashed] (.75,2.75) -- (.75, 3.25);
\node at (1.5,1.5) [draw,thick, fill=white] {$\mu_{M}$};
\node at (.5,.5) [draw,thick, fill=white] {$ c^{\pm1}$};
\node at (.75,2.75) [draw,thick, fill=white] {$e_M$};
\node at (0,-.25) {$A$};
\node at (1,-.25) {$DM$};
\node at (2,-.25) {$M$};
\node at (.75, 3.5) {$K$};
\end{tikzpicture}
= \begin{tikzpicture}[scale = 1, baseline = {(current bounding box.center)}, line width=0.75pt]
\draw (0,0) -- (.5,.5) -- (.5,1) .. controls (.5, 2) and (2,2) .. (2,1) -- (2,0);
\draw (1,0) -- (.5,.5);
\draw[dashed] (1.25,1.75) -- (1.25, 2.25);
\node at (.5,.5) [draw,thick, fill=white] {$\mu_{DM}^\pm$};
\node at (1.25,1.75) [draw,thick, fill=white] {$e_M$};
\node at (0,-.25) {$A$};
\node at (1,-.25) {$DM$};
\node at (2,-.25) {$M$};
\node at (1.25, 2.5) {$K$};
\end{tikzpicture} .
    \end{align*}
    This proves that $(\cC_A^\loc, (DA,\mu_{DA}^+))$ is a Grothendieck-Verdier category if $\theta_A^2=\id_A$. The dualizing functor is $D_A: (M,\mu_M)\mapsto (DM,\mu_M^+)), f\mapsto Df$.

    Now suppose that $\theta_A =\id_A$; we want to show that $\theta$ defines a twist on $\cC_A^\loc$ such that $D_A(\theta_M) =\theta_{D_A(M)}$ for all $(M,\mu_M)\in\cC_A^\loc$. First, if $(M,\mu_M)\in\cC_A^\loc$, then $\theta_M\in\Hom_{\cC_A^\loc}(M,M)$ because
    \begin{align*}
        \theta_M\circ\mu_M =\mu_M\circ\theta_{A\otimes M}=\mu_M\circ c_{M,A}\circ c_{A,M}\circ(\theta_A\otimes\theta_M) =\mu_M\circ(\id_A\otimes\theta_M),
    \end{align*}
    using the assumptions that $\theta_A=\id_A$ and $M$ is local. Thus $\theta$ restricts to a natural isomorphism $\id_{\cC_A^\loc}\rightarrow\id_{\cC_A^\loc}$. To show that $\theta$ satisfies the balancing equation for $\cC_A^\loc$, suppose $(M,\mu_M),(N,\mu_N)\in\cC_A^\loc$ and recall the coequalizer map $ \pi_{M,N}: M\otimes N\rightarrow M\otimes_A N$ which is used to define $\mu_{M\otimes_A N}$. Then
    \begin{align*}
        \theta_{M\otimes_A N}\circ  \pi_{M,N} =  \pi_{M,N}\circ\theta_{M\otimes N} = \pi_{M,N} \circ c_{N,M}\circ c_{M,N}\circ(\theta_M\otimes\theta_N) = c_{N,M}^A\circ c_{M,N}^A\circ(\theta_M\otimes_A\theta_N)\circ  \pi_{M,N},
    \end{align*}
    where $ c^A$ is the braiding on $\cC_A^\loc$. Since $ \pi_{M,N}$ is surjective, it follows that
    \begin{equation*}
        \theta_{M\otimes_A N}= c_{N,M}^A\circ c_{M,N}^A\circ(\theta_M\otimes_A\theta_N),
    \end{equation*}
    as required. Finally, $D_A(\theta_M)=\theta_{D_A(M)}$ because $D_A(\theta_M)=D\theta_M$ by Proposition \ref{prop:DA(f)=f'} and because $D\theta_M=\theta_{DM}$. Thus $\cC_A^\loc$ is a ribbon Grothendieck-Verdier category in this case.
\end{proof}


\subsection{Rigidity of r-categories}

Following \cite{boyarchenko2013duality}, we call a Grothendieck-Verdier category $(\cC,K)$ an \textit{r-category} if the dualizing object $K$ is isomorphic to the unit object $\one$. For example, if $\cC$ is a rigid monoidal category, then $\cC$ is an r-category with dualizing functor $D$ given by taking left duals, and quasi-inverse $D^{-1}$ given by taking right duals.

In particular, if $\cC$ is an abelian rigid braided monoidal category and $A$ is a commutative algebra in $\cC$, then by Theorem \ref{thm:GVinRepA}, the category $\cC_A$ of left $A$-modules in $\cC$ is a Grothendieck-Verdier category with dualizing object $(A^*,\mu_{A^*})$. Here, $\mu_{A^*}$ is defined by \eqref{eqn:mu-DM-def} with $e_A=\ev_A$; but $A^*$ also has the left $A$-action $\mu_{A^*}^r\circ c_{A,A^*}$ where $\mu_{A^*}^r$ is defined in Lemma \ref{lem:dual-action}. In fact, these two left $A$-actions are the same:
\begin{proposition}\label{prop:rcat-dual-and-rigid-dual-same}
    Let $A$ be a commutative algebra in an abelian rigid braided monoidal category $\cC$. Then for any $(M,\mu_M)\in\cC_A$, the maps $\mu_{M^*}: A\otimes M^*\rightarrow M^*$ of \eqref{eqn:mu-DM-def} (with $e_M=\ev_M$) and $\mu_{M^*}^r\circ c_{A,M^*}: A\otimes M^*\rightarrow M^*$ of Lemma \ref{lem:dual-action} are equal.
\end{proposition}
\begin{proof}
    By the universal property of $(M^*,\ev_M)$ in the r-category $\cC$, equivalently the isomorphism
    \begin{equation*}
        \Hom_\cC(A\otimes M^*,M^*)\cong\Hom_\cC(A\otimes M^*,\vac\otimes M^*)\cong\Hom_\cC(A\otimes M^*\otimes M,\vac)
    \end{equation*}
    in the rigid category $\cC$, it is enough to show
    \begin{equation*}
        \ev_M\circ((\mu_{M^*}^r\circ c_{A,M^*})\otimes\id_M) =\ev_M\circ(\mu_{M^*}\otimes\id_M).
    \end{equation*}
    Indeed, using the definitions and the rigidity of $M$,
    \begin{align*}
        \ev_M & \circ((\mu_{M^*}^r\circ c_{A,M^*})\otimes\id_M)\nonumber\\
        & = \ev_M\circ(\ev_M\otimes\id_{M^*\otimes M})\circ(\id_{M^*}\otimes\mu_M\otimes\id_{M^*\otimes M})\circ(\id_{M^*\otimes A}\otimes\coev_M\otimes\id_M)\circ(c_{A,M^*}\otimes\id_M)\nonumber\\
        & =\ev_M\circ(\id_{M^*}\otimes\mu_M)\circ(\id_{M^*\otimes A\otimes M}\otimes\ev_M)\circ(\id_{M^*\otimes A}\otimes\coev_M\otimes\id_M)\circ(c_{A,M^*}\otimes\id_M)\nonumber\\
        & =\ev_M\circ(\id_{M^*}\otimes\mu_M)\circ(c_{A,M^*}\otimes\id_M)\nonumber\\
        & =\ev_M\circ(\mu_{M^*}\otimes\id_M),
    \end{align*}
    as required.
\end{proof}

In the rest of this subsection, we study when a general r-category is rigid. First we have:
\begin{theorem}\label{thm:rigidityfrom simples}
If $\cC$ is a locally finite abelian r-category whose simple objects are rigid, then $\cC$ is rigid.
\end{theorem}

\begin{proof}
For $X,Y\in\cC$, let $\Phi_{X,Y}:DY\otimes DX\rightarrow D(X\otimes Y)$ be the unique morphism such that the diagram
\[\begin{tikzcd}[column sep=5pc]
	{DY\otimes DX\otimes X\otimes Y} & {DY\otimes Y} \\
	{D(X\otimes Y)\otimes X\otimes Y} & \one
	\arrow["{\id_{DY}\otimes e_X\otimes \id_Y}", from=1-1, to=1-2]
	\arrow["{e_Y}", from=1-2, to=2-2]
	\arrow["{\Phi_{X,Y}\otimes \id_{X\otimes Y}}"', from=1-1, to=2-1]
	\arrow["{e_{X\otimes Y}}"', from=2-1, to=2-2]
\end{tikzcd}\]
commutes. By \cite[Corollary 4.5]{boyarchenko2013duality}, $\Phi_{X,Y}$ is an isomorphism if $X$ is rigid, and by \cite[Corollary~4.6]{boyarchenko2013duality}, $\cC$ is rigid if and only if $\Phi_{X,Y}$ is an isomorphism for all $X,Y \in\cC$. Now by exactly the same argument as in \cite[Theorem~4.4.1]{Creutzig:2020qvs}, $\Phi_{X,Y}$ is an isomorphism for all finite-length objects of $\cC$ since it is an isomorphism when $X$ and $Y$ are simple.
\end{proof}

We want to apply this theorem to the category of modules for a commutative algebra in a Grothendieck-Verdier category. More precisely, we prove:
\begin{theorem}\label{thm:GVto rigid}
    Let $(\cC,K,\theta)$ be a locally finite abelian ribbon Grothendieck-Verdier category and let $A$ be a commutative algebra in $\cC$ such that $\theta_A^2=\id_A$. If the abelian braided monoidal category $\cC_A^\loc$ is rigid and every simple object of $\cC_A$ is an object of $\cC_A^\loc$, then $\cC_A$ is rigid.
\end{theorem}
\begin{proof}
First note that any $(M,\mu_M)\in\cC_A$ has finite length, by induction on the length of $M$ considered as an object of $\cC$. Indeed, if $M\neq 0$, then let $N\subseteq M$ be a proper $\cC_A$-submodule whose length (considered as an object of $\cC$) is maximal. Then $M/N$ is a simple object of $\cC_A$, and as an object of $\cC$, $N$ has length strictly less than $M$. Thus by induction $N$ has a finite composition series, and combining this series with $M$ yields a finite composition series for $M$. Thus $\cC_A$ is locally finite, and every simple object of $\cC_A$ is rigid by assumption. Now Theorem \ref{thm:rigidityfrom simples} will imply $\cC_A$ is rigid if we can show that $\cC_A$ has the structure of an r-category. By Theorem \ref{thm:GVinRepA}, $\cC_A$ is a Grothendieck-Verdier category with dualizing object $(DA,\mu_{DA})$, where $D$ is the dualizing functor of $(\cC,K)$. But $(DA,\mu_{DA})$ might not be isomorphic to $(A,\mu_A)$.

By Theorem \ref{thm:CAloc_ribbon_GV}, $(DA,\mu_{DA})$ is a dualizing object of $\cC_A^\loc$, and $(A,\mu_A)$ is also a dualizing object of $\cC_A^\loc$ since $\cC_A^\loc$ is rigid. Thus by \cite[Proposition 2.3(i)]{boyarchenko2013duality}, there is an invertible object $B\in\cC_A^\loc$ such that $DA\otimes_A B^{-1}\cong A$. Tensoring with $B$ on the right shows $DA\cong B$, that is,
$DA$ is invertible in $\cC_A^{\loc}$. As $DA$ is still invertible as an object of $\cC_A$, \cite[Proposition 2.3(i)]{boyarchenko2013duality} again shows $DA\otimes_A (DA)^{-1}\cong A$ is another dualizing object of $\cC_A$. So $\cC_A$ admits an r-category structure and is thus rigid by Theorem \ref{thm:rigidityfrom simples}. 
\end{proof}

In the setting of Theorem \ref{thm:GVto rigid}, we get rigidity of $\cC_A$ from rigidity of $\cC_A^{\loc}$. In the next subsection, under additional conditions, we will get rigidity of $\cC$ from rigidity of $\cC_A$, by embedding $\cC$ into a suitable rigid category related to $\cC_A$. This will require the following lemma:

\begin{lemma}\label{lem:GVto rigid}
Let $\cC$ be a Grothendieck-Verdier category with dualizing functor $D$, and let $F:\cC\rightarrow\cD$ be a fully faithful monoidal functor to a rigid category $\cD$. If $F(DX) \cong F(X)^*$ for all objects $X\in\cC$, then $\cC$ is rigid.
\end{lemma}
\begin{proof}
    Given $X\in\cC$, fix an isomorphism $\varphi_X: F(DX)\rightarrow F(X)^*$. Since $F$ is fully faithful, we can define $\ev_X: DX\otimes X\rightarrow\vac_\cC$ and $\coev_X: \vac_\cC\rightarrow X\otimes DX$ to be the unique morphisms such that $F(\ev_X)$ and $F(\coev_X)$ are given by the following compositions, respectively:
    \begin{align*}
        & F(DX\otimes X)\xrightarrow{F_2(DX,X)^{-1}} F(DX)\otimes F(X)\xrightarrow{\varphi_X\otimes\id_{F(X)}} F(X)^*\otimes F(X)\xrightarrow{\ev_{F(X)}} \vac_\cD\xrightarrow{F_0} F(\vac_\cC),\\
        & F(\vac_\cC)\xrightarrow{F_0^{-1}} \vac_\cD\xrightarrow{\coev_{F(X)}} F(X)\otimes F(X)^*\xrightarrow{\id_{F(X)}\otimes\varphi_X^{-1}} F(X)\otimes F(DX)\xrightarrow{F_2(X,DX)} F(X\otimes DX).
    \end{align*}
Then the commutative diagram
\begin{equation*}
\begin{tikzcd}[column sep = 5pc]
F(X) \ar[r, "l_{F(X)}^{-1})"] \ar[d, "F(l_X^{-1})"] & \vac_\cD\otimes F(X) \ar[d, "F_0\otimes\id_{F(X)}"] \ar[rd, "\coev_{F(X)}\otimes\id_{F(X)}"] & \\
F(\vac_\cC\otimes X)  \ar[d, "F(\coev_X\otimes\id_X)"] & \ar[l, "F_2({\vac_\cC,X})"'] F(\vac_\cC)\otimes F(X) \ar[d, "F(\coev_X)\otimes\id_{F(X)}"] & (F(X)\otimes F(X)^*)\otimes F(X) \ar[d, "(\id_{F(X)}\otimes\varphi_X^{-1})\otimes\id_{F(X)}"]\\
F((X\otimes DX)\otimes X) \ar[d, "F(\cA_{X,DX,X}^{-1})"] &  \ar[l, "F_2({X\otimes DX,X})"'] F(X\otimes DX)\otimes F(X)  & \ar[l, "F_2({X,DX})\otimes\id_{F(X)}"'] (F(X)\otimes F(DX))\otimes F(X) \ar[d, "\cA_{F(X),F(DX),F(X)}^{-1}"] \\
F(X\otimes(DX\otimes X)) \ar[d, "F(\id_X\otimes\ev_X)"] & \ar[l, "F_2({X,DX\otimes X})"']  F(X)\otimes F(DX\otimes X) \ar[d, "\id_{F(X)}\otimes F(\ev_X)"] & \ar[l, "\id_{F(X)}\otimes F_2({DX,X})"']  F(X)\otimes(F(DX)\otimes F(X)) \ar[d, "\id_{F(X)}\otimes(\varphi_X\otimes\id_{F(X)})"] \\
F(X\otimes\vac_\cC) \ar[d, "F(r_X)"] & \ar[l, "F_2({X,\vac_\cC)})"']  F(X)\otimes F(\vac_\cC) \ar[d, "\id_{F(X)}\otimes F_0^{-1}"] & F(X)\otimes(F(X)^*\otimes F(X)) \ar[ld, "\id_{F(X)}\otimes\ev_{F(X)}"] \\
F(X) & F(X)\otimes \vac_\cD \ar[l, "r_{F(X)}"'] & 
\end{tikzcd}
\end{equation*}
together with the left rigidity of $F(X)$ implies that
\begin{equation*}
    F((\id_X\otimes\ev_X)\circ(\coev_X\otimes\id_X)) =\id_{F(X)}=F(\id_X),
\end{equation*}
and thus $(\id_X\otimes\ev_X)\circ(\coev_X\otimes\id_X)=\id_X$ as $F$ is faithful. A similar commutative diagram shows that
\begin{align*}
    F((\ev_X\otimes\id_{DX})\circ(\id_{DX}\otimes\coev_X)) & = (\ev_{F(X)}\otimes\id_{F(DX)})\circ(\varphi_X\otimes\id_{F(X)}\otimes\varphi_X^{-1})\circ(\id_{F(DX)}\otimes\coev_{F(X)})\nonumber\\
    & = \varphi_X\circ(\ev_{F(X)}\otimes\id_{F(X)^*})\circ(\id_{F(X)^*}\otimes\coev_{F(X)})\circ\varphi_X^{-1} =\id_{F(DX)},
\end{align*}
so also $(\ev_X\otimes\id_{DX})\circ(\id_{DX}\otimes\coev_X)=\id_{DX}$ as $F$ is faithful. Thus $DX$ is a left dual of $X$ for any $X\in\cC$.

To show that $\cC$ also has right duals, let $D^{-1}$ be a quasi-inverse of $D$. Then for any $X\in\cC$,
\begin{equation*}
    F(D^{-1}X) \cong {}^*(F(D^{-1}X)^*)\cong{}^*F(DD^{-1}X)\cong {}^*F(X).
\end{equation*}
Now a similar argument to the above shows that $D^{-1}X$ is a right dual of $X$ in $\cC$. Thus $\cC$ is rigid.
\end{proof}

To help check the condition $F(DX)\cong F(X)^*$ in the previous lemma, we have the following theorem, which is a variant of \cite[Theorem 6.11]{McRae:2023ado} and has the same proof:
\begin{theorem}\label{thm:FcommuteswithD}
    Let $(\cC,K)$ and $(\cD,L)$ be abelian Grothendieck-Verdier categories and let $F:\cC\rightarrow\cD$ be a right exact monoidal functor such that $F(K)\cong L$. Assume also that for any simple object $X\in\cC$, the map $e_X: DX\otimes X\rightarrow K$ is surjective and any non-zero morphism from $F(DX)$ to $DF(X)$ is an isomorphism. Then $F(DX)\cong DF(X)$ for all finite-length objects $X\in\cC$.
\end{theorem}
\begin{proof}
Fix an isomorphism $\psi: F(K)\rightarrow L$. Then for any object $X\in\cC$, there is a unique morphism $\varphi_X: F(DX)\rightarrow DF(X)$ such that the diagram
\begin{equation*}
\begin{tikzcd}[column sep=4pc]
F(DX)\otimes F(X) \ar[r, "F_2({DX,X})"] \ar[d, "\varphi_X\otimes\Id_{F(X)}"] & F(DX\otimes X) \ar[d, "\psi\circ F(e_X)"] \\
DF(X)\otimes F(X) \ar[r, "e_{F(X)}"] & L
\end{tikzcd}
\end{equation*}
commutes. If $X$ is simple, then by hypothesis $\varphi_X$ is either $0$ or an isomorphism. If $\varphi_X=0$, then
\begin{equation*}
    \psi\circ F(e_X)\circ F_2(DX,X)=0
\end{equation*}
as well, and thus $F(e_X)=0$. But since $e_X$ is surjective by assumption and $F$ is right exact, $F(e_X)$ is also surjective and thus $L=0$. However, $L=0$ only if every object of $\cD$ is $0$, in which case the theorem is trivial. Thus we may assume $\varphi_X$ is an isomorphism for all simple objects $X\in\cC$.

 Next, we show that the morphisms $\varphi_X$ determine a natural transformation $\varphi: F\circ D\rightarrow D\circ F$, that is,
\[
\varphi_{X}\circ F(Df)=DF(f)\circ\varphi_{Y}
\]
for all morphisms $f: X\rightarrow Y$ in $\cC$. The universal property of $DF(X)$ implies it is enough to show
\begin{equation*}
e_{F(X)}\circ\left((\varphi_{X}\circ F(Df))\otimes\Id_{F(X)}\right)=e_{F(X)}\circ\left((DF(f)\circ\varphi_{Y})\otimes \Id_{F(X)}\right),
\end{equation*}
and indeed the definitions imply
\begin{align*}
e_{F(X)} & \circ\left((\varphi_{X}\circ F(Df))\otimes\Id_{F(X)}\right)  =\psi\circ F(e_{X})\circ F_2(DX,X)\circ(F(Df)\otimes\Id_{F(X)})\nonumber\\
& =\psi\circ F(e_{X}\circ(Df\otimes\Id_{X}))\circ F_2(DY,X) =\psi\circ F(e_{Y}\circ(\Id_{DY}\otimes f))\circ F_2(DY,X)\nonumber\\
& = \psi\circ F(e_{Y})\circ F_2(DY,Y)\circ(\Id_{FD(Y)}\otimes F(f)) =e_{F(Y)}\circ(\varphi_{Y}\otimes F(f))\nonumber\\
& =e_{F(X)}\circ \left((DF(f)\circ\varphi_{Y})\otimes\Id_{F(X)}\right),
\end{align*}
as desired.

So far, $\varphi_X$ is an isomorphism when $X$ has length $0$ or $1$. If the length of $X$ is greater than $1$, there is some exact sequence
\begin{equation*}
0\longrightarrow Y_1\xrightarrow{f_1} X\xrightarrow{f_2} Y_2\longrightarrow 0
\end{equation*}
such that the lengths of $Y_1$ and $Y_2$ are strictly less than the length of $X$. Thus we assume $\varphi_{Y_1}$ and $\varphi_{Y_2}$ are isomorphisms by induction on length. Using the naturality of $\varphi$, the right exactness of $F$, and the exactness of $D$ for $\cC$ and $\cD$, we then get a commutative diagram
\begin{equation*}
\begin{tikzcd}[column sep=2pc]
& F(DY_2) \ar[rr, "F(Df_2)"] \ar[d, "\varphi_{Y_2}"] && F(DX) \ar[rr, "F(Df_1)"] \ar[d, "\varphi_X"] && F(DY_1) \ar[r] \ar[d, "\varphi_{Y_1}"] & 0\\
0 \ar[r] & DF(Y_2) \ar[rr, "DF(f_2)"] && DF(W) \ar[rr, "DF(f_1)"] && DF(Y_1) & 
\end{tikzcd}
\end{equation*}
with exact rows. The short five lemma diagram chase shows that $\varphi_X$ is an isomorphism, completing the induction to show that $\varphi_X$ is an isomorphism for all finite-length $X$.
\end{proof}

\subsection{Main results on rigidity of \texorpdfstring{$\cC$}{C}}
Let $\cC$ be a braided r-category. Our goal in this subsection is to show that $\cC$ is rigid under certain conditions, by finding a fully faithful monoidal functor from $\cC$ to a rigid category and applying Lemma \ref{lem:GVto rigid}. If $A$ is a commutative algebra in $\cC$, then as previously we identify the category $\cC_A$ of left $A$-modules in $\cC$ with the monoidal subcategory $\cC_A^\sigma\subseteq {}_A\cC_A$, where $\sigma = c_{A,-}^{-1}$. If $\cC_A$ with this monoidal structure is rigid, then we may consider the induction functor $F_{A,\sigma}:\cC\rightarrow\cC_A$. But $F_{A,\sigma}$ is unlikely to be fully faithful, so we need to eliminate some morphisms in $\cC_A$ by imposing extra structure.

There are two ways we can impose extra structure. First, recall from Section \ref{subsec:algebras-and-modules} that $(A,\sigma)$ is a commutative algebra in the Drinfeld center $\cZ(\cC)$. Then we construct a braided monoidal functor $I: \cC\rightarrow\cZ(\cC)_{(A,\sigma)}^{\loc}$
as follows. First, we have the braided monoidal embedding $i_+:\cC \rightarrow \cZ(\cC)$ given by $i_+(X) = (X,c_{-,X})$ on objects and given by the identity on morphisms. Then, we have the monoidal induction functor $\widetilde{F}_{A,\sigma}:\cZ(\cC)\rightarrow \cZ(\cC)_{(A,\sigma)}$, which restricts to a braided monoidal functor $\cZ(\cC)^0\rightarrow\cZ(\cC)_{(A,\sigma)}^{\loc}$, where $\cZ(\cC)^0\subseteq\cZ(\cC)$ is the full subcategory of objects that double braid trivially with $(A,\sigma)$. From the definitions, the image of $i_+$ is contained in $\cZ(\cC)^0$, so we get a braided monoidal functor
\[ I : \cC \xrightarrow{i_+} \cZ(\cC) \xrightarrow{\widetilde{F}_{A,\sigma}} \cZ(\cC)_{(A,\sigma)}^{\loc}. \]
Specifically, this functor is defined by
\begin{equation}\label{eqn:I-on-objects}
    I(X) = ((A\otimes X, (\id_A\otimes c_{-,X})\circ(c_{A,-}^{-1}\otimes\id_X)), \mu_A\otimes\id_X)
\end{equation}
for objects $X\in\cC$, and $I(f)=\id_A\otimes f$ for morphisms $f\in\cC$.

Alternatively, we can consider the Drinfeld center $\cZ(\cC_A)$, where again we identify $\cC_A = \cC_A^\sigma\subseteq {}_A\cC_A$ as a monoidal category. Then induction $F_{A,\sigma}:\cC\rightarrow\cC_A$ is a central functor, that is, there is a braided monoidal functor $F_{A,\sigma}':\cC\rightarrow\cZ(\cC_A)$ such that $F_{A,\sigma}$ is the composition of $F_{A,\sigma}'$ with the forgetful functor $U: \cZ(\cC_A)\rightarrow\cC_A$; see \cite[\S 3.2]{davydov2013witt} for a brief discussion. Specifically, for any $X\in\cC$, we equip $F_{A,\sigma}(X)=(A\otimes X,\mu_A\otimes\id_X)$ with the half-braiding
\begin{equation*}
    \overline{\gamma}^X: -\otimes_A F_{A,\sigma}(X) \longrightarrow F_{A,\sigma}(X)\otimes_A -,
\end{equation*}
where for $(M,\mu_M)\in\cC_A$, $\overline{\gamma}^X_M$ is characterized by the commutative diagram
\begin{equation}\label{eqn:half-braiding-for-F'}
\begin{matrix}
\begin{tikzcd}[column sep=3pc]
M\otimes A\otimes X \ar[d, "\pi_{M,F_{A,\sigma(X)}}"] \ar[r, "c_{A,M}^{-1}\otimes\id_X"] & A\otimes M\otimes X \ar[r, "\id_A\otimes c_{M,X}"] & A\otimes X\otimes M \ar[d, "\pi_{F_{A,\sigma}(X),M}"] \\
    M\otimes_A F_{A,\sigma}(X) \ar[rr, "\overline{\gamma}_M^X"] & & F_{A,\sigma}(X)\otimes_A M
   \end{tikzcd}
    \end{matrix}
\end{equation}
Tedious but standard calculations show that $\overline{\gamma}^X$ is a well-defined half-braiding and induces a functor $F_{A,\sigma}': \cC\rightarrow\cZ(\cC_A)$ such that 
\begin{equation*}
    F_{A,\sigma}'(X) =((A\otimes X), \mu_A\otimes\id_X),\overline{\gamma}^X)
\end{equation*}
for objects $X\in\cC$, and $F_{A,\sigma}'(f) =\id_A\otimes f$ for morphisms $f\in\cC$. Further calculations show that the isomorphisms $F_0: F_{A,\sigma}(\vac)\rightarrow (A,\mu_A)$ and $F_2(X,Y): F_{A,\sigma}(X)\otimes_A F_{A,\sigma}(Y)\rightarrow F_{A,\sigma}(X\otimes Y)$, which are part of the data of the monoidal functor $F_{A,\sigma}$, are morphisms in $\cZ(\cC_A)$, and that
\begin{equation*}
    F_{A,\sigma}'(c_{X,Y})\circ F_2(X,Y) = F_2(Y,X)\circ\gamma^Y_{F_{A,\sigma}(X)},
\end{equation*}
and thus $F_{A,\sigma}'$ is a braided monoidal functor. So $F_{A,\sigma}$ is indeed a central functor since $F_{A,\sigma}=U\circ F_{A,\sigma}'$.

The functors $I:\cC\rightarrow\cZ(\cC)_{(A,\sigma)}^{\loc}$ and $F_{A,\sigma}':\cC\rightarrow\cZ(\cC_A)$ are related by the Schauenburg equivalence \cite[Corollary 4.5]{schauenburg2001monoidal}; see also a relative version in \cite[Theorem 5.15]{laugwitz2023constructing} and the discussion in \cite[\S 3]{CLR}. This is an isomorphism of categories $\cS:  \cZ(\cC)_{(A,\sigma)}^{\loc} \xrightarrow{\sim} \cZ(\cC_A)$ given on objects by
\begin{equation}\label{eqn:S-on-objects}
    \cS((M,\gamma),\mu_M) = ((M,\mu_M),\overline{\gamma}),
\end{equation}
where $\overline{\gamma}_N$ for $N\in\cC_A$ is characterized by the commutative diagram
\begin{equation*}
\begin{tikzcd}[column sep=3pc]
N\otimes M \ar[d, "\pi_{N,M}"] \ar[r, "\gamma_N"] & M\otimes N \ar[d, "\pi_{M,N}"]\\
N\otimes_A M \ar[r, "\overline{\gamma}_N"] & M\otimes_A N
\end{tikzcd}
\end{equation*}
The assumption that $M$ is a local $(A,\sigma)$-module is critical for showing that $\overline{\gamma}$ is well defined. The inverse functor is given on objects by
\begin{equation*}
    \cS^{-1}((M,\mu_M),\gamma) = ((M,\widetilde{\gamma}),\mu_M),
\end{equation*}
where the half-braiding $\widetilde{\gamma}$ is characterized by the commutative diagram
\begin{equation*}
  \begin{tikzcd}[column sep=3pc]
    X\otimes M \ar[r, "\widetilde{\gamma}_X"] \ar[d, "\Phi_{X,M}^l"] & M\otimes X\ar[d, "\Phi_{X,M}^r"]\\
    F_{A,\sigma}(X)\otimes_A M \ar[r, "\gamma_{F_{A,\sigma}(X)}"] & M\otimes_A F_{A,\sigma}(X) 
 \end{tikzcd}
\end{equation*}
Here $\Phi_{X,M}^l$ and $\Phi_{X,M}^r$ are isomorphisms given by the compositions
\begin{align*}
    & \Phi_{X,M}^l: X\otimes M \xrightarrow{\iota_A\otimes\id_{X\otimes M}} A\otimes X\otimes M\xrightarrow{\pi_{F_{A,\sigma}(X),M}} F_{A,\sigma}(X)\otimes_A M,\nonumber\\
    & \Phi_{X,M}^r: M\otimes X \xrightarrow{\id_{M}\otimes\iota_A\otimes\id_X} M\otimes A\otimes X\xrightarrow{\pi_{M,F_{A,\sigma}(X)}} M\otimes_A F_{A,\sigma}(X).
\end{align*}
On morphisms, both $\cS$ and $\cS^{-1}$ are given by the identity. One can check that the isomorphism $\cS$ identifies the braided monoidal structure on $\cZ(\cC)_{(A,\sigma)}^{\loc}$ with that on $\cZ(\cC_A)$, so $\cS$ and $\cS^{-1}$ are also isomorphisms of braided monoidal categories. We can now show:
\begin{lemma}\label{lem:I-F'-same}
    The functors $\cS\circ I$ and $F_{A,\sigma}'$ from $\cC$ to $\cZ(\cC_A)$ are equal.
\end{lemma}
\begin{proof}
    By \eqref{eqn:I-on-objects}, \eqref{eqn:S-on-objects}, and \eqref{eqn:half-braiding-for-F'},
    \begin{equation*}
        (\cS\circ I)(X) = \left((A\otimes X,\mu_A\otimes\id_X),\overline{(\id_A\otimes c_{-,X})\circ(c_{A,-}^{-1}\otimes\id_X)}\right) = (F_{A,\sigma}(X),\overline{\gamma}^X) = F_{A,\sigma}'(X)
    \end{equation*}
    for objects $X\in\cC$. Also, $\cS\circ I$ and $F_{A,\sigma}'$ are both given by $\id_A\otimes f$ on morphisms, so $\cS\circ I=F_{A,\sigma}'$.
\end{proof}

Now suppose $\cC_A$ is rigid, so that as discussed in Section \ref{subsec:rig-mon-cats}, $\cZ(\cC_A)$ is rigid as well.  Thus $\cZ(\cC)_{(A,\sigma)}^{\loc}$ is also rigid by the Schauenburg monoidal equivalence. Let $D$ be the dualizing functor of the r-category $\cC$. In view of Lemma \ref{lem:GVto rigid}, we would like $I(DX)\cong I(X)^*$ for all $X\in\cC$. This is probably too much to ask in general, but we will show that there is such an isomorphism if there is a suitable isomorphism $F_{A,\sigma}(DX)\cong F_{A,\sigma}(X)^*$ in $\cC_A$. Namely, given $X\in\cC$, recall from the proof of Theorem \ref{thm:FcommuteswithD} that there is a unique $\cC_A$-morphism $\varphi_X: F_{A,\sigma}(DX)\rightarrow F_{A,\sigma}(X)^*$ such that the diagram
\begin{equation}\label{eqn:phiX-def}
\begin{matrix}
  \begin{tikzcd}[column sep=4pc]
    F_{A,\sigma}(DX)\otimes_{A} F_{A,\sigma}(X) \ar[d, "\varphi_X\otimes_{A}\id_{F_{A,\sigma}(X)}"] \ar[r, "F_2({DX,X})"] & F_{A,\sigma}(DX\otimes X) \ar[d, "F_0\circ F_{A,\sigma}(e_X)"] \\
    F_{A,\sigma}(X)^*\otimes_{A} F_{A,\sigma}(X) \ar[r, "\ev_{F_{A,\sigma}(X)}"] & A
    \end{tikzcd}
    \end{matrix}
\end{equation}
commutes. 

\begin{lemma}\label{lem:central-functor-duality}
Let $A$ be a commutative algebra in a braided r-category $\cC$ with duality functor $D$ such that $\cC_A$ is rigid. If the $\cC_A$-morphism $\varphi_X: F_{A,\sigma}(DX)\rightarrow F_{A,\sigma}(X)^*$ defined in \eqref{eqn:phiX-def} is an isomorphism for some object $X\in\cC$, then  $I(DX)\cong I(X)^*$ in $\cZ(\cC)_{(A,\sigma)}^{\loc}$.
\end{lemma}
\begin{proof}
    Since the Schauenburg equivalence $\cS$ is monoidal and thus preserves duals \cite[Exercise 2.10.6]{etingof2016tensor}, Lemma \ref{lem:I-F'-same} implies that it is enough to show that $F_{A,\sigma}'(DX)\cong F_{A,\sigma}'(X)^*$. Indeed, since the morphisms $F_2(DX,X)$, $F_0$, and $F_{A,\sigma}(e_X)=F_{A,\sigma}'(e_X)$ in \eqref{eqn:phiX-def} are morphisms in $\cZ(\cC_A)$, there is a  unique $\cZ(\cC_A)$-morphism $\varphi_X': F_{A,\sigma}'(DX)\rightarrow F_{A,\sigma}'(X)^*$ such that the diagram
\begin{equation}\label{eqn:phiX'-def}
\begin{matrix}
 \begin{tikzcd}[column sep=4pc, row sep=2pc]
    F_{A,\sigma}'(DX)\otimes_{\cZ(\cC_A)} F_{A,\sigma}'(X) \ar[d,"\varphi_X'\otimes_{\cZ(\cC_A)}\id_{F_{A,\sigma}'(X)}"] \ar[r, "F_2({DX,X})"] & F_{A,\sigma}'(DX\otimes X) \ar[d, "F_0\circ F_{A,\sigma}'(e_X)"] \\
    F_{A,\sigma}'(X)^*\otimes_{\cZ(\cC_A)} F_{A,\sigma}'(X) \ar[r, "\ev_{F_{A,\sigma}'(X)}"] & A
  \end{tikzcd}
    \end{matrix}
\end{equation}
commutes. From the definitions of $F_{A,\sigma}'$, $\otimes_{\cZ(\cC_A)}$, and duals in $\cZ(\cC_A)$, as well as from the uniqueness of $\varphi_X$, applying the forgetful functor $U:\cZ(\cC_A)\rightarrow\cC_A$ to the diagram in \eqref{eqn:phiX'-def} yields the diagram in \eqref{eqn:phiX-def}. In particular, $U(\varphi_X')=\varphi_X$, and since $U$ reflects isomorphisms, this means that $\varphi_X'$ is an isomorphism in $\cZ(\cC_A)$ whenever $\varphi_X$ is an isomorphism in $\cC_A$.
\end{proof}

To apply Lemma \ref{lem:GVto rigid} to the functor $I: \cC\rightarrow\cZ(\cC)_{(A,\sigma)}^{\loc}$, it remains to show that $I$ is fully faithful under suitable conditions. We first show that $I$ is faithful under very mild conditions:
\begin{proposition}\label{prop:I-faithful}
    Let $\cC$ be a braided r-category with duality functor $D$, and let $A$ be a commutative algebra in $\cC$ such that the unit morphism $\iota_A:\vac\rightarrow A$ is injective. Then $I: \cC\rightarrow\cZ(\cC)_{(A,\sigma)}^{\loc}$ is faithful.
\end{proposition}
\begin{proof}
    Let $f_1, f_2: X\rightarrow Y$ be two morphisms such that $I(f_1)=I(f_2)$; we need to show that $f_1=f_2$. From the definition of $I$ on morphisms, $I(f_1)=I(f_2)$ is equivalent to $\id_A\otimes f_1=\id_A\otimes f_2$ as morphisms in $\cC$, and thus also $f_1\otimes\id_A=f_2\otimes\id_A$ since $\cC$ is braided. This implies that the compositions
    \begin{equation*}
        DY\otimes X \xrightarrow{\id_{DY\otimes X}\otimes\iota_A} DY\otimes X\otimes A\xrightarrow{\id_{DY}\otimes f_i\otimes\id_A} DY\otimes Y\otimes A \xrightarrow{e_Y\otimes\id_A} A
    \end{equation*}
    are equal for $i=1,2$. By properties of unit isomorphisms, this implies
    \begin{equation*}
        \iota_A\circ e_Y\circ(\id_{DY}\otimes f_1) =\iota_A\circ e_Y\circ(\id_{DY}\otimes f_2),
    \end{equation*}
    and then 
    \begin{equation*}
        e_Y\circ(\id_{DY}\otimes f_1) = e_Y\circ(\id_{DY}\otimes f_2)
    \end{equation*}
    since $\iota_A$ is injective. By the definition of $D$ on morphisms, this implies $Df_1=Df_2$, and then $f_1=f_2$ since $D$ is an anti-equivalence.
\end{proof}

Now under further conditions, we can prove that $I$ is fully faithful:
\begin{theorem}\label{thm:faithfulness-of-I}
Let $\cC$ be a locally finite $\mathbb{K}$-linear abelian braided r-category with duality functor $D$, and let $A$ be a commutative algebra in $\cC$ such that the unit morphism $\iota_A:\vac\rightarrow A$ is injective. Also assume that the following three conditions hold:
\begin{enumerate}
    \item[(a)] The monoidal category $\cC_A$ is rigid (where we identify $\cC_A=\cC_A^\sigma\subseteq {}_A\cC_A$ with $\sigma = c_{A,-}^{-1}$).
    
    \item[(b)] The map $\varphi_X: F_{A,\sigma}(DX) \rightarrow F_{A,\sigma}(X)^*$ in \eqref{eqn:phiX-def} is an isomorphism in $\cC_A$ for any object $X\in\cC$.
    \item[(c)] For any $\cC$-subobject $s: S\hookrightarrow A$ such that $S$ is not contained in $\mathrm{Im}\,\iota_A$ (that is, $s\neq \iota_A\circ f$ for any $f\in\mathrm{Hom}_\cC(S,\vac)$), there exists an object $Z\in\cC$ such that $c_{Z,S}\neq c_{S,Z}^{-1}$ and
        the map $s\otimes\id_Z: S\otimes Z\rightarrow A\otimes Z$ is injective.
\end{enumerate}
Then the functor $I:\cC \rightarrow \cZ(\cC)_{(A,\sigma)}^{\loc}$ is fully faithful. 
\end{theorem}
\begin{proof}
Since Proposition \ref{prop:I-faithful} implies that $I$ is faithful, it suffices to show that $\Hom_{\cZ(\cC)_{(A,\sigma)}^{\loc}}(I(X),I(Y))$ and $\Hom_\cC(X,Y)$ are isomorphic (finite-dimensional) $\mathbb{K}$-vector spaces for all objects $X,Y\in\cC$. Indeed,
\begin{equation}\label{eq:homs}
    \begin{split}
        \Hom_{\cZ(\cC)_{(A, \sigma)}^{\loc}}\left(I(X), I(Y)\right) 
        &\cong  \Hom_{\cZ(\cC)_{(A, \sigma)}^{\loc}}\left(I(Y)^*\otimes I(X), (A, \sigma)\right) \\
        &\cong \Hom_{\cZ(\cC)_{(A, \sigma)}^{\loc}}\left(I(DY)\otimes I(X), (A, \sigma)\right) \\
        & \cong  \Hom_{\cZ(\cC)_{(A, \sigma)}^{\loc}}\left(I(DY\otimes X), (A, \sigma)\right) \\
        & \cong \Hom_{\cZ(\cC)_{(A, \sigma)}^{\text{loc}}}(\widetilde{F}_{A,\sigma}(i_+(DY\otimes X)), (A, \sigma))\\
        &\cong \Hom_{\cZ(\cC)}\left(i_+(DY\otimes X), (A, \sigma)\right)\\
        &\cong \Hom_{\cZ(\cC)}((DY\otimes X, c_{ - ,DY\otimes X)}, (A, c_{A,  - }^{-1}))\\
        &\hookrightarrow \Hom_\cC(DY\otimes X,A),
    \end{split}
\end{equation}
where the second isomorphism uses Lemma \ref{lem:central-functor-duality}, the third uses the fact that $I$ is monoidal, and the fifth uses the adjunction between the induction functor $\widetilde{F}_{A,\sigma}$ and the forgetful functor.
On the other hand,
\begin{align}\label{eq:homs2}
    \Hom_\cC(X,Y) \cong \Hom_\cC(DY,DX)\cong \Hom_\cC(DY\otimes X,\vac)\hookrightarrow\Hom_\cC(DY\otimes X,A),
\end{align}
where the last map, given by $f\mapsto\iota_A\circ f$ for $f\in\Hom_\cC(DY\otimes X,\vac)$, is injective because $\iota_A$ is injective. Thus it suffices to show that a $\cC$-morphism $f: DY\otimes X\rightarrow A$ commutes with the half-braidings $c_{-,DY\otimes X}$ and $c_{A,-}^{-1}$ if and only if $\mathrm{Im}\,f\subseteq\mathrm{Im}\,\iota_A$.

Writing $W$ for $DY\otimes X$, a $\cC$-morphism $f: W\rightarrow A$ is a morphism in $\Hom_{\cZ(\cC)}((W, c_{ - , W}), (A, c_{A,  - }^{-1}))$ if and only if the diagram
\begin{equation*}
\begin{tikzcd}[column sep=3pc]
Z \otimes W  \ar[d, "c_{Z, W}"] \ar[r, "\id_Z \otimes f"] &  Z \otimes A \ar[d, "c_{A, Z}^{-1}"] \\
W \otimes Z \ar[r, "f \otimes \id_Z"]  & A \otimes Z
\end{tikzcd}
\end{equation*}
commutes for any object $Z\in\cC$. As $\cC$ is abelian, we can factor $f$ as $f: W \stackrel{\pi}{\twoheadrightarrow} S \xhookrightarrow{s} A$ for some surjection $\pi$ and some injection $s$, where $S$ is the image of $f$. By naturality of the braiding we get
\begin{align*}
    (f\otimes\id_Z)\circ c_{Z,W} & = (s\otimes\id_Z)\circ c_{Z,S}\circ(\id_Z\otimes\pi),\nonumber\\
    c_{A,Z}^{-1}\circ(\id_Z\otimes f) & = (s\otimes\id_Z)\circ c_{S,Z}^{-1}\circ(\id_Z\otimes\pi).
\end{align*}
Since the tensor product of an abelian Grothendieck-Verdier category is right exact, $\id_Z \otimes \pi$ is still a surjection. Hence $f$ is a morphism in $\cZ(\cC)$ if and only if
\begin{equation}\label{eqn:morph-in-Z(C)-cond}
    (s\otimes\id_Z)\circ c_{Z,S} = (s\otimes\id_Z)\circ c_{S,Z}^{-1}
\end{equation}
for all objects $Z\in\cC$.

If $S$ is not contained in $\mathrm{Im}\,\iota_A$, then by condition (c) there is an object $Z\in\cC$ such that $c_{S,Z}\neq c_{Z,S}^{-1}$ and $s\otimes\id_Z$ is injective. Thus \eqref{eqn:morph-in-Z(C)-cond} does not hold and $f\notin\Hom_{\cZ(\cC)}((W, c_{ - , W}), (A, c_{A,  - }^{-1}))$ in this case. If $S$ is contained in $\mathrm{Im}\,\iota_A$ on the other hand, then $s=\iota_A\circ g$ for some $g\in\Hom_\cC(S,\vac)$ and \eqref{eqn:morph-in-Z(C)-cond} holds due to the commutative diagram
\begin{equation*}
 \begin{tikzcd}[column sep=3pc, row sep=2pc]
      Z\otimes S \ar[d, "c_{Z,S}"] \ar[r, "\id_Z\otimes g"] & Z\otimes\vac \ar[d, "c_{Z,\vac}=c_{\vac,Z}^{-1}"] \ar[r, "\id_Z\otimes\iota_A"] & Z\otimes A \ar[d, "c_{A,Z}^{-1}"] \\
S\otimes Z \ar[r, "g\otimes\id_Z"] & \vac\otimes Z \ar[r, "\iota_A\otimes\id_Z"] & A\otimes Z
 \end{tikzcd}
\end{equation*}
Thus $f\in\Hom_{\cZ(\cC)}((W, c_{ - , W}), (A, c_{A,  - }^{-1}))$ in this case, showing that the compositions \eqref{eq:homs} and \eqref{eq:homs2} have the same image in $\Hom_\cC(DY\otimes X,A)$. This completes the proof that $I$ is fully faithful.
\end{proof}

Now we prove the main result of this section. As previously, $\cC_A$ denotes the monoidal category $\cC_A^{\sigma}$ with $\sigma = c_{-,A}^{-1}$. Since $\sigma$ is understood, we write simply $F_A$ for the induction functor $F_{A,\sigma}$.
\begin{theorem}\label{rigidity-of-C-from-CAloc}
Let $(\cC,D,\theta)$ be a locally finite $\KK$-linear abelian ribbon $r$-category, and let $A$ be a commutative algebra in $\cC$ such that $\theta_A^2=\id_A$ and the unit morphism $\iota_A:\vac\rightarrow A$ is injective. Also assume that the following conditions hold:
\begin{enumerate}
\item[(a)] $\cC_A^{\loc}$ is rigid and every simple object of $\cC_A$ is an object of $\cC_A^{\loc}$.
\item[(b)] For any simple object $X\in\cC$, the map $e_X: DX\otimes X\rightarrow\vac$ is surjective and any non-zero $\cC_A$-morphism from $F_A(DX)$ to $F_A(X)^*$ is an isomorphism.
\item[(c)] For any $\cC$-subobject $s: S\hookrightarrow A$ such that $S$ is not contained in $\mathrm{Im}\,\iota_A$, there exists an object $Z\in\cC$ such that $c_{Z,S}\neq c_{S,Z}^{-1}$ and the map $s\otimes\id_Z: S\otimes Z\rightarrow A\otimes Z$ is injective.
\end{enumerate}
Then $\cC$ is rigid.
\end{theorem}
\begin{proof}
As $\theta_A^2=\id_A$, condition (a) and Theorem \ref{thm:GVto rigid} imply that $\cC_A$ is rigid. Next, condition (b) and Theorem \ref{thm:FcommuteswithD} imply that the map $\varphi_X: F_A(DX)\rightarrow F_A(X)^*$ in \eqref{eqn:phiX-def} is a $\cC_A$-isomorphism for all objects $X\in\cC$ (note that $F_A$ is right exact because it has a right adjoint, namely, the forgetful functor $\cC_A\rightarrow\cC$). Theorem \ref{thm:faithfulness-of-I} now implies that $I: \cC\rightarrow\cZ(\cC)_{(A,\sigma)}^{\loc}$ is fully faithful, and $\cZ(\cC)_{(A,\sigma)}^{\loc}\cong\cZ(\cC_A)$ is rigid because $\cC_A$ is. Thus $\cC$ is rigid by Lemma~\ref{lem:GVto rigid}. 
\end{proof}


\section{Consequences for vertex operator algebra extensions}\label{sec:VOAs}

In this section, $\KK =\CC$. We apply the results of the preceding sections to braided monoidal categories of modules for vertex operator algebra (VOAs) and their extensions.

\subsection{Contragredient modules in vertex tensor categories}

Let $V$ be a  $\ZZ$-graded vertex operator algebra \cite{FLM, FHL, LL}. In particular, $V=\bigoplus_{n\in\ZZ} V_{(n)}$ is a $\ZZ$-graded vector space graded by conformal weights, there is a vertex operator map $Y_V(-,x): V\rightarrow \End(V)[[x,x^{-1}]]$, a vacuum vector $\vac\in V_{(0)}$ such that $Y_V(\vac,x) =\id_V$, and a conformal vector $\omega\in V_{(2)}$ such that $Y(\omega,x)=\sum_{n\in\ZZ} L_n\,x^{-n-2}$, where the operators $L_n$ span a representation of the Virasoro Lie algebra on $V$. The conformal weight $\ZZ$-grading of $V$ is given by eigenvalues for the Virasoro zero-mode $L_0$.

Let $W$ be a lower-bounded generalized $V$-module. In particular, $W=\bigoplus_{h\in\CC} W_{[h]}$ is a graded vector space and there is a vertex operator action $Y_W(-,x): V\rightarrow \End(W)[[x,x^{-1}]]$ such that $Y_W(\vac,x)=\id_W$ and $Y_W(\omega,x)=\sum_{n\in\ZZ} L_n\,x^{-n-2}$, where the operators $L_n$ again span a representation of the Virasoro algebra on $W$. Each homogeneous subspace $W_{[h]}$ is the generalized $L_0$-eigenspace with generalized eigenvalue $h\in\CC$, and $W_{[h]}=0$ for $\mathrm{Re}\,h$ sufficiently negative. As in \cite{FHL}, the \textit{contragredient} of $W$ is the unique lower-bounded generalized $V$-module structure on the graded dual vector space $W'=\bigoplus_{h\in\CC} W_{[h]}^*$ such that
\begin{equation}\label{eqn:VOA_contra_structure}
    \langle Y_{W'}(v,x)w', w\rangle =\langle w', Y_W(e^{xL_1}(-x^{-2})^{L_0}v,x^{-1})w\rangle
\end{equation}
for $v\in V$, $w'\in W'$, $w\in W$. A lower-bounded generalized $V$-module is \textit{grading restricted} if $\dim W_{[h]}<\infty$ for all $h\in\CC$. For brevity, we refer to grading-restricted generalized $V$-modules simply as $V$-modules.

Now let $\cC$ be a vertex tensor category of $V$-modules in the sense of \cite{HLZ1, HLZ2, HLZ3, HLZ4, HLZ5, HLZ6, HLZ7, HLZ8}, so that in particular $\cC$ has a braided monoidal category structure specified in \cite{HLZ8}; see also the exposition in \cite[\S 3.3]{creutzig2017tensor}. Let $A$ be another VOA, either $\ZZ$-graded or $\frac{1}{2}\ZZ$-graded, that contains $V$ as a vertex operator subalgebra; in particular, the conformal vectors of $V$ and $A$ coincide. If $A$ is an object of $\cC$ and $\Hom_{\cC}(V, A)\cong \CC$,
then $A$ has the structure of a commutative haploid algebra in $\cC$ \cite{Huang:2014ixa}. Moreover, the category of $A$-modules that lie in $\cC$ is a vertex tensor category and is isomorphic to $\cC_A^{\loc}$ as a braided monoidal category \cite{creutzig2017tensor}. We also have the larger monoidal category $\cC_A$ of left modules for $A$ considered as a commutative algebra in $\cC$. Following vertex operator notation, we use $\boxtimes$ rather than $\otimes$ to denote tensor product bifunctors, to distinguish them from vector space tensor products. In particular, we write $\boxtimes_V$ for the tensor product in $\cC$, and we write $\boxtimes_A$ for the tensor product of $A$-modules in $\cC_A$.

For $V$-modules $W_1$, $W_2\in\cC$, there is a (possibly logarithmic) tensor product intertwining operator $\cY_\boxtimes: W_1\otimes W_2\rightarrow (W_1\boxtimes_V W_2)[\log x]\lbrace x\rbrace$ \cite{FHL, HLZ2} such that the pair $(W_1\boxtimes_V W_2,\cY_\boxtimes)$ satisfies a universal property \cite{HLZ3}: For any intertwining operator $\cY: W_1\otimes W_2\rightarrow W_3[\log x]\lbrace x\rbrace$ such that $W_3$ is an object of $\cC$, there is a unique $V$-module homomorphism $f: W_1\boxtimes_V W_2\rightarrow W_3$ such that $f\circ\cY_\boxtimes =\cY$. In particular, the algebra multiplication $\mu_A: A\boxtimes_V A\rightarrow A$ is the unique $V$-module homomorphism such that $\mu_A\circ\cY_\boxtimes = Y_A$. Moreover, for any object $(M,\mu_M)\in\cC_A$, we get an intertwining operator $Y_M: A\otimes M\rightarrow M[\log x]\lbrace x\rbrace$ such that $\mu_M\circ\cY_\boxtimes = Y_M$.

If $\cC$ is closed under contragredient modules, then $\cC$ is a ribbon Grothendieck-Verdier category with dualizing object $V'$, dualizing functor $D=(-)'$, and ribbon twist $\theta = e^{2\pi i L_0}$ \cite{ALSW}. Thus by Theorem \ref{thm:GVinRepA}, the category $\cC_A$ of non-local $A$-modules in $\cC$ is also a Grothendieck-Verdier category. Moreover, since $A$ is at worst $\frac{1}{2}\ZZ$-graded by $L_0$-eigenvalues, by Theorem \ref{thm:CAloc_ribbon_GV}, $\cC_A^{\loc}$ is a braided Grothendieck-Verdier category with dualizing object $(A',\mu_{A'})$, where $\mu_{A'}$ is defined by \eqref{eqn:mu-DM-def}. But since $\cC_A^{\loc}$ is also the category of $A$-modules (where $A$ is considered as a VOA) which are objects of $\cC$ (when considered as $V$-modules), $\cC_A^{\loc}$ already has a Grothendieck-Verdier category structure by \cite{ALSW}, with dualizing object $(A',Y_{A'})$, where $Y_{A'}$ is defined by \eqref{eqn:VOA_contra_structure}. The main goal of this subsection is to show that $Y_{A'}=\mu_{A'}\circ\cY_\boxtimes$, so that these two Grothendieck-Verdier category structures on $\cC_A^{\loc}$ are the same, and further to describe the Grothendieck-Verdier category structure on the non-local module category $\cC_A$ in vertex algebraic terms:
\begin{theorem}\label{thm:VOA_and_tens_cat_contras}
Let $V$ be a VOA, let $\cC$ be a vertex tensor category of $V$-modules which is closed under contragredients, and let $A$ be a $\frac{1}{2}\ZZ$-graded VOA which contains $V$ as a vertex operator subalgebra and which is an object of $\cC$ when considered as a $V$-module. For any $(M,\mu_M)\in\cC_A$, set $Y_{M'}=\mu_{M'}^-\circ\cY_\boxtimes$, where $\mu_{M'}^-: A\boxtimes_V M'\rightarrow M'$ is the left $A$-action of \eqref{eqn:muDM-_def}. Then
\begin{equation}\label{eqn:YM'_char}
    \langle Y_{M'}(a,x)m', m\rangle = \langle m', Y_M(e^{x L_1}(e^{\pi i} x^{-2})^{L_0} a,x^{-1})m\rangle
\end{equation}
for $a\in A$, $m'\in M'$, $m\in M$.
\end{theorem}

\begin{proof}
    The left $A$-action $\mu_{M'}^-: A\boxtimes_V M' \rightarrow M'$ of \eqref{eqn:muDM-_def} is characterized by
    \begin{equation}\label{eqn:muM'_def}
        e_M\circ(\mu_{M'}\boxtimes_V\id_M)\circ(c_{M',A}\boxtimes_V\id_M) = e_M\circ(\id_{M'}\boxtimes_V\mu_M)\circ\cA_{M',A,M}^{-1}.
    \end{equation}
    The braiding isomorphism $c_{M',A}$ and associativity isomorphism $\cA_{M',A,M}$ are defined in \cite{HLZ8} (see also \cite[\S 3.3]{creutzig2017tensor}). The map $e_M: M'\boxtimes_V M\rightarrow V'$ is defined in \cite[\S 3.2]{McRae2021} in the special case $V\cong V'$; more generally, we define an intertwining operator $\cE_M = \Omega_0(A_0(\Omega_0(Y_M))): M'\otimes M\rightarrow V'[\log x]\lbrace x\rbrace$ using the skew-symmetry $\Omega_0$ and adjoint $A_0$ intertwining operator constructions of \cite{FHL, HLZ2}. As in \cite{McRae2021}, a calculation shows that 
    \begin{equation}\label{eqn:EM_def}
        \langle\cE_M(m',x)m, v\rangle =\langle e^{-x^{-1}L_1} m', Y_M(e^{xL_1}v,x^{-1})e^{-xL_1} e^{-\pi i L_0}x^{-2L_0} m\rangle
    \end{equation}
for $m'\in M'$, $m\in M$, $v\in V$. Then $e_M$ is the unique $V$-module homomorphism such that $e_M\circ\cY_\boxtimes =\cE_M$.

Now \eqref{eqn:muM'_def} together with the definitions of the commutativity and associativity isomorphisms in $\cC$ imply that for any $a\in A$, $m'\in M'$, $m\in M$, and $v\in V$,
\begin{align}\label{eqn:associativity}
    \langle \cE_M(e^{x_0 L_{-1}} Y_{M'}(a,e^{\pi i}x_0)m', x_2)m, v\rangle =\langle \cE_M(m',x_1)Y_M(a,x_2)m, v\rangle.
\end{align}
Here $x_1$, $x_2$, and $x_0=x_1-x_2$ are any real numbers such that $x_1>x_2>x_1-x_2>0$ and we substitute powers of the formal variable $x$ in an intertwining operator with powers of the real number $x_i$, $i=0,1,2$, using the real-valued branch of logarithm $\ln x_i$. Taking $v=\vac$ in \eqref{eqn:associativity} and applying \eqref{eqn:EM_def}, we get
\begin{align}\label{eqn:YM'_char_1}
\langle e^{-x_2^{-1}L_1} e^{x_0 L_{-1}}  Y_{M'}(a,e^{\pi i} & x_0)m', e^{-x_2 L_1} e^{-\pi i L_0}x_2^{-2L_0} m\rangle\nonumber\\
&=\langle e^{-x_1^{-1}L_1} m', e^{-x_1L_1} e^{-\pi i L_0}x_1^{-2L_0} Y_M(a,x_2)m\rangle.
\end{align}
The left side is a series in powers of $x_0$ and $x_2$ which converges absolutely when $x_2>x_0>0$, and the right side is a series in powers of $x_1$ and $x_2$ which converges absolutely when $x_1>x_2>0$.

We analyze the right side of \eqref{eqn:YM'_char_1} first, using the $L_0$- and $L_1$-conjugation formulas from \cite[Proposition 3.36]{HLZ2} and the commutator formula for exponentials of $L_0$ and $L_1$ in \cite[Remark 3.38]{HLZ2}:
\begin{align}\label{eqn:YM'_char_2}
 &   \langle e^{-x_1^{-1}L_1} m', e^{-x_1L_1}  e^{-\pi i L_0}x_1^{-2L_0} Y_M(a,x_2)m\rangle\nonumber\\
    &= \langle e^{x_1^{-1}L_1}e^{-\pi i L_0}m', e^{x_1L_1} Y_M(x_1^{-2L_0} a, x_1^{-2} x_2)x_1^{-2L_0} m\rangle\nonumber\\
    & = \left\langle e^{x_1^{-1}L_1}e^{-\pi i L_0}m', Y_M\left(e^{x_1(1-x_1^{-1}x_2)L_1} (1-x_1^{-1}x_2)^{-2L_0} x_1^{-2L_0}a,\frac{x_1^{-2}x_2}{1-x_1^{-1}x_2}\right)e^{x_1L_1}x_1^{-2L_0}m\right\rangle\nonumber\\
    & =\left\langle e^{x_1^{-1}L_1}e^{-\pi i L_0}m', Y_M\left(e^{(x_1-x_2)L_1}(x_1-x_2)^{-2L_0} a,\frac{x_1^{-1}x_2}{x_1-x_2}\right)e^{x_1L_1}x_1^{-2L_0}m\right\rangle\nonumber\\
    & =\left\langle(x_0x_1)^{-L_0}e^{x_1^{-1}L_1}e^{-\pi i L_0}m', Y_M\left((x_0x_1)^{L_0}e^{x_0 L_1}x_0^{-2L_0} a, x_2\right))(x_0x_1)^{L_0}e^{x_1L_1}x_1^{-2L_0}m\right\rangle\nonumber\\
    & =\left\langle e^{x_0 L_1}(e^{\pi i}x_0x_1)^{-L_0}m', Y_M\left(e^{x_1^{-1}L_1}(x_0^{-1}x_1)^{L_0} a,x_2\right)e^{x_0^{-1}L_1}(x_0x_1^{-1})^{L_0} m\right\rangle\nonumber\\
    & =\sum_{h\in\CC}\sum_{k\geq 0}\langle  e^{x_0 L_1}(e^{\pi i}x_0x_1)^{-L_0}m', [e^{x_1^{-1}L_1}(x_0^{-1}x_1)^{L_0} a]_{h;k} e^{x_0^{-1}L_1}(x_0x_1^{-1})^{L_0} m\rangle x_2^{-h-1}(\ln x_2)^k,
\end{align}
where we use typical notation for the coefficients of powers of $x_2$ in $Y_M(-,x_2)$ in the last step. Due to the lower truncation property of intertwining operators, there are finitely many non-zero summands on the right side of \eqref{eqn:YM'_char_2}, and each summand is an analytic function in $x_1$ and $x_2$, since $x_0=x_1-x_2$. Moreover, we get the series expansion on the right side of \eqref{eqn:YM'_char_1} by taking the series expansions in $x_1$ and $x_2$ of each of the finitely many non-zero summands in \eqref{eqn:YM'_char_2} and adding them up. Specifically, we expand $\ln x_0$ and powers of $x_0$ as series in non-negative integer powers of $x_2$ because $x_1>x_2$:
\begin{equation*}
    \ln(x_1-x_2)=\ln x_1-\sum_{i\geq 1} \frac{1}{i} x_1^{-i}x_2^i,\qquad (x_1-x_2)^h =\sum_{i\geq 0}(-1)^i\binom{h}{i} x_1^{h-i}x_2^i
\end{equation*}
for $h\in\CC$.

Note that the analytic function on the right side of \eqref{eqn:YM'_char_2} is a finite linear combination of monomials in $x_1$, $\ln x_1$, $x_2$, $\ln x_2$, $x_0$, and $\ln x_0$. In view of the left side of \eqref{eqn:YM'_char_1}, we want to expand this analytic function as a series in $x_0$ and $x_2$ in the region $x_2>x_0>0$. To do so, we make the substitutions
\begin{equation*}
    x_1^h\mapsto (x_2+x_0)^h =\sum_{i\geq 0} \binom{h}{i} x_0^i x_2^{h-i},\qquad \ln x_1\mapsto\ln(x_2+x_0)=\ln x_2-\sum_{i\geq 1} \frac{(-1)^{i}}{i} x_0^i x_2^{-i}
\end{equation*}
for $h\in\CC$. Thus in terms of $x_0$ and $x_2$, \eqref{eqn:YM'_char_2} is equal to the following series, which we continue to analyze using the conjugation formulas of \cite[Proposition 3.36 and Remark 3.38]{HLZ2}:
\begin{align}\label{eqn:YM'_char_3}
   & \left\langle e^{x_0 L_1} (e^{\pi i}x_0(x_2+x_0))^{-L_0} m',Y_M\left(e^{(x_2+x_0)^{-1}L_1} (x_0^{-1}(x_2+x_0))^{L_0}a, x_2\right) e^{x_0^{-1}L_1} (x_0(x_2+x_0)^{-1})^{L_0} m\right\rangle\nonumber\\
   &\quad =\left\langle m', (e^{\pi i}x_0(x_2+x_0))^{-L_0}  e^{x_0 L_{-1}} Y_M\left(e^{(x_2+x_0)^{-1}L_1} (x_0^{-1}(x_2+x_0))^{L_0}a, x_2\right) e^{x_0^{-1}L_1} (x_0(x_2+x_0)^{-1})^{L_0} m\right\rangle\nonumber\\
   &\quad =\left\langle m', (e^{\pi i}x_0(x_2+x_0))^{-L_0}   Y_M\left(e^{(x_2+x_0)^{-1}L_1} (x_0^{-1}(x_2+x_0))^{L_0}a, x_2+x_0\right)\cdot\right.\nonumber\\
   &\hspace{20em}\cdot\left.e^{x_0 L_{-1}} e^{x_0^{-1}L_1} (x_0(x_2+x_0)^{-1})^{L_0} m\right\rangle\nonumber\\
   &\quad =\left\langle m', Y_M\left( (e^{\pi i}x_0(x_2+x_0))^{-L_0} e^{(x_2+x_0)^{-1}L_1} (x_0^{-1}(x_2+x_0))^{L_0}a, (e^{\pi i}x_0)^{-1}\right)\cdot\right.\nonumber\\
   &\hspace{20em} \cdot\left. (e^{\pi i}x_0(x_2+x_0))^{-L_0}e^{x_0 L_{-1}} e^{x_0^{-1}L_1} (x_0(x_2+x_0)^{-1})^{L_0} m\right\rangle\nonumber\\
   &\quad =\left\langle m', Y_M(e^{-x_0 L_1} e^{\pi i L_0}(e^{\pi i}x_0)^{-2L_0} a, (e^{\pi i}x_0)^{-1}) e^{-(x_2+x_0)^{-1}L_{-1}} e^{-(x_2+x_0)L_1} e^{-\pi i L_0}(x_2+x_0)^{-2L_0} m\right\rangle.
\end{align}
We continue to analyze the final terms in this expression using the skew-symmetry intertwining operator construction $\Omega_0$ of \cite{HLZ2}, as well as \cite[Proposition 3.36 and Remark 3.38]{HLZ2}:
\begin{align*}
    &e^{-(x_2+x_0)^{-1}L_{-1}} e^{-(x_2+x_0)L_1} e^{-\pi i L_0}(x_2+x_0)^{-2L_0} m\nonumber\\
    &\quad = \left.e^{x_0 L_1}e^{-x_0 L_1} e^{-x_1^{-1}L_{-1}} e^{-x_1 L_1}e^{-\pi i L_0}x_1^{-2L_0} m\right\vert_{x_1=x_2+x_0}\nonumber\\
    &\quad = \left.e^{x_0L_1} e^{-x_0 L_1}\Omega_0(Y_M)(e^{-x_1 L_1}e^{-\pi i L_0}x_1^{-2L_0} m,  -x_1^{-1})\vac\right\vert_{x_1=x_2+x_0}\nonumber\\
    &\quad = \left.e^{x_0L_1}\Omega_0(Y_M)\left(e^{-x_0(1-x_0x_1^{-1})L_1}(1-x_0x_1^{-1})^{-2L_0} e^{-x_1 L_1} e^{-\pi iL_0} x_1^{-2L_0} m, -\frac{ x_1^{-1}}{1-x_0x_1^{-1}}\right)\vac\right\vert_{x_1=x_2+x_0} \nonumber\\
     &\quad = \left.e^{x_0L_1}\Omega_0(Y_M)(e^{-x_0(1-x_0x_1^{-1})L_1} e^{-x_1(1-x_0x_1^{-1})^2 L_1} e^{-\pi iL_0}(x_1-x_0)^{-2L_0} m, -(x_1-x_0)^{-1})\vac\right\vert_{x_1=x_2+x_0}\nonumber\\
     &\quad = e^{x_0L_1}e^{-x_2^{-1}L_{-1}}e^{-x_2L_1}e^{-\pi iL_0}x_2^{-2L_0} m.
\end{align*}
Plugging this back into \eqref{eqn:YM'_char_3} and then recalling \eqref{eqn:YM'_char_1}, we get
\begin{align}\label{eqn:YM'_char_4}
\langle e^{-x_2^{-1}L_1} & e^{x_0 L_{-1}}  Y_{M'}(a,e^{\pi i}  x_0)m', e^{-x_2 L_1} e^{-\pi i L_0}x_2^{-2L_0} m\rangle\nonumber\\
&=\langle m', Y_M(e^{-x_0 L_1} e^{\pi i L_0}(e^{\pi i}x_0)^{-2L_0} a, (e^{\pi i}x_0)^{-1})e^{x_0L_1}e^{-x_2^{-1}L_{-1}}e^{-x_2L_1}e^{-\pi iL_0}x_2^{-2L_0} m\rangle
\end{align}
as absolutely convergent series in powers $x_0$ and $x_2$ in the region $x_2>x_0>0$.

The equality of analytic functions in \eqref{eqn:YM'_char_4} is also an equality of formal series (see \cite[Proposition 7.8]{HLZ5} and \cite[Proposition 2.1]{huang:applicability}). Thus substituting $x_0\mapsto e^{-\pi i}x_0$ and $m\mapsto x_2^{2L_0}e^{\pi i L_0}e^{x_2L_1} m$, we get
\begin{equation*}
   \langle Y_{M'}(a,  x_0)m', e^{-x_0 L_{1}} e^{-x_2^{-1}L_{-1}}  m\rangle\nonumber\\
=\langle m', Y_M(e^{x_0 L_1} e^{\pi i L_0}x_0^{-2L_0} a, x_0^{-1})e^{-x_0L_1}e^{-x_2^{-1}L_{-1}} m\rangle 
\end{equation*}
Finally, substituting $m\mapsto e^{x_2^{-1}L_{-1}}e^{x_0L_1}m$, we get \eqref{eqn:YM'_char}.
\end{proof}

In the setting of the preceding theorem, one can show from the definitions that the left $A$-action $\mu_{M'}^-$ of \eqref{eqn:muDM-_def} and the left action $\mu_{M'}$ (or $\mu_{M'}^+$) of \eqref{eqn:mu-DM-def} are related by the double braiding:
\begin{equation*}
    \mu_{M'}^- =\mu_{M'}^+\circ c_{A,M'}^{-1}\circ c_{M',A}^{-1}.
\end{equation*}
So by Theorem \ref{thm:CAloc_ribbon_GV}, if $(M,\mu_M)\in\cC_A^{\loc}$, then $(M',\mu_{M'}^+)\in\cC_A^{\loc}$ as well, and therefore $\mu_{M'}^-=\mu_{M'}^+$:
\begin{corollary}\label{cor:VOA-and-tens-cat-contras-same}
    In the setting of Theorem \ref{thm:VOA_and_tens_cat_contras}, if $(M,\mu_M)\in\cC_A^{\loc}$, then $Y_{M'}=\mu_{M'}\circ\cY_\boxtimes$, where $Y_{M'}: A\rightarrow \End(M')[[x,x^{-1}]]$ is the vertex operator for the contragredient $A$-module $M'$, and $\mu_{M'}: A\boxtimes_V M'\rightarrow M'$ is the left $A$-action defined in \eqref{eqn:mu-DM-def}. In particular, the ribbon Grothendieck-Verdier category structure on $\cC_A^{\loc}$ considered as a vertex tensor category of modules for the VOA $A$ is the same as that on $\cC_A^{\loc}$ considered as the braided monoidal category of local modules for the commutative algebra $A$ in $\cC$.
\end{corollary}

\begin{remark}
    The result from Theorem \ref{thm:VOA_and_tens_cat_contras} that \eqref{eqn:YM'_char} defines a non-local $A$-module structure on the graded dual space $M'$ of a non-local $A$-module $M$ (in a braided monoidal category for a vertex operator subalgebra $V$) seems to be new in general. However, this result is known when $V=A^g$ is the fixed-point subalgebra of an automorphism $g$ and $M$ is a $g$-twisted $A$-module, in which case $M'$ is a $g^{-1}$-twisted $A$-module; see \cite[Proposition 2.5]{Xu} for the case that $g$ has finite order, and \cite[Proposition 3.3]{Huang:Tw-Intw-Ops} for the case of general $g$. In these results on $g$-twisted $A$-modules, it is not necessary to assume that $M$ is contained in some braided monoidal category of $A^g$-modules. 
\end{remark}

\subsection{Rigidity for vertex operator algebra extensions}\label{subsec:rigidity-for-VOAs}

We continue in the setting of the previous subsection, which we now formalize as follows:
\begin{setup}\label{setup1}
Let $V$ be a self-contragredient $\ZZ$-graded VOA and $\cC$ a vertex tensor category of $V$-modules in the sense of \cite{HLZ1, HLZ2, HLZ3, HLZ4, HLZ5, HLZ6, HLZ7, HLZ8} that is closed under contragredients. In particular, $\cC$ is a ribbon r-category \cite{ALSW} with twist $\theta=e^{2\pi i L_0}$. Let $A$ be a $\ZZ$- or $\frac{1}{2}\ZZ$-graded VOA which contains $V$ as a vertex operator subalgebra such that $A$ is an object of $\cC$ and $\Hom_{\cC}(V, A)\cong \CC$. In particular, $\cC_A^{\loc}$ is isomorphic as a braided monoidal category to the vertex tensor category of $A$-modules that lie in $\cC$ \cite{creutzig2017tensor}. We continue to denote this category by $\cC_A^{\loc}$.
\end{setup}

Note that $A$ is a simple VOA if and only if it is simple as an object in  $\cC_A^{\loc}$. Since there are full embeddings $\cC_A^{\loc} \hookrightarrow \cC_A \hookrightarrow {}_A\cC_A$ it is then in particular a simple object in ${}_A\cC_A$, as well as a simple commutative haploid algebra in the braided monoidal category $\cC$. 
If $\cC$ is rigid, then since $\cC$ is closed under contragredients, necessarily $W'\cong W^*$ for any object $W\in\cC$. Then $W'\cong {}^*W$ as well since $\cC$ is braided. We now translate the results of Section \ref{sec:rig-of-CA} into statements about representation categories of VOAs. Corollary~\ref{cor:EO-exact-commutative} and Theorem~\ref{thm:com-exact-2} translate into:
\begin{corollary}\label{thm:com-exact-2-VOA}
Under Setup \ref{setup1}, let $\cC$ be a finite tensor category and $A$ a simple VOA. If either of the following equivalent conditions holds,
\begin{enumerate}
\item[(a)] There is an $A$-module embedding $A'\hookrightarrow A\boxtimes_V W$ for some $W\in\cC$,
\item[(b)] $A'\boxtimes_A A' \neq 0$,
\end{enumerate}
then $\cC_A$ is a finite tensor category and $\cC_A^{\loc}$ is a braided finite tensor category. If $\cC$ is non-degenerate, then so is $\cC_A^{\loc}$. These conclusions hold in particular if $A$ is self-contragredient, that is, $A\cong A'$ as $A$-modules.
\end{corollary}
\begin{proof}
By Proposition \ref{prop:rcat-dual-and-rigid-dual-same} and Corollary \ref{cor:VOA-and-tens-cat-contras-same}, the $A$-module $(A',Y_{A'})$ in $\cC_A^{\loc}$ considered as a vertex tensor category corresponds to the object $(A^*,\mu_{A^*}^r\circ c_{A,A^*})$ in $\cC_A^{\loc}$ considered as a braided monoidal category. Moreover, the tensor product $\boxtimes_A$ in the vertex tensor category $\cC_A^{\loc}$ is the same as the tensor product $\otimes_A$ in Corollary \ref{cor:EO-exact-commutative} by \cite[Theorem 3.65]{creutzig2017tensor}. $A$ being haploid also implies that it is indecomposable.
Thus the conditions in this corollary are the same as those in Corollary \ref{cor:EO-exact-commutative}, and then the conclusions follow from Theorem \ref{thm:com-exact-2}.
\end{proof}

Corollary \ref{cor:com-exact-integral} translates to:
\begin{corollary}\label{cor:com-exact-integral-VOA}
Under Setup \ref{setup1}, let $\cC$ be integral and finite. If $A$ is a simple VOA, then $\cC_A^{\loc}$ is a braided finite tensor category and $\cC_A$ is a finite tensor category. 
\end{corollary}

Theorem \ref{thm:comm-exact-semisimple} translates into: 
\begin{corollary}\label{thm:comm-exact-semisimple-VOA}
Under Setup \ref{setup1},
let $\cC$ be a braided fusion category. If $A$ is a simple VOA, then $\cC_A^{\loc}$ and $\cC_A$ are also fusion categories.
\end{corollary}

Finally, Corollary \ref{cor:com-exact-5} becomes: 
\begin{corollary}\label{cor:com-exact-5-VOA}
Under Setup \ref{setup1},
let $\cC$ be a finite tensor category and $A$ a simple VOA. If $A$ satisfies either of the following conditions, then $A$ is exact:
\begin{enumerate}
\item[(a)] $A$ is a Frobenius algebra.
    \item[(b)] $\Hom_{\cC}(A,V)\neq 0$.
\end{enumerate}
Consequently, in these cases,  $\cC_A^{\loc}$ is a braided finite tensor category and $\cC_A$ is a finite tensor category.
\end{corollary}

We now discuss the consequences of Corollaries \ref{thm:com-exact-2-VOA} and \ref{thm:comm-exact-semisimple-VOA} for extensions of $C_2$-cofinite and rational VOAs.
Following \cite{creutzig_gannon}, we call a VOA \textit{strongly finite} if:
\begin{itemize}
    \item $V$ is $\NN$-graded by conformal weights, that is, $V=\bigoplus_{n=0}^\infty V_{(n)}$,
    \item $V$ is simple,
    \item $V$ is self-contragredient, that is, $V\cong V'$ as a $V$-module, and
    \item $V$ is $C_2$-cofinite, that is, $\dim V/C_2(V)<\infty$ where $C_2(V)=\mathrm{span}\lbrace u_{-2} v\mid u,v\in V\rbrace$.
\end{itemize}
For a strongly finite VOA $V$, we use $\Rep(V)$ to denote the category of grading-restricted generalized $V$-modules. By \cite{huangC2}, $\Rep(V)$ is a finite abelian category and a vertex tensor category in the sense of \cite{HLZ1, HLZ2, HLZ3, HLZ4, HLZ5, HLZ6, HLZ7, HLZ8}. Thus if $\mathrm{Rep}(V)$ is rigid, then it is a braided finite ribbon category, and in this case $\mathrm{Rep}(V)$ is non-degenerate, that is, it is a (possibly non-semisimple) modular tensor category \cite{McRae2021}.

We will consider extensions of $C_2$-cofinite VOAs. For this, we need the following well-known result, giving a proof for completeness:
\begin{lemma}\label{lem:C2_of_extension}
    If $V$ is a $\NN$-graded $C_2$-cofinite VOA and $V\subseteq A$ is a VOA extension, then $A$ is $C_2$-cofinite.
\end{lemma}
\begin{proof}
By assumption, $A$ is a grading-restricted generalized $V$-module. Since the Zhu algebra $A(V)$ is finite dimensional \cite[Theorem 2.5]{miyamoto2004modular}, $V$ has finitely many irreducible grading-restricted modules, and therefore every grading-restricted $V$-module has finite length (see \cite[Propositions 3.14 and 3.15]{huangC2}). In particular, every grading-restricted generalized $V$-module is finitely generated. Then the spanning set of \cite[Lemma 2.4]{miyamoto2004modular} (for singly-generated weak $V$-modules) shows that every grading-restricted generalized $V$-module is $C_2$-cofinite. Thus $A$ is $C_2$-cofinite even as a $V$-module, and is therefore $C_2$-cofinite as a VOA.
\end{proof}

A strongly finite VOA $V$ is called \textit{strongly rational} if $\Rep(V)$ is semisimple. In this case, Huang proved that $\Rep(V)$ is rigid and non-degenerate \cite{Huang:Rig_Mod}, and thus Corollary \ref{thm:comm-exact-semisimple-VOA} applies to extensions of $V$:
\begin{theorem}\label{thm:VOA_ext_rational}
If $V$ is a strongly rational VOA and $V\subseteq A$ is a VOA extension such that $A$ is simple and $\NN$-graded, then $A$ is strongly rational. Moreover, the category $\Rep(V)_A$ of non-local $A$-modules in $\Rep(V)$ is a fusion category.
\end{theorem}
\begin{proof}
By assumption, $A$ is simple and $\NN$-graded, and $A$ is $C_2$-cofinite by Lemma \ref{lem:C2_of_extension}. To show $A$ is self-contragredient,
we use a similar argument as in the proofs of \cite[Theorem 4.14]{McRaeCptOrb} and \cite[Theorem 1.2]{McRaeSS}. Specifically, $A$ is self-contragredient if and only if there is a non-degenerate invariant bilinear form $(-,-): A\times A\rightarrow\CC$. By \cite[Theorem 3.1]{Li_bilinear}, the space of invariant bilinear forms on $A$ is isomorphic to $(A_{(0)}/L_1A_{(1)})^*$, and any non-zero invariant form is non-degenerate because $A$ is simple. Thus it is enough to show that $A_{(0)}/L_1A_{(1)}\neq 0$. Indeed, since $A$ is an object of the semisimple category $\Rep(V)$, we can decompose $A=V\oplus\widetilde{A}$ as a $V$-module, and then
\begin{equation*}
    L_1A_{(1)} =L_1V_{(1)}\oplus L_1\widetilde{A}_{(1)}\subsetneq V_{(0)}\oplus\widetilde{A}_{(0)}= A_{(0)}.
\end{equation*}
The inclusion is strict because $V$ is self-contragredient and therefore $V_{(0)}/L_1V_{(1)}\neq 0$, again by \cite[Theorem 3.1]{Li_bilinear}. Thus $A$ is self-contragredient.

Finally, the category $\Rep(A)$ of grading-restricted generalized modules for $A$ considered as a VOA is isomorphic to the category $\Rep(V)_A^{\mathrm{loc}}$ of local modules for $A$ considered as a commutative algebra in $\Rep(V)$ by \cite[Theorem 3.4]{Huang:2014ixa}, since every grading-restricted generalized $A$-module is also a grading-restricted generalized $V$-module. Thus $\Rep(A)$ is semisimple by Corollary \ref{thm:comm-exact-semisimple-VOA}, and thus $A$ is strongly rational. Moreover, the category $\Rep(V)_A$ of non-local $A$-modules in $\Rep(V)$ is a fusion category.
\end{proof}

\begin{remark}
This theorem substantially strengthens \cite[Theorem 1.2(1)]{McRaeSS} by removing the requirement that the dimension of $A$ in the modular tensor category of $V$-modules be non-zero. 
\end{remark}

If $V$ is strongly finite but not rational, then it is not known in general whether $\mathrm{Rep}(V)$ is rigid. However, Corollary \ref{thm:com-exact-2-VOA} shows that rigidity is preserved under extensions under mild conditions:
\begin{theorem}\label{thm:exts_of_C2_VOAs}
    Let $V$ be a strongly finite VOA such that $\Rep(V)$ is rigid, and let $V\subseteq A$ be a VOA extension such that $A$ is simple and either of the following equivalent conditions holds:
    \begin{enumerate}
        \item[(a)] There is an $A$-module embedding $A'\hookrightarrow A\boxtimes_V W$ for some $W\in\Rep(V)$,
        \item[(b)] $A'\boxtimes_A A'\neq 0$.
    \end{enumerate}
    Then the category $\Rep(A)$ of grading-restricted generalized $A$-modules is a non-degenerate braided finite tensor category and the category $\Rep(V)_A$ of non-local $A$-modules in $\Rep(V)$ is a finite tensor category. These conclusions hold in particular if $A$ is self-contragredient. Moreover, if $A$ is self-contragredient and $\ZZ$-graded, then $\Rep(A)$ is a (possibly non-semisimple) modular tensor category.
\end{theorem}
\begin{proof}
    Since $V$ is strongly finite and $\Rep(V)$ is rigid, it is non-degenerate by \cite[Main Theorem 1]{McRae2021}. Thus if either of conditions (a) or (b) hold, then $\Rep(V)_A^\loc$ is a non-degenerate braided finite tensor category and $\Rep(V)_A$ is a finite tensor category by Corollary \ref{thm:com-exact-2-VOA}. Moreover, $\Rep(V)_A^{\loc}=\Rep(A)$ as in the proof of Theorem \ref{thm:VOA_ext_rational}. If $A$ is self-contragredient, then both conditions (a) and (b) hold because there is an $A$-module embedding $A'\xrightarrow{\sim} A\xrightarrow{\sim}A\boxtimes_V V$ and because $A'\boxtimes_A A'\cong A\boxtimes_A A\cong A\neq 0$.
If $A$ is self-contragredient and $\ZZ$-graded, then $\Rep(A)$ will be a modular tensor category if it is ribbon. Indeed, since $A$ is $\ZZ$-graded, $\theta=e^{2\pi i L_0}$ defines a natural automorphism of $\id_{\Rep(A)}$, and since $A$ is self-contragredient, duals in $\Rep(A)$ are given by contragredient modules. Thus $\theta$ is a ribbon structure on $\Rep(A)$ by the same proof as in \cite[Theorem 4.1]{Huang:Rig_Mod}.
\end{proof}

\begin{remark}\label{rem:A-not-self-contra}
In the setting of the preceding theorem, even if $A$ is not self-contragredient, it is possible that $\Rep(A)$ could be a modular tensor category. For example, the triplet algebra $\mathcal{W}(p)$ \cite{Kausch, Adamovic:2007er} is a strongly finite vertex operator subalgebra of the lattice VOA $V_{\sqrt{2p}\ZZ}$. Although $V_{\sqrt{2p}\ZZ}$ is not self-contragredient for the conformal vector of $\mathcal{W}(p)$, it is strongly rational with a different conformal vector. Thus $\Rep(V_{\sqrt{2p}\ZZ})$ is even a semisimple modular tensor category, although $\Rep(\mathcal{W}(p))$ is non-semisimple.
\end{remark}


A VOA $V$ may have more than one conformal vector, giving different Virasoro actions on $V$, and it is not clear in general whether the braided monoidal structure on a category $\cC$ of $V$-modules is independent of the choice of conformal vector. Here we summarize some of the discussion of this issue from \cite[\S 2]{McRae2021}. Let $\omega$ and $\widetilde{\omega}$ be two conformal vectors in $V$, let $L_n$ and $\tilL_n$ denote the Virasoro operators on $V$ associated to $\omega$ and $\widetilde{\omega}$, respectively, and let $\cC$ be a category of $V$-modules.
\begin{itemize}
    \item For $\mathcal{C}$ to be a braided monoidal category, any module $W\in\cC$ should have a conformal weight grading $W=\bigoplus_{h\in\CC} W_{[h]}$, since associativity isomorphisms in $\cC$ are constructed using graded duals $W'=\bigoplus_{h\in\CC} W_{[h]}^*$. Thus for braided monoidal structure on $\cC$ to be the same whether $\omega$ or $\widetilde{\omega}$ is the conformal vector of $V$, any $W\in\cC$ should decompose into generalized eigenspaces for both $L_0$ and $\tilL_0$, and the graded duals of $W$ with respect to $L_0$ and $\tilL_0$ should embed as the same subspace of the full dual $W^*$. By Lemma 2.5 and Proposition 2.6 of \cite{McRae2021} and their proofs, these conditions hold if $W$ decomposes into finite-dimensional generalized $L_0$-eigenspaces and $[L_0,\tilL_0]=0$ on $W$.

    \item The definition of intertwining operators (which are used to define tensor products in $\cC$) and the right unit and braiding isomorphisms of $\cC$ use the Virasoro $L_{-1}$ operator. Thus for braided monoidal structure on $\cC$ to be the same whether we use $\omega$ or $\widetilde{\omega}$, we need $L_{-1}=\tilL_{-1}$ on any module in $\cC$.

    \item For a module $W\in\cC$, the contragredient module structure on $W'$ in \eqref{eqn:VOA_contra_structure} uses the Virasoro $L_0$ and $L_1$ operators. Thus different conformal vectors may yield different Grothendieck-Verdier duality structures on $\cC$. The ribbon twist $\theta=e^{2\pi i L_0}$ also depends on the choice of conformal vector.
\end{itemize}
One common way to obtain a new conformal vector $\widetilde{\omega}$ from an old one $\omega$ is to take $\widetilde{\omega}=\omega+L_{-1}v$ for some $v\in V$ such that $L_0v=v$.  If $v_0\omega =0$ and every object of $\cC$ decomposes into finite-dimensional generalized $L_0$-eigenspaces, then the braided monoidal structure on a category $\cC$ of $V$-modules does not depend on whether we use $\omega$ or $\widetilde{\omega}$ as the conformal vector $V$.

In view of the preceding discussion, we should always assume that the braided (ribbon) category structure on $\cC_A^{\loc}$ in Corollary \ref{thm:com-exact-2-VOA} through Theorem \ref{thm:exts_of_C2_VOAs} is that corresponding to the conformal vector $\omega$ of $V$, even if $A$ is not self-contragredient with respect to $\omega$ but is self-contragredient with respect to another conformal vector $\widetilde{\omega}$ (recall the triplet and lattice VOA example mentioned in Remark \ref{rem:A-not-self-contra}). However, in many cases, at least the braided monoidal category structure on $\cC_A^{\loc}$ will not depend on whether we use $\omega$ or $\widetilde{\omega}$, though the ribbon Grothendieck-Verdier duality structures will normally be different. In the next technical result, we will exploit multiple conformal vectors to check the condition in Theorem \ref{thm:exts_of_C2_VOAs}.

If $V$ is $C_2$-cofinite, let $\Rep(V,\omega)$ denote the category of $V$-modules which are grading restricted with respect to the $L_0$ operator of the conformal vector $\omega$. It is a full subcategory of the category of weak $V$-modules (which are $V$-modules with no generalized $L_0$-eigenspace grading assumptions). By the results of \cite{miyamoto2004modular, huangC2} (see also \cite[Proposition 2.11]{McRae2021}), if $V$ is $C_2$-cofinite and $\NN$-graded by $L_0$-eigenvalues, then $\Rep(V,\omega)$ equals the category of finitely-generated weak $V$-modules. 
\begin{theorem}\label{thm:rig-for-C2-exts}
Let $V$ be a strongly finite VOA with conformal vector $\omega$ such that $\Rep(V)$ is rigid, and let $V\subseteq A$ be a VOA extension such that $A$ is  strongly finite for a conformal vector $\widetilde{\omega}\in \omega+L_{-1}A$. Then the category $\Rep(A,\omega)$ of grading-restricted generalized $A$-modules is a (possibly non-semisimple) modular tensor category, and the category $\Rep(V)_A$ of non-local $A$-modules in $\Rep(V)$ is a finite tensor category.
\end{theorem}
\begin{proof}
We use $L_n$ and $\tilL_n$ to denote the Virasoro operators associated to $\omega$ and $\widetilde{\omega}$, respectively, acting on any weak $A$-module. Since we assume $\widetilde{\omega}=\omega+L_{-1}a$ for some $a\in A$, we have
\begin{equation*}
    \tilL_{-1}=L_{-1}+\mathrm{Res}_x\frac{d}{dx}Y_M(a,x) = L_{-1}
\end{equation*}
on any weak $A$-module $M$. However, $L_0$ and $\tilL_0$ may act differently. Nonetheless, we claim that the categories $\Rep(A,\omega)$ and $\Rep(A,\widetilde{\omega})$ of grading-restricted generalized $A$-modules with respect to $\omega$ and $\widetilde{\omega}$ give the same full subcategory of the category of weak $A$-modules.

To prove the claim, $\Rep(A,\widetilde{\omega})$ is the category of finitely-generated weak $A$-modules since $(A,\widetilde{\omega})$ is $\NN$-graded. On the other hand, every object $M\in\Rep(A,\omega)$ is a grading-restricted generalized $V$-module; so because $(V,\omega)$ is $\NN$-graded, $M$ is finitely generated even as a $V$-module. This shows that $\Rep(A,\omega)\subseteq\Rep(A,\widetilde{\omega})$. For the reverse inclusion, it is enough to show that if $M$ is a singly-generated weak $A$-module, then $M\in\Rep(A,\omega)$. Indeed, if $m$ generates $M$ as an $A$-module, then because $(A,\widetilde{\omega})$ is $\NN$-graded, \cite[Lemma 2.4]{miyamoto2004modular} implies there is a finite set $S\subseteq A$ and $N\in\ZZ$ such that
    \begin{equation*}
        M=\mathrm{span}\lbrace v^{(1)}_{-n_1}v^{(2)}_{-n_2}\cdots v^{(k)}_{-n_k}m\mid v^{(1)},v^{(2)},\ldots, v^{(k)}\in S,\,\,n_1> n_2>\cdots >n_k\geq N\rbrace.
    \end{equation*}
    By replacing elements in $S$ with all of their projections to the $L_0$-eigenspaces of $A$, we may assume that the elements of $S$ are $L_0$-eigenvectors. Moreover, since $(V,\omega)$ is $\NN$-graded and $C_2$-cofinite, and since $M$ is a weak $V$-module, \cite[Theorem 2.7]{miyamoto2004modular} implies that the generator $m\in M$ is a sum of generalized $L_0$-eigenvectors. Without loss of generality, assume $m$ is a generalized $L_0$-eigenvector with generalized eigenvalue $\lambda$. Then the spanning set element $v^{(1)}_{-n_1}v^{(2)}_{-n_2}\cdots v^{(k)}_{-n_k}m$ of $M$ has generalized $L_0$-eigenvalue 
    \begin{equation*}
        \mathrm{wt}\,v^{(1)}+\cdots+\mathrm{wt}\, v^{(k)}+n_1+\cdots+n_k -k +\lambda,
    \end{equation*}
   showing that $M$ is the direct sum of generalized $L_0$-eigenspaces. Moreover, the inequality
    \begin{equation*}
        \mathrm{wt}\,v^{(1)}+\cdots+\mathrm{wt}\, v^{(k)}+n_1+\cdots+n_k -k \geq \frac{1}{2}k(k-1) +(M+N-1)k,
    \end{equation*}
    where $M$ is the minimum $L_0$-eigenvalue of any element of the finite set $S$, implies that the generalized $L_0$-eigenspaces are finite dimensional, and that the set of generalized $L_0$-eigenvalues are bounded below. Thus $M$ is a grading-restricted generalized $(A,\omega)$-module, proving the claim.
    
    Now let $A'$ be the contragredient of $A$ with respect to $\omega$; we will show that $A'\boxtimes_A A'\neq 0$, where $\boxtimes_A$ is the tensor product in $\Rep(A,\omega)$. First, since $\Rep(A,\omega)=\Rep(A,\widetilde{\omega})$, $A'$ has a contragredient $\tilA$ with respect to $\widetilde{\omega}$, and $\tilA$ is an object of $\Rep(A,\omega)$ as well as of $\Rep(A,\widetilde{\omega})$. Then since $A$ is self-contragredient with respect to $\widetilde{\omega}$, there is a non-zero $(A,\widetilde{\omega})$-module intertwining operator $\mathcal{Y}: \tilA\otimes A'\rightarrow A[\log x]\lbrace x\rbrace$. Since $L_{-1}=\tilL_{-1}$ on any weak $A$-module, $\cY$ is also an $(A,\omega)$-module intertwining operator. Thus $\cY$ induces a non-zero map $f: \widetilde{A}\boxtimes_A A'\rightarrow A$ which is also surjective since $A$ is simple. We now have a surjection
    \begin{equation*}
        \tilA\boxtimes_A(A'\boxtimes_A A')\xrightarrow{\sim} (\tilA\boxtimes_A A')\boxtimes_A A'\xrightarrow{f\boxtimes_A\id_{A'}} A\boxtimes_A A'\xrightarrow{\sim} A'.
    \end{equation*}
     Thus $A'\boxtimes_A A'\neq 0$ since $A'\neq 0$.
    Theorem \ref{thm:exts_of_C2_VOAs} now implies that $\Rep(A,\omega)=\Rep(V)_A^\loc$ is a non-degenerate braided finite tensor category and $\Rep(V)_A$ is a finite tensor category.

    To show that $\Rep(A,\omega)$ is a modular tensor category, it remains to show that $\Rep(A,\omega)$ is ribbon. Since $A$ is self-contragredient with respect to $\widetilde{\omega}$ but not necessarily $\omega$, we use $\theta=e^{2\pi i \widetilde{L}_0}$. Then $\theta$ is a natural automorphism of $\id_{\Rep(A,\omega)}$ because $A$ is $\ZZ$-graded with respect to $\widetilde{\omega}$. Also, $\theta$ satisfies the balancing equation (as in the proof of \cite[Theorem 4.1]{Huang:Rig_Mod}) because $L_{-1}=\widetilde{L}_{-1}$ on any weak $A$-module, and therefore tensor product intertwining operators in $\Rep(A,\omega)$ are also $(A,\widetilde{\omega})$-module intertwining operators. To show that $\theta_{M}^*=\theta_{M^*}$ for all $M\in\Rep(A,\omega)$, first note that because $(A,\omega)$-intertwining operators are the same as $(A,\widetilde{\omega})$-intertwining operators, and because $A$ is self-contragredient with respect to $\widetilde{\omega}$, $\Rep(A,\omega)$ has an r-category structure given by $(A,\widetilde{\omega})$-contragredients, as in \cite{ALSW}. Thus left duals in $\Rep(A,\omega)$ are given by $(A,\widetilde{\omega})$-contragredients. Under this identification, by \cite[Lemma 4.3.4(1)]{Creutzig:2020qvs}, the dual of a morphism $f: M_1\rightarrow M_2$ in $\Rep(A,\widetilde{\omega})$ is the map $f': M_2'\rightarrow M_1'$ such that
    \begin{equation*}
        \langle f'(m_2'), m_1\rangle =\langle m_2',f(m_1)\rangle
    \end{equation*}
    for $m_1\in M_1$, $m_2'\in M_2'$, where $\langle -,-\rangle$ is the bilinear pairing between an $A$-module and its $(A,\widetilde{\omega})$-contragredient. Since $\widetilde{L}_0$ is self-adjoint in such pairings by \eqref{eqn:VOA_contra_structure}, it follows that $\theta_M'=\theta_{M'}$ for all $M\in\Rep(A,\omega)$, as required.
\end{proof}

\subsection{Vertex operator superalgebra extensions}\label{subsec:superalgebras}

In this subsection, we discuss how our results extend to vertex operator superalgebra (VOSA) extensions of vertex operator algebras. 
\begin{setup}\label{setup:superalgebra}
 Let $V$ be a VOA and $\cC$ a vertex tensor category of $V$-modules. Let $A = A_0\oplus A_1$ be a VOSA extension of $V$ in $\cC$, where $A_0$ and $A_1$ are the even and odd parts of $A$. That is, $A$ is an object in $\cC$ and it is a VOSA such that $V$ is a vertex operator subalgebra of $A_0$. Assume $\Hom_{\cC}(V, A)\cong \CC$, and set $\cD = \cC_{A_0}^{\loc}$.
\end{setup}

In this setup, $A$ can be identified with a commutative superalgebra in both $\cC$ and $\cD$ \cite{Creutzig:2015buk}. Let $\cS\cD$ be the supercategory associated to $\cD$ as in \cite[Definition 2.11]{creutzig2017tensor}. Objects in $\cS\cD$ are ordered pairs $(X_0, X_1)$ such that $X_0,X_1\in\cD$, and
\begin{equation*}
    \Hom_{\cS\cD}((X_0,X_1),(Y_1,Y_2)) = \Hom_\cD(X_0\oplus X_1, Y_0\oplus Y_1).
\end{equation*}
The parity involution of $(X_0,X_1)$ is $\id_{X_0}\oplus(-\id_{X_1})$. In particular, $\cD$ is a subcategory of $\cS\cD$ by identifying $X\in\cD$ with $(X, 0)\in\cS\cD$, and we will just write $X$ for $(X, 0)$. 
Morphisms in $\cS\cD$ need not commute with parity involutions, but it is also useful to consider the underlying category $\underline{\cS\cD}$, which contains the same objects as $\cS\cD$ but only the even morphisms, that is, the morphisms which commute with parity.
If $\cD$ is rigid, then so is $\cS\cD$ \cite[Lemma 2.72]{creutzig2017tensor}. 
The underlying category is then also rigid as evaluations and coevaluations in $\cS\cD$ can be taken even. 

Now, $A$ is a commutative algebra in $\cS\cD$ with even unit and multiplication, equivalently a commutative algebra in $\underline{\cS\cD}$. The categories $\cS\cD_A$ and $\underline{\cS\cD}_A$ are defined as usual, except we require the left $A$-action on any object of $\cS\cD_A$ to be even, so that $\underline{\cS\cD_A}=\underline{\cS\cD}_A$. Also as usual, there are monoidal induction functors $F_A: \cS\cD\rightarrow\cS\cD_A$ and $\underline{F}_A: \underline{\cS\cD} \rightarrow \underline{\cS\cD}_{A}$.

\begin{proposition}\label{prop:simple-induced-modules}
    If $A$ is simple and $X$ is a simple object of ${\cS\cD}$, then $F_A(X)$ is simple in $\cS\cD_A$ and $\underline{F}_A(X)$ is simple in $\underline{\cS\cD}_A$. 
\end{proposition}
\begin{proof}
We first prove that $\underline{F}_A(X)$ is simple in $\underline{\cS\cD}_A$. First suppose $X=(X_0,0)$ where $X_0\in\cD$ is simple. Then $\underline{F}_A(X) \cong (X_0, X_1)$ as an object in $\underline{\cS\cD}$, where  $X_1 = A_1 \boxtimes_{A_0} X_0$. Since $A_1$ is invertible in $\cD$ \cite[Proposition 4.22]{creutzig2017tensor}, $X_1$ is simple \cite[Proposition 2.5(3)]{Creutzig:2016ehb}. Now if $M=(M_0,M_1)$ is an $\underline{\cS\cD}_A$-submodule of $\underline{F}_A(X)$, then $M_i$ is a $\cD$-submodule of $X_i$ for $i=0,1$ because morphisms in $\underline{\cS\cD}_A$ are even. So if $M$ is proper and non-zero, then $M=(X_0,0)$ or $M=(0,X_1)$.
    But then $A_1$ must act as $0$ on $M$ since $A_1$ is odd, and this implies the ideal of $A$ generated by $A_1$ acts as $0$ on $M$. This is impossible if $A$ is simple, so $\underline{F}_A(X)$ must be simple. The case that $X=(0,X_1)$ where $X_1\in\cD$ is simple is similar, since then $\underline{F}_A(X)\cong(A_1\boxtimes_{A_0} X_1,X_1)$ as an object of $\underline{\cS\cD}$.
    

    Now we prove that $F_A(X)$ is simple in $\cS\cD_A$. We may assume $X\cong(X_0,0)$ for some simple $X_0\in\cD$. Since $\underline{F}_A(X)$ is simple in $\underline{\cS\cD}_A=\underline{\cS\cD_A}$, \cite[Proposition 4.19]{creutzig2017tensor} and its proof imply that any proper non-zero $A$-submodule $M\subseteq F_A(X)$ must be isomorphic to $X_0$ as an $A_0$-module. Since $X_0$ is simple, this implies $M\cong(X_0,0)$ or $M\cong(0,X_0)$ as an object of $\cS\cD$, which is impossible as in the previous paragraph. That is, no proper non-zero submodule $M\subseteq F_A(X)$ can have a $\ZZ/2\ZZ$-grading such that the left $A$-action $\mu_M$ is even, and thus $F_A(X)$ is simple as an object of $\cS\cD_A$. 
\end{proof}

\begin{remark}\label{rem:larger-SDA}
If we relax the definition of $\cS\cD_A$ to allow objects $(M,\mu_M)$ such that $\mu_M$ is not an even morphism in $\cS\cD$, then \cite[Proposition 4.19]{creutzig2017tensor} implies that Proposition \ref{prop:simple-induced-modules} should be revised as follows: If $A$ is simple and $X=(X_0,0)$ for some simple object $X_0\in\cD$, then either $F_A(X)$ is simple in $\cS\cD_A$ or $F_A(X)\cong M_+\oplus M_-$ where $(M_\pm,\mu_{M_\pm})$ are simple left $A$-modules such that $M_{\pm}\cong X_0$ as objects of $\cS\cD$ and $\mu_{M_\pm}$ are not even. This second possibility occurs only if $X_0\cong A_1\boxtimes_{A_0} X_0$ in $\cD$.
\end{remark}

\begin{proposition}\label{prop:simple-mods-are-induced}
    If $A$ is simple and $M$ is a simple object of $\cS\cD_A$, respectively $\underline{\cS\cD}_A$, then $M\cong F_A(X)$, respectively $M\cong\underline{F}_A(X)$, for some simple object $X\in\cS\cD$.
\end{proposition}
\begin{proof}
    If $M=(M_0,M_1)$ is simple in $\underline{\cS\cD}_A=\underline{\cS\cD_A}$, then \cite[Proposition 4.22 and Corollary 4.23]{creutzig2017tensor} imply that $M_0$ is simple in $\cD$ and $M\cong\underline{F}_A(M_0)$. Now suppose $M=(M_0,M_1)$ is simple in $\cS\cD_A$; since any embedding $N\hookrightarrow M$ in $\underline{\cS\cD_A}$ is also an embedding in $\cS\cD_A$, it follows that $M$ is also simple in $\underline{\cS\cD_A}$. Thus again $M_0$ is simple in $\cD$ and $M\cong F_A(M_0)$.
\end{proof}

Using the preceding two propositions, we can now prove:
\begin{theorem}\label{thm:rig-and-ss-for-super}
Under Setup \ref{setup:superalgebra}, assume that $A$ is a simple VOSA.
\begin{enumerate}
    \item If $\cD$ is semisimple, then $\cS\cD_{A}$ and $\underline{\cS\cD}_{A}$ are semisimple.
    \item If $A_0$ is self-contragredient and $\cD$ is locally finite, closed under contragredients, and rigid, then $\cS\cD_A$ and $\underline{\cS\cD}_{A}$ are rigid.
\end{enumerate}
\end{theorem}
\begin{proof}
(1) If $\cD$ is semisimple, then so is $\cS\cD$, and thus every induced module in $\cS\cD_A$ is semisimple by Proposition \ref{prop:simple-induced-modules}. Now if $(M,\mu_M)$ is any object of $\cS\cD_A$, then $\mu_M: F_A(M)\rightarrow M$ is an even surjection in $\cS\cD_A$. So $M$ is semisimple because it is a quotient of a semisimple induced module. This shows $\cS\cD_A$ is semisimple, and similarly so is $\underline{\cS\cD}_A$.

(2) Now suppose $A_0$ is self-contragredient and $\cD$ is locally finite, closed under contragredients, and rigid, so the left dual of any object of $\cD$ is its contragredient. Then $\cS\cD$ and $\underline{\cS\cD}$ are also rigid. We claim that $A'\cong A$ in $\cS\cD_A$ and $\underline{\cS\cD}_A$. Indeed, since $A_0$ is self-contragredient, there is an $A_0$-module injection $A_0\hookrightarrow A'$. Since the induction functor $\underline{F}_A:\underline{\cS\cD}\rightarrow\underline{\cS\cD}_A$ is left adjoint to the forgetful functor, this inclusion induces an even non-zero $A$-module homomorphism $\underline{F}_A(A_0)\cong A\rightarrow A'$ As $A$ is simple, and thus $A'$ is simple as well, this non-zero map is an isomorphism, so $A\cong A'$ in $\underline{\cS\cD}_A$ as well as in $\cS\cD_A$.



It now follows from Theorem \ref{thm:GVinRepA} that $\underline{\cS\cD}_A$ is an r-category, and $\underline{\cS\cD}_A$ is locally finite as in the proof of Theorem \ref{thm:GVto rigid}. Because $\underline{\cS\cD}$ is rigid, every induced module in $\underline{\cS\cD}_A$ is rigid \cite[Exercise 2.10.6]{etingof2016tensor}, and therefore all simple objects of $\underline{\cS\cD}_A$ are rigid by Proposition \ref{prop:simple-mods-are-induced}. It now follows from Theorem \ref{thm:rigidityfrom simples} that $\underline{\cS\cD}_A$ is rigid. Moreover, since the (even) evaluations and coevaluations in $\underline{\cS\cD}_A=\underline{\cS\cD_A}$ are also $\cS\cD_A$-morphisms, $\cS\cD_A$ is rigid as well.
%
%
\end{proof}

\begin{remark}
    If we relax the definition of $\cS\cD_A$ to allow objects $(M,\mu_M)$ such that $\mu_M$ is not even, then Remark \ref{rem:larger-SDA} shows that if $A$ is simple and $\cD$ is semisimple, then all induced modules in $\cS\cD_A$ are still semisimple. Hence this larger $\cS\cD_A$ is still semisimple. However, it is not clear whether this larger $\cS\cD_A$ is still a monoidal category, so it does not seem to make sense to discuss its rigidity.
\end{remark}

One can now study the representation theory of $A$ based on the representation theory of a vertex operator subalgebra $V\subseteq A_0$, under suitable conditions. If the category $\cC$ of $V$-modules in Setup \ref{setup:superalgebra} is a (semisimple) braided finite tensor category, for example, then one can try to use the results in Corollary \ref{thm:com-exact-2-VOA} through Theorem \ref{thm:rig-for-C2-exts} to conclude the same for $\cD=\cC_{A_0}^{\loc}$. One then applies Theorem \ref{thm:rig-and-ss-for-super} to conclude the same for $\cS\cD_A$ and $\underline{\cS\cD}_A$, at least if $A_0$ is self-contragredient. Note that even if $A_0$ is not self-contragredient with respect to the conformal vector of $V$, it may still be self-contragredient with respect to a different conformal vector, as in Theorem \ref{thm:rig-for-C2-exts} for example. As an example of the kind of results one can obtain in this way, we present the superalgebra generalization of Theorem \ref{thm:VOA_ext_rational}:
\begin{theorem}\label{thm:VOSA-ext-rational}
    Let $V$ be a strongly rational VOA and let $V\subseteq A$ be a VOSA extension whose even part $A_0$ is $\ZZ$-graded and such that $V\subseteq A_0$. If $A$ is simple and $\ZZ_{\geq 0}$- or $\frac{1}{2}\ZZ_{\geq 0}$-graded, then $A$ is strongly rational.
\end{theorem}
\begin{proof}
    The even part $A_0$ is a simple VOA (see for example \cite[Proposition 4.22]{creutzig2017tensor}), so by Theorem \ref{thm:VOA_ext_rational}, $A_0$ is strongly rational and $\cD=\mathrm{Rep}(V)_{A_0}^\loc=\mathrm{Rep}(A_0)$ is a semisimple modular tensor category. It then follows from Theorem \ref{thm:rig-and-ss-for-super}(1) that $\cS\cD_A$ and $\underline{\cS\cD}_A$ are semisimple, and then so are $\cS\cD_A^\loc$ and $\underline{\cS\cD}_A^\loc$ since it is easy to show that any submodule of a local module must be local.

    Now to show $A$ is strongly rational, first $A$ is non-negatively graded and simple by hypothesis. Also, since $A_0$ is self-contragredient, so is $A$ as in the proof of Theorem \ref{thm:rig-and-ss-for-super}(2). Next, $A$ is $C_2$-cofinite even as an $A_0$-module as in the proof of Lemma \ref{lem:C2_of_extension}. Finally, because $\cS\cD_A^\loc$ and $\underline{\cS\cD}_A^\loc$ are semisimple, the proof of \cite[Theorem 5.4(1)]{McRaeSS} shows that every $\NN$-gradable or $\frac{1}{2}\NN$-gradable $A$-module is semisimple, and thus $A$ is rational.
\end{proof}

\begin{remark}
    Theorem \ref{thm:VOSA-ext-rational} strengthens \cite[Theorem 5.4(1)]{McRaeSS} by removing the requirement that the dimension of $A_0$ be non-zero in the modular tensor category of $V$-modules.
\end{remark}


\subsection{Examples}\label{sec:ex}

We now give some examples of the results in the preceding subsections.

\begin{example} 
Let $\mathfrak g$ be a simple Lie algebra or $\mathfrak g = \mathfrak{osp}_{1|2n}$, with Killing form normalized so that long roots have squared length $2$. Let $V^k(\mathfrak g)$ be the affine VO(S)A of $\mathfrak g$ at level $k\in\CC$, and let $L_k(\mathfrak g)$ be its simple quotient. Also let $W^\ell(\mathfrak g)$ be the principal $W$-(super)algebra associated to $V^\ell(\mathfrak g)$, and let $W_\ell(\mathfrak g)$ be its simple quotient. We assume that $W_\ell(\mathfrak g)$ is strongly rational, so that the category $\cC_\ell^W(\mathfrak g)$ of $W_\ell(\mathfrak g)$-modules is a semisimple (super)modular tensor category. Strong rationality is known for $\mathfrak g$ a simple Lie algebra and $\ell$  a non-degenerate admissible level \cite{Ar15a, Ar15b}, and for $\mathfrak g = \mathfrak{osp}_{1|2n}$ at certain levels \cite[\S 7]{CL-osp}. We also assume that the category $\cC_k^L(\mathfrak g)$ of ordinary $L_k(\mathfrak g)$-modules is a fusion category. In fact, for $k$ admissible, $\cC_k^L(\mathfrak{g})$ is always a semisimple braided monoidal category \cite{Arakawa:2012xrk, Creutzig:2017gpa}, and rigidity is known in types $ADE$ \cite{Creutzig:2019qje}, for most levels in type $C$ \cite{Creutzig:2022riy}, and for some levels in type $B$ \cite{CVL}. Similar results also hold for $\mathfrak{osp}_{1|2n}$ \cite{Creutzig:2022riy}.

In this setup, assume $A$ is a simple VO(S)A that is a conformal extension of $W_\ell(\mathfrak g) \otimes L_k(\tilde{\mathfrak g})$. By \cite[Theorem 5.5]{creutzig22gluing}, the category of grading-restricted $W_\ell(\mathfrak g) \otimes L_k(\tilde{\mathfrak g})$-modules is a vertex tensor category equivalent to the Deligne product $\cC_\ell^W(\mathfrak g) \boxtimes \cC_k^L(\tilde{\mathfrak g})$. In particular, it is a braided fusion category and so by Corollary \ref{thm:comm-exact-semisimple-VOA} and Theorem \ref{thm:rig-and-ss-for-super}, the category of grading-restricted $A$-modules is also braided fusion. Many hook-type $W$-(super)algebras satisfy these conditions as corollaries of the trialities of \cite{CL-sl, CL-osp}. Here are two examples:
\begin{enumerate}
    \item $A = W_{k}(\mathfrak{osp}_{2n|2m}, f_{\mathfrak{sp}_{2m}})$ for $k + m-n+1 = -\frac{n-m-1}{2(m+r)+1}$ and $r$ a positive integer such that $\gcd(2n+2r-1, 2m+2r+1)=1$. In this case, $A$ is a conformal extension of $W_s(\mathfrak{sp}_{2r})\otimes L_\ell(\mathfrak{so}_{2n})$ at the non-degenerate admissible level $s = -(r+1) + \frac{2m+2r+1}{2(2r+2n-1)}$ and the admissible level $\ell = -(2n-2) +\frac{2n+2r-1}{2m+2r+1}$. 
    
    \item $A = W_{k}(\mathfrak{so}_{2n+2m+1}, f_{\mathfrak{so}_{2m+1}})$ for $k + 2m+2n-1 = \frac{2n+2m+2r+1}{2(m+1)}$ and $r$ a positive integer such that  $\gcd(2n+2r-1, m+1)=1$. In this case, $A$ is a conformal extension of $W_s(\mathfrak{sp}_{2r})\otimes L_\ell(\mathfrak{so}_{2n})$ at the non-degenerate admissible level $s = -(r+1) + \frac{2n+2m+2r+1}{2(2r+2n-1)}$ and the admissible level $\ell = -(2n-2) +\frac{2n+2r-1}{2(m+1)}$.
\end{enumerate}
Many more examples can be read off from Appendices B, C, D  of \cite{CL-osp} and Corollary 6.5 of \cite{CL-sl}. 

The two listed examples follow from Theorems B.1 and B.2 of \cite{CL-osp} together with some standard VOA invariant theory:
 Let $A^k = W^{k}(\mathfrak{osp}_{2n|2m}, f_{\mathfrak{sp}_{2m}})$ or  $W^{k}(\mathfrak{so}_{2n+2m+1}, f_{\mathfrak{so}_{2m+1}})$ viewed as a family of vertex (super)algebras as in \cite{CL-coset}. The algebra $A^k$ has an affine vertex subalgebra $V^\ell(\mathfrak{so}_{2n})$, where $\ell = -2k-2n -1$ in the first case and $\ell = k+2m$ in the second case. Also, $A$ is the simple quotient of $A^k$ for $k = n-m-1-\frac{n-m-1}{2(m+r)+1}$  in the first case and for $k = -2m-2n+1+ \frac{2n+2m+2r+1}{2(m+1)}$ in the second case; we fix $k$ to be this number. 
 Let $C^k$ be the coset of $V^\ell(\mathfrak{so}_{2n})$ in $A^k$, and let $C$ be the corresponding coset in $A$. Then by \cite[Theorem 8.1]{CL-coset}, $C$ is a quotient of $C^k$. 
 Since $A$ is simple and $k+2n-2$ is positive, $C$ must be simple \cite[Theorem 4.1]{ACK}, and so by Theorems B.1 and B.2 of \cite{CL-osp},
  $C=W_s(\mathfrak{sp}_{2r})$ with $s$ as specified above. 
 Now since $W_s(\mathfrak{sp}_{2r})$ is rational, by \cite[Corollary 2.2]{ACKL} the coset of $W_s(\mathfrak{sp}_{2r})$ in $A$ is simple and hence a conformal extension of $L_\ell(\mathfrak{so}_{2n})$ \cite[Theorem 3.5]{AEM}.

 In these two examples, we \textit{cannot} use the earlier result \cite[Theorem 5.12]{creutzig22gluing} to show that the category of $A$-modules is a braided fusion category, because we do not know in general whether $W_s(\mathfrak{sp}_{2r})$ and $L_\ell(\mathfrak{so}_{2n})$ form a dual pair in $A$. That is, we do not know whether the commutant of $W_s(\mathfrak{sp}_{2r})$ is $L_\ell(\mathfrak{so}_{2n})$ itself or some conformal extension.
 \end{example}

 \begin{example} (\cite[\S 8]{ACF})
Let $k-1$ be a non-degenerate admissible level for a simply-laced simple Lie algebra $\mathfrak g$, and define $\ell$ by 
\[
\frac{1}{k + h^\vee}+ \frac{1}{\ell + h^\vee} =1
\]
where $h^\vee$ is the dual Coxeter number of $\mathfrak g$. Then $k, \ell$ are non-degenerate admissible and $A = W_{k-1}(\mathfrak g) \otimes L_1(\mathfrak g)$ is a conformal extension of $W_{k}(\mathfrak g) \otimes W_{\ell}(\mathfrak g)$, provided $L_1(\mathfrak{g})$ is equipped with the Urod conformal vector of \cite[\S 6]{ACF}. Thus by Theorem \ref{thm:VOA_ext_rational}, $L_1(\mathfrak{g})$ is self-contragredient with respect to the Urod conformal vector, and the non-local module category $\mathrm{Rep}(W_{k}(\mathfrak g) \otimes W_{\ell}(\mathfrak g))_A$ is a fusion category.
 \end{example}

\begin{example} (\cite[\S 6.5.3]{CDGG})
The vertex operator superalgebra $V_{SF}$ of a pair of symplectic fermions \cite{Kausch2} is one of the best-known non-rational $C_2$-cofinite vertex operator (super)algebras. The vertex operator superalgebra $A$ of four free fermions is an extension of $V_{SF} \otimes V_{SF}$. As a $V_{SF} \otimes V_{SF}$-module, this extension is indecomposable, is projective as a module for either of the commuting symplectic fermion subalgebras, and has four composition factors, each isomorphic to $V_{SF} \otimes V_{SF}$ (up to parity). See \cite[\S 6.5.3]{CDGG} for details. In particular, condition (b) of Corollary \ref{cor:com-exact-5-VOA} holds, so $A$ is exact. 
\end{example}

We also give an example of Theorem \ref{rigidity-of-C-from-CAloc}, giving an easier proof of rigidity for the $p=2$ triplet VOA then is given in \cite{Tsuchiya:2012ru}. Further examples of Theorem \ref{rigidity-of-C-from-CAloc} will appear in \cite{CMY24}.

\begin{example}\label{ex:Wp}
Let $V$ be the $p=2$ triplet VOA \cite{Kausch, Adamovic:2007er, Tsuchiya:2012ru}. It has four inequivalent simple modules $ X^\pm_i$ for $i=1,2$, where $X_1^+=V$. The modules $X^\pm_2$ are projective, and $V$ is a subalgebra of the lattice VOA $A = V_{2\mathbb Z}$. Let $\cC$ be the category of $V$-modules and let $F_A:\cC\rightarrow\cC_A$ and $U:\cC_A\rightarrow\cC$ be the induction and restriction functors as previously. The VOA $A = V_{2\mathbb Z}$ is strongly rational (for a different conformal vector than that of $V$) and has four inequivalent simple modules $V_i:=V_{\frac{i}{2}+2\ZZ}$ for $i=0, 1, 2, 3$. The modules $X^\pm_1$ allow the following extensions (using $U$ to indicate that these are exact sequences in $\cC$),
\[
0 \rightarrow X^+_1 \rightarrow U(V_0) \rightarrow X^-_1 \rightarrow 0, \qquad 
0 \rightarrow X^-_1 \rightarrow U(V_2) \rightarrow X^+_1 \rightarrow 0,
\]
while  $U(V_1) \cong X^+_2$ and $U(V_3) \cong X^-_2$. 
The fusion rules of $A = V_0$ with the simple objects of $\cC$ are:
\[
A \boxtimes_V X^+_1 = U(V_0), \qquad 
A \boxtimes_V X^-_1 = U(V_2), \qquad 
A \boxtimes_V X^\pm_2 = U(V_1) \oplus U(V_3).
\]
It is shown in \cite{Tsuchiya:2012ru} that $\cC$ is rigid; we now reprove this from the above fusion rules and Theorem \ref{rigidity-of-C-from-CAloc}.

First, we need to show that any simple object $M\in\cC_A$ is local. Since $M$ contains some simple $V$-submodule $X^\pm_i$, Frobenius reciprocity yields a non-zero (and thus surjective) $\cC_A$-morphism $F_A(X_i^\pm)\rightarrow M$. Thus it is enough to show that every simple composition factor of $F_A(X^\pm_i)$ for $i=1, 2$ is local. Clearly $F_A(X^+_1) \cong A$. Next, since $U(F_A(X^-_1)) \cong U(V_2)$ and since $\mathrm{Hom}_{\cC}(X^-_1, U(V_2)) \cong \mathrm{Hom}_{\cC_A}(F_A(X^-_1), V_2)$ by Frobenius reciprocity, it follows that $F_A(X^-_1) \cong V_2$. Similarly, there are $\cC_A$-surjections $f_{\pm}: F_A(X_2^\pm)\rightarrow V_{2\mp 1}$ since $\mathrm{Hom}_{\cC}(X^\pm_2, U(V_{2\mp 1})) \cong \mathrm{Hom}_{\cC_A}(F_A(X^\pm_2), V_{2\mp 1})$. From the fusion rules,  $U(\mathrm{Ker}\, f_\pm) \cong U(V_{2\pm 1})$. Since $V_1$ and $V_3$ are local, the  twist $\theta=e^{2\pi i L_0}$ acts on them by $\cC_A$-morphisms \cite[Lemma 2.81]{creutzig2017tensor}, and since $V_1$ and $V_3$ are also simple, $\theta$ therefore acts on them by scalars. Since $L_0$ is also the Virasoro zero-mode for the subalgebra $V$ the twist also acts by scalars on $\mathrm{Ker}\, f_\pm$. This means that the twist on $\cC$ gives a $\cC_A$-morphism on $\mathrm{Ker}\, f_\pm$ and thus $\mathrm{Ker}\, f_\pm$ are local \cite[Lemma 2.81]{creutzig2017tensor} and hence isomorphic to $V_3$ and $V_1$. This proves that every simple object of $\cC_A$ is local, and thus $\cC_A$ is rigid by Theorem \ref{thm:GVto rigid}.

Next, we need to show that any non-zero $\cC_A$-morphism $F_A(X')\rightarrow F_A(X)^*$ is an isomorphism for any simple $X\in\cC$. This holds for $X_1^{\pm}$ because $F_A((X_1^\pm)')\cong F_A(X_1^\pm)\cong V_{1\mp 1}$ and $F_A(X_1^\pm)^*\cong V_{1\mp 1}^*\cong V_{1\mp 1}$ are both simple. For $X_2^\pm$, the previous paragraph together with $(X_2^\pm)'\cong X_2^\pm$ yields a short exact sequence
\[
0 \rightarrow V_{2 \pm 1} \rightarrow F_A((X^\pm_2)') \rightarrow V_{2 \mp 1} \rightarrow 0.
\]
This sequence does not split because by Frobenius reciprocity,
\begin{equation*}
    \Hom_{\cC_A}(F_A(X^\pm_2), V_{2\pm 1}) =\Hom_\cC(X^\pm_2, X^{\mp}_2) = 0.
\end{equation*}
Dualizing this exact sequence and using $V_{2\pm 1}^*\cong V_{2\mp 1}$ yields a non-split exact sequence
\[
0 \rightarrow V_{2 \pm 1} \rightarrow F_A(X^\pm_2)^* \rightarrow V_{2 \mp 1} \rightarrow 0.
\]
Now any non-zero homomorphism $f:F_A((X_2^\pm)')\rightarrow F_A(X_2^\pm)^*$ in $\cC_A$ is injective because otherwise $\mathrm{Im}\,f$ would be isomorphic to $V_{2\mp 1}$, but $V_{2\mp 1}$ is not a submodule of $F_A(X_2^\pm)^*$. Similarly, $f$ is surjective, so $f$ is an isomorphism, as required.

Finally, for condition (c) in Theorem \ref{rigidity-of-C-from-CAloc}, note that the only $V$-submodule of $A$ which is not contained in $V$ is $A$ itself, and that $\id_A\otimes\id_Z: A\otimes Z\rightarrow A\otimes Z$ is injective for any $V$-module $Z$. Thus we just need a $V$-module $Z$ such that $c_{Z,A}\neq c_{A,Z}^{-1}$, and we claim that $Z=X_2^+$ suffices. Indeed, because $\cC_A^{\loc}$ is semisimple and $F_A(X_2^+)$ is not semisimple, $F_A(X_2^+)$ is not local and therefore $c_{X_2^+,A}\neq c_{A,X_2^+}^{-1}$ (see for example \cite[Proposition 2.65]{creutzig2017tensor}).
Thus $\cC$ is rigid by Theorem \ref{rigidity-of-C-from-CAloc}.
\end{example}


\section{Statements and Declarations}

There are no financial interests connected to this work. There is no conflict of interest and we kept the highest ethical standards and there is no associated data.

{\small
\bibliographystyle{alpha}
\bibliography{references}}

\newcommand{\etalchar}[1]{$^{#1}$}
\begin{thebibliography}{NOHRCW24}

\bibitem[Abe07]{Abe}
Toshiyuki Abe.
\newblock {A $\mathbb{Z}_2$-orbifold model of the symplectic fermionic vertex operator superalgebra}.
\newblock {\em Math. Z.}, 255(4):755--792, 2007.

\bibitem[ACF22]{ACF}
Tomoyuki Arakawa, Thomas Creutzig, and Boris Feigin.
\newblock Urod algebras and translation of {$W$}-algebras.
\newblock {\em Forum Math. Sigma}, 10:Paper No. e33, 31 pp., 2022.

\bibitem[ACK24]{ACK}
Tomoyuki Arakawa, Thomas Creutzig, and Kazuya Kawasetsu.
\newblock On lisse non-admissible minimal and principal {W}-algebras.
\newblock arXiv:2408.04584, 2024.

\bibitem[ACK25]{Arakawa:2023msa}
Tomoyuki Arakawa, Thomas Creutzig, and Kazuya Kawasetsu.
\newblock {Weight representations of affine Kac-Moody algebras and small quantum groups}.
\newblock {\em Adv. Math.}, 477:Paper No. 110365, 48 pp., 2025.

\bibitem[ACKL17]{ACKL}
Tomoyuki Arakawa, Thomas Creutzig, Kazuya Kawasetsu, and Andrew~R. Linshaw.
\newblock Orbifolds and cosets of minimal {$\mathcal{W}$}-algebras.
\newblock {\em Comm. Math. Phys.}, 355(1):339--372, 2017.

\bibitem[Ada19]{Ad-IQHR}
Dra\v{z}en Adamovi\'{c}.
\newblock {Realizations of simple affine vertex algebras and their modules: the cases $\widehat{sl(2)}$ and $\widehat{osp(1,2)}$}.
\newblock {\em Comm. Math. Phys.}, 366(3):1025--1067, 2019.

\bibitem[ALSW25]{ALSW}
Robert Allen, Simon Lentner, Christoph Schweigert, and Simon Wood.
\newblock Duality structures for module categories of vertex operator algebras and the {F}eigin {F}uchs boson.
\newblock {\em Selecta Math. (N.S.)}, 31(2):Paper No. 36, 57 pp., 2025.

\bibitem[AM08]{Adamovic:2007er}
Dra\v{z}en Adamovi\'{c} and Antun Milas.
\newblock On the triplet vertex algebra {$\mathcal W(p)$}.
\newblock {\em Adv. Math.}, 217(6):2664--2699, 2008.

\bibitem[Ara15a]{Ar15a}
Tomoyuki Arakawa.
\newblock Associated varieties of modules over {K}ac-{M}oody algebras and {$C_2$}-cofiniteness of {$W$}-algebras.
\newblock {\em Int. Math. Res. Not. IMRN}, {}(22):11605--11666, 2015.

\bibitem[Ara15b]{Ar15b}
Tomoyuki Arakawa.
\newblock Rationality of {$W$}-algebras: principal nilpotent cases.
\newblock {\em Ann. of Math. (2)}, 182(2):565--604, 2015.

\bibitem[Ara16]{Arakawa:2012xrk}
Tomoyuki Arakawa.
\newblock {Rationality of admissible affine vertex algebras in the category ${\mathcal{O}}$}.
\newblock {\em Duke Math. J.}, 165(1):67--93, 2016.

\bibitem[AvE23]{AvE}
Tomoyuki Arakawa and Jethro van Ekeren.
\newblock Rationality and fusion rules of exceptional {$\mathcal {W}$}-algebras.
\newblock {\em J. Eur. Math. Soc. (JEMS)}, 25(7):2763--2813, 2023.

\bibitem[AvEM24]{AEM}
Tomoyuki Arakawa, Jethro van Ekeren, and Anne Moreau.
\newblock Singularities of nilpotent {S}lodowy slices and collapsing levels of {$W$}-algebras.
\newblock {\em Forum Math. Sigma}, 12:Paper No. e95, 92, 2024.

\bibitem[AW22]{Allen:2020kkt}
Robert Allen and Simon Wood.
\newblock Bosonic ghostbusting: the bosonic ghost vertex algebra admits a logarithmic module category with rigid fusion.
\newblock {\em Comm. Math. Phys.}, 390(2):959--1015, 2022.

\bibitem[BD13]{boyarchenko2013duality}
Mitya Boyarchenko and Vladimir Drinfeld.
\newblock A duality formalism in the spirit of {G}rothendieck and {V}erdier.
\newblock {\em Quantum Topol.}, 4(4):447--489, 2013.

\bibitem[BN24]{Ballin:2022rto}
Andrew Ballin and Wenjun Niu.
\newblock 3{D} mirror symmetry and the {$\beta \gamma $} {VOA}.
\newblock {\em Commun. Contemp. Math.}, 26(1):Paper No. 2250069, 52 pp., 2024.

\bibitem[CDGG24]{CDGG}
Thomas Creutzig, Tudor Dimofte, Niklas Garner, and Nathan Geer.
\newblock A {QFT} for non-semisimple {TQFT}.
\newblock {\em Adv. Theor. Math. Phys.}, 28(1):161--405, 2024.

\bibitem[CG17]{creutzig_gannon}
Thomas Creutzig and Terry Gannon.
\newblock Logarithmic conformal field theory, log-modular tensor categories and modular forms.
\newblock {\em J. Phys. A}, 50(40):Paper No. 404004, 37 pp., 2017.

\bibitem[CGL24]{Creutzig:2022riy}
Thomas Creutzig, Naoki Genra, and Andrew Linshaw.
\newblock Ordinary modules for vertex algebras of {$\mathfrak {osp}_{1|2n}$}.
\newblock {\em J. Reine Angew. Math.}, 817:1--31, 2024.

\bibitem[CGPM14]{CGPM}
Francesco Costantino, Nathan Geer, and Bertrand Patureau-Mirand.
\newblock Quantum invariants of 3-manifolds via link surgery presentations and non-semi-simple categories.
\newblock {\em J. Topol.}, 7(4):1005--1053, 2014.

\bibitem[CHY18]{Creutzig:2017gpa}
Thomas Creutzig, Yi-Zhi Huang, and Jinwei Yang.
\newblock {Braided tensor categories of admissible modules for affine Lie algebras}.
\newblock {\em Comm. Math. Phys.}, 362(3):827--854, 2018.

\bibitem[CJOH{\etalchar{+}}21]{Creutzig:2020zvv}
Thomas Creutzig, Cuipo Jiang, Florencia Orosz~Hunziker, David Ridout, and Jinwei Yang.
\newblock Tensor categories arising from the {V}irasoro algebra.
\newblock {\em Adv. Math.}, 380:Paper No. 107601, 35 pp., 2021.

\bibitem[CKL20]{Creutzig:2015buk}
Thomas Creutzig, Shashank Kanade, and Andrew~R. Linshaw.
\newblock Simple current extensions beyond semi-simplicity.
\newblock {\em Commun. Contemp. Math.}, 22(1):Paper No. 1950001, 49 pp., 2020.

\bibitem[CKL22]{CVL}
Thomas Creutzig, Vladimir Kovalchuk, and Andrew~R. Linshaw.
\newblock Generalized parafermions of orthogonal type.
\newblock {\em J. Algebra}, 593:178--192, 2022.

\bibitem[CKL24]{CKL24}
Thomas Creutzig, Vladimir Kovalchuk, and Andrew~R. Linshaw.
\newblock Building blocks for {$W$}-algebras of classical types.
\newblock arXiv:2409.03465, 2024.

\bibitem[CKLR19]{Creutzig:2016ehb}
Thomas Creutzig, Shashank Kanade, Andrew~R. Linshaw, and David Ridout.
\newblock Schur-{W}eyl duality for {H}eisenberg cosets.
\newblock {\em Transform. Groups}, 24(2):301--354, 2019.

\bibitem[CKM22]{creutzig22gluing}
Thomas Creutzig, Shashank Kanade, and Robert McRae.
\newblock Gluing vertex algebras.
\newblock {\em Adv. Math.}, 396:Paper No. 108174, 72 pp., 2022.

\bibitem[CKM24]{creutzig2017tensor}
Thomas Creutzig, Shashank Kanade, and Robert McRae.
\newblock Tensor categories for vertex operator superalgebra extensions.
\newblock {\em Mem. Amer. Math. Soc.}, 295(1472):vi+181 pp., 2024.

\bibitem[CL19]{CL-coset}
Thomas Creutzig and Andrew~R. Linshaw.
\newblock Cosets of affine vertex algebras inside larger structures.
\newblock {\em J. Algebra}, 517:396--438, 2019.

\bibitem[CL22a]{CL-sl}
Thomas Creutzig and Andrew~R. Linshaw.
\newblock Trialities of {$\mathcal{W}$}-algebras.
\newblock {\em Camb. J. Math.}, 10(1):69--194, 2022.

\bibitem[CL22b]{CL-osp}
Thomas Creutzig and Andrew~R. Linshaw.
\newblock Trialities of orthosymplectic {$\mathcal{W}$}-algebras.
\newblock {\em Adv. Math.}, 409:Paper No. 108678, 79 pp., 2022.

\bibitem[CLR23]{CLR}
Thomas Creutzig, Simon Lentner, and Matthew Rupert.
\newblock An algebraic theory for logarithmic {K}azhdan-{L}usztig correspondences.
\newblock arXiv:2306.11492, 2023.

\bibitem[CMOHY24]{CORY24}
Thomas Creutzig, Robert McRae, Florencia Orosz~Hunziker, and Jinwei Yang.
\newblock {$N=1$ super Virasoro tensor categories}.
\newblock arXiv:2412.18127, 2024.

\bibitem[CMY21]{Creutzig:2020qvs}
Thomas Creutzig, Robert McRae, and Jinwei Yang.
\newblock On ribbon categories for singlet vertex algebras.
\newblock {\em Comm. Math. Phys.}, 387(2):865--925, 2021.

\bibitem[CMY22]{Creutzig:2020zom}
Thomas Creutzig, Robert McRae, and Jinwei Yang.
\newblock Tensor structure on the {K}azhdan-{L}usztig category for affine {$\mathfrak{gl}(1|1)$}.
\newblock {\em Int. Math. Res. Not. IMRN}, {}(16):12462--12515, 2022.

\bibitem[CMY23a]{Creutzig:2022lep}
Thomas Creutzig, Robert McRae, and Jinwei Yang.
\newblock {Ribbon tensor structure on the full representation categories of the singlet vertex algebras}.
\newblock {\em Adv. Math.}, 413:Paper No. 108828, 79 pp., 2023.

\bibitem[CMY23b]{Creutzig:2022ugv}
Thomas Creutzig, Robert McRae, and Jinwei Yang.
\newblock Rigid tensor structure on big module categories for some {$W$}-(super)algebras in type {$A$}.
\newblock {\em Comm. Math. Phys.}, 404(1):339--400, 2023.

\bibitem[CMY24]{CMY24}
Thomas Creutzig, Robert McRae, and Jinwei Yang.
\newblock {Ribbon categories of weight modules for affine $\mathfrak{sl}_2$ at admissible levels}.
\newblock arXiv:2411.11386, 2024.

\bibitem[CN24]{Cre-Niu}
Thomas Creutzig and Wenjun Niu.
\newblock {Kazhdan-Lusztig correspondence for vertex operator superalgebras from abelian gauge theories}.
\newblock arXiv:2403.02403, 2024.

\bibitem[CR24]{CR24}
Thomas Creutzig and David Ridout.
\newblock In preparation, 2024.

\bibitem[Cre19]{Creutzig:2019qje}
Thomas Creutzig.
\newblock Fusion categories for affine vertex algebras at admissible levels.
\newblock {\em Selecta Math. (N.S.)}, 25(2):Paper No. 27, 21 pp., 2019.

\bibitem[Cre24a]{C24}
Thomas Creutzig.
\newblock {Resolving Verlinde's formula of logarithmic CFT}.
\newblock arXiv:2411.11383, 2024.

\bibitem[Cre24b]{Creutzig:2023rlw}
Thomas Creutzig.
\newblock {Tensor categories of weight modules of $\widehat{\mathfrak{sl}}_2$ at admissible level}.
\newblock {\em J. Lond. Math. Soc. (2)}, 110(6):Paper No. e70037, 38 pp., 2024.

\bibitem[CSZ25]{coulembier2025simple}
Kevin Coulembier, Mateusz Stroi\'{n}ski, and Tony Zorman.
\newblock Simple algebras and exact module categories.
\newblock arXiv:2501.06629, 2025.

\bibitem[CY21]{CJ}
Thomas Creutzig and Jinwei Yang.
\newblock Tensor categories of affine {L}ie algebras beyond admissible levels.
\newblock {\em Math. Ann.}, 380(3-4):1991--2040, 2021.

\bibitem[DGH98]{DGH}
Chongying Dong, Robert~L. Griess, Jr., and Gerald H\"{o}hn.
\newblock Framed vertex operator algebras, codes and the moonshine module.
\newblock {\em Comm. Math. Phys.}, 193(2):407--448, 1998.

\bibitem[DJX13]{DJX}
Chongying Dong, Xiangyu Jiao, and Feng Xu.
\newblock Quantum dimensions and quantum {G}alois theory.
\newblock {\em Trans. Amer. Math. Soc.}, 365(12):6441--6469, 2013.

\bibitem[DMNO13]{davydov2013witt}
Alexei Davydov, Michael M\"{u}ger, Dmitri Nikshych, and Victor Ostrik.
\newblock The {W}itt group of non-degenerate braided fusion categories.
\newblock {\em J. Reine Angew. Math.}, 677:135--177, 2013.

\bibitem[DRGG{\etalchar{+}}22]{de20223}
Marco De~Renzi, Azat~M. Gainutdinov, Nathan Geer, Bertrand Patureau-Mirand, and Ingo Runkel.
\newblock 3-{D}imensional {TQFT}s from non-semisimple modular categories.
\newblock {\em Selecta Math. (N.S.)}, 28(2):Paper No. 42, 60 pp., 2022.

\bibitem[DSPS19]{douglas2019balanced}
Christopher~L. Douglas, Christopher Schommer-Pries, and Noah Snyder.
\newblock The balanced tensor product of module categories.
\newblock {\em Kyoto J. Math.}, 59(1):167--179, 2019.

\bibitem[DSPS20]{douglas2018dualizable}
Christopher~L. Douglas, Christopher Schommer-Pries, and Noah Snyder.
\newblock Dualizable tensor categories.
\newblock {\em Mem. Amer. Math. Soc.}, 268(1308):vii+88 pp., 2020.

\bibitem[EGNO15]{etingof2016tensor}
Pavel Etingof, Shlomo Gelaki, Dmitri Nikshych, and Victor Ostrik.
\newblock {\em Tensor categories}, volume 205 of {\em Mathematical Surveys and Monographs}.
\newblock American Mathematical Society, Providence, RI, 2015.

\bibitem[ENO05]{etingof2005fusion}
Pavel Etingof, Dmitri Nikshych, and Viktor Ostrik.
\newblock On fusion categories.
\newblock {\em Ann. of Math. (2)}, 162(2):581--642, 2005.

\bibitem[EO04]{etingof2004finite}
Pavel Etingof and Viktor Ostrik.
\newblock Finite tensor categories.
\newblock {\em Mosc. Math. J.}, 4(3):627--654, 782--783, 2004.

\bibitem[EO21]{etingof2021frobenius}
Pavel Etingof and Victor Ostrik.
\newblock On the {F}robenius functor for symmetric tensor categories in positive characteristic.
\newblock {\em J. Reine Angew. Math.}, 773:165--198, 2021.

\bibitem[EP24]{etingof2024rigidity}
Pavel Etingof and David Penneys.
\newblock Rigidity of non-negligible objects of moderate growth in braided categories.
\newblock arXiv:2412.17681, 2024.

\bibitem[FHL93]{FHL}
Igor Frenkel, Yi-Zhi Huang, and James Lepowsky.
\newblock On axiomatic approaches to vertex operator algebras and modules.
\newblock {\em Mem. Amer. Math. Soc.}, 104(494):viii+64 pp., 1993.

\bibitem[FLM88]{FLM}
Igor Frenkel, James Lepowsky, and Arne Meurman.
\newblock {\em Vertex operator algebras and the {M}onster}, volume 134 of {\em Pure and Applied Mathematics}.
\newblock Academic Press, Inc., Boston, MA, 1988.

\bibitem[Gan23]{gannon}
Terry Gannon.
\newblock {Exotic quantum subgroups and extensions of affine Lie algebra VOAs -- part I}.
\newblock arXiv:2301.07287, 2023.

\bibitem[GRW09]{GRW}
Matthias~R. Gaberdiel, Ingo Runkel, and Simon Wood.
\newblock Fusion rules and boundary conditions in the {$c=0$} triplet model.
\newblock {\em J. Phys. A}, 42(32):Paper No. 325403, 43 pp., 2009.

\bibitem[HKL15]{Huang:2014ixa}
Yi-Zhi Huang, Alexander Kirillov, Jr., and James Lepowsky.
\newblock {Braided tensor categories and extensions of vertex operator algebras}.
\newblock {\em Comm. Math. Phys.}, 337(3):1143--1159, 2015.

\bibitem[HLZ10a]{HLZ2}
Yi-Zhi Huang, James Lepowsky, and Lin Zhang.
\newblock {Logarithmic tensor category theory for generalized modules for a conformal vertex algebra, II: Logarithmic formal calculus and properties of logarithmic intertwining operators}.
\newblock arXiv:1012.4196, 2010.

\bibitem[HLZ10b]{HLZ3}
Yi-Zhi Huang, James Lepowsky, and Lin Zhang.
\newblock {Logarithmic tensor category theory for generalized modules for a conformal vertex algebra, III: Intertwining maps and tensor product bifunctors}.
\newblock arXiv:1012.4197, 2010.

\bibitem[HLZ10c]{HLZ4}
Yi-Zhi Huang, James Lepowsky, and Lin Zhang.
\newblock {Logarithmic tensor category theory for generalized modules for a conformal vertex algebra, IV: Constructions of tensor product bifunctors and the compatibility conditions}.
\newblock arXiv:1012.4198, 2010.

\bibitem[HLZ10d]{HLZ5}
Yi-Zhi Huang, James Lepowsky, and Lin Zhang.
\newblock {Logarithmic tensor category theory for generalized modules for a conformal vertex algebra, V: Convergence condition for intertwining maps and the corresponding compatibility condition}.
\newblock arXiv:1012.4199, 2010.

\bibitem[HLZ10e]{HLZ6}
Yi-Zhi Huang, James Lepowsky, and Lin Zhang.
\newblock {Logarithmic tensor category theory for generalized modules for a conformal vertex algebra, VI: Expansion condition, associativity of logarithmic intertwining operators, and the associativity isomorphisms}.
\newblock arXiv:1012.4202, 2010.

\bibitem[HLZ11a]{HLZ7}
Yi-Zhi Huang, James Lepowsky, and Lin Zhang.
\newblock {Logarithmic tensor category theory for generalized modules for a conformal vertex algebra, VII: Convergence and extension properties and applications to expansion for intertwining maps}.
\newblock arXiv:1110.1929, 2011.

\bibitem[HLZ11b]{HLZ8}
Yi-Zhi Huang, James Lepowsky, and Lin Zhang.
\newblock {Logarithmic tensor category theory for generalized modules for a conformal vertex algebra, VIII: Braided tensor category structure on categories of generalized modules for a conformal vertex algebra}.
\newblock arXiv:1110.1931, 2011.

\bibitem[HLZ14]{HLZ1}
Yi-Zhi Huang, James Lepowsky, and Lin Zhang.
\newblock Logarithmic tensor category theory for generalized modules for a conformal vertex algebra, {I}: Introduction and strongly graded algebras and their generalized modules.
\newblock In {\em Conformal field theories and tensor categories}, Math. Lect. Peking Univ., pages 169--248. Springer, Heidelberg, 2014.

\bibitem[Hua05a]{Huang:diff-eqns}
Yi-Zhi Huang.
\newblock {Differential equations and intertwining operators}.
\newblock {\em Commun. Contemp. Math.}, 7(3):375--400, 2005.

\bibitem[Hua05b]{huang2005vertex}
Yi-Zhi Huang.
\newblock Vertex operator algebras, the {V}erlinde conjecture, and modular tensor categories.
\newblock {\em Proc. Natl. Acad. Sci. USA}, 102(15):5352--5356, 2005.

\bibitem[Hua08]{Huang:Rig_Mod}
Yi-Zhi Huang.
\newblock Rigidity and modularity of vertex tensor categories.
\newblock {\em Commun. Contemp. Math.}, 10:871--911, 2008.

\bibitem[Hua09]{huangC2}
Yi-Zhi Huang.
\newblock Cofiniteness conditions, projective covers and the logarithmic tensor product theory.
\newblock {\em J. Pure Appl. Algebra}, 213(4):458--475, 2009.

\bibitem[Hua17]{huang:applicability}
Yi-Zhi Huang.
\newblock On the applicability of logarithmic tensor category theory.
\newblock arXiv:1702.00133, 2017.

\bibitem[Hua18]{Huang:Tw-Intw-Ops}
Yi-Zhi Huang.
\newblock Intertwining operators among twisted modules associated to not-necessarily-commuting automorphisms.
\newblock {\em J. Algebra}, 493:346--380, 2018.

\bibitem[Kau91]{Kausch}
Horst~G. Kausch.
\newblock Extended conformal algebras generated by a multiplet of primary fields.
\newblock {\em Phys. Lett. B}, 259(4):448--455, 1991.

\bibitem[Kau00]{Kausch2}
Horst~G. Kausch.
\newblock Symplectic fermions.
\newblock {\em Nuclear Phys. B}, 583(3):513--541, 2000.

\bibitem[KL93a]{kazhdan1993tensor}
David Kazhdan and George Lusztig.
\newblock Tensor structures arising from affine {L}ie algebras. {I}.
\newblock {\em J. Amer. Math. Soc.}, 6(4):905--947, 1993.

\bibitem[KL93b]{kazhdan1993tensor2}
David Kazhdan and George Lusztig.
\newblock Tensor structures arising from affine {L}ie algebras. {II}.
\newblock {\em J. Amer. Math. Soc.}, 6(4):949--1011, 1993.

\bibitem[KL94a]{kazhdan1994tensor}
David Kazhdan and George Lusztig.
\newblock Tensor structures arising from affine {L}ie algebras. {III}.
\newblock {\em J. Amer. Math. Soc.}, 7(2):335--381, 1994.

\bibitem[KL94b]{kazhdan1994tensor2}
David Kazhdan and George Lusztig.
\newblock Tensor structures arising from affine {L}ie algebras. {IV}.
\newblock {\em J. Amer. Math. Soc.}, 7(2):383--453, 1994.

\bibitem[KO02]{kirillov2002q}
Alexander Kirillov, Jr. and Viktor Ostrik.
\newblock On a {$q$}-analogue of the {M}c{K}ay correspondence and the {$ADE$} classification of {$\mathfrak{sl}_2$} conformal field theories.
\newblock {\em Adv. Math.}, 171(2):183--227, 2002.

\bibitem[Li94]{Li_bilinear}
Haisheng Li.
\newblock Symmetric invariant bilinear forms on vertex operator algebras.
\newblock {\em J. Pure Appl. Algebra}, 96(3):279--297, 1994.

\bibitem[Li99]{Li-fin}
Haisheng Li.
\newblock Some finiteness properties of regular vertex operator algebras.
\newblock {\em J. Algebra}, 212(2):495--514, 1999.

\bibitem[LL04]{LL}
James Lepowsky and Haisheng Li.
\newblock {\em Introduction to vertex operator algebras and their representations}, volume 227 of {\em Progress in Mathematics}.
\newblock Birkh\"{a}user Boston, Inc., Boston, MA, 2004.

\bibitem[LW23]{laugwitz2023constructing}
Robert Laugwitz and Chelsea Walton.
\newblock Constructing non-semisimple modular categories with local modules.
\newblock {\em Comm. Math. Phys.}, 403(3):1363--1409, 2023.

\bibitem[Lyu95]{Lyu}
Volodymyr~V. Lyubashenko.
\newblock Invariants of {$3$}-manifolds and projective representations of mapping class groups via quantum groups at roots of unity.
\newblock {\em Comm. Math. Phys.}, 172(3):467--516, 1995.

\bibitem[McR20]{McRaeCptOrb}
Robert McRae.
\newblock On the tensor structure of modules for compact orbifold vertex operator algebras.
\newblock {\em Math. Z.}, 296(1-2):409--452, 2020.

\bibitem[McR21]{McRae2021}
Robert McRae.
\newblock {On rationality for $C_2$-cofinite vertex operator algebras}.
\newblock arXiv:2108.01898, 2021.

\bibitem[McR22]{McRaeSS}
Robert McRae.
\newblock On semisimplicity of module categories for finite non-zero index vertex operator subalgebras.
\newblock {\em Lett. Math. Phys.}, 112(2):Paper No. 25, 28 pp., 2022.

\bibitem[McR23]{mcrae:deligne}
Robert McRae.
\newblock Deligne tensor products of categories of modules for vertex operator algebras.
\newblock arXiv:2304.14023, 2023.

\bibitem[McR24]{Mc2}
Robert McRae.
\newblock A general mirror equivalence theorem for coset vertex operator algebras.
\newblock {\em Sci. China Math.}, 67(10):2237–2282, 2024.

\bibitem[Miy04]{miyamoto2004modular}
Masahiko Miyamoto.
\newblock Modular invariance of vertex operator algebras satisfying {$C_2$}-cofiniteness.
\newblock {\em Duke Math. J.}, 122(1):51--91, 2004.

\bibitem[Miy14]{miyamoto:C1-fus-prod}
Masahiko Miyamoto.
\newblock {$C_1$-cofiniteness and fusion products for vertex operator algebras}.
\newblock In {\em {Conformal field theories and tensor categories}}, Math. Lect. Peking Univ., pages 271--279. Springer, Heidelberg, 2014.

\bibitem[ML98]{mac2013categories}
Saunders Mac~Lane.
\newblock {\em Categories for the working mathematician}, volume~5 of {\em Graduate Texts in Mathematics}.
\newblock Springer-Verlag, New York, second edition, 1998.

\bibitem[MS89]{Moore:1988qv}
Gregory Moore and Nathan Seiberg.
\newblock Classical and quantum conformal field theory.
\newblock {\em Comm. Math. Phys.}, 123(2):177--254, 1989.

\bibitem[MS24]{McR-Sop}
Robert McRae and Valerii Sopin.
\newblock {Fusion and (non)-rigidity of Virasoro Kac modules in logarithmic minimal models at $(p,q)$-central charge}.
\newblock {\em Phys. Scr.}, 99(3):Paper No. 035233, 42 pp., 2024.

\bibitem[MY23]{MY-c25}
Robert McRae and Jinwei Yang.
\newblock An {$\mathfrak{sl}_2$}-type tensor category for the {V}irasoro algebra at central charge $25$ and applications.
\newblock {\em Math. Z.}, 303(2):Paper No. 32, 40 pp., 2023.

\bibitem[MY25a]{McRae:2023ado}
Robert McRae and Jinwei Yang.
\newblock The nonsemisimple {K}azhdan-{L}usztig category for affine {$\mathfrak{sl}_2$} at admissible levels.
\newblock {\em Proc. Lond. Math. Soc. (3)}, 130(4):Paper No. e70043, 83 pp., 2025.

\bibitem[MY25b]{MY-cp1}
Robert McRae and Jinwei Yang.
\newblock {Structure of Virasoro tensor categories at central charge $13-6p-6p^{-1}$ for integers $p>1$}.
\newblock {\em Trans. Amer. Math. Soc.}, 378(10):7451--7509, 2025.

\bibitem[Nah94]{Nahm}
Werner Nahm.
\newblock {Quasi-rational fusion products}.
\newblock {\em Internat. J. Modern Phys. B}, 8(25-26):3693--3702, 1994.

\bibitem[Nak23]{Nakano}
Hiromu Nakano.
\newblock {Fusion rules for the triplet $W$-algebra $\mathcal{W}_{p+,p-}$}.
\newblock arXiv:2308.15954, 2023.

\bibitem[NOHRCW24]{Flor24}
Hiromu Nakano, Florencia Orosz~Hunziker, Ana Ros~Camacho, and Simon Wood.
\newblock Fusion rules and rigidity for weight modules over the simple admissible affine $\mathfrak{sl}_2$ and $\mathcal{N}=2$ superconformal vertex operator superalgebras.
\newblock arXiv:2411.11387, 2024.

\bibitem[Ost03]{ostrik2003module}
Victor Ostrik.
\newblock Module categories, weak {H}opf algebras and modular invariants.
\newblock {\em Transform. Groups}, 8(2):177--206, 2003.

\bibitem[Par77]{pareigis1977non}
Bodo Pareigis.
\newblock Non-additive ring and module theory. {II}. {${\mathcal C}$}-categories, {${\mathcal C}$}-functors and {${\mathcal C}$}-morphisms.
\newblock {\em Publ. Math. Debrecen}, 24(3-4):351--361, 1977.

\bibitem[Par95]{pareigis1995braiding}
Bodo Pareigis.
\newblock On braiding and dyslexia.
\newblock {\em Journal of Algebra}, 171(2):413--425, 1995.

\bibitem[RT91]{reshetikhin1991invariants}
Nicolai Reshetikhin and Vladimir~G. Turaev.
\newblock Invariants of {$3$}-manifolds via link polynomials and quantum groups.
\newblock {\em Invent. Math.}, 103(3):547--597, 1991.

\bibitem[Sch01]{schauenburg2001monoidal}
Peter Schauenburg.
\newblock The monoidal center construction and bimodules.
\newblock {\em J. Pure Appl. Algebra}, 158(2-3):325--346, 2001.

\bibitem[SY24]{shimizu2024exact}
Kenichi Shimizu and Harshit Yadav.
\newblock {Commutative exact algebras and modular tensor categories}.
\newblock arXiv:2408.06314, 2024.

\bibitem[Tur92]{turaev1992modular}
Vladimir~G. Turaev.
\newblock Modular categories and {$3$}-manifold invariants.
\newblock {\em Internat. J. Modern Phys. B}, 6(11-12):1807--1824, 1992.
\newblock Topological and quantum group methods in field theory and condensed matter physics.

\bibitem[TW13]{Tsuchiya:2012ru}
Akihiro Tsuchiya and Simon Wood.
\newblock The tensor structure on the representation category of the {$\mathcal{W}_p$} triplet algebra.
\newblock {\em J. Phys. A}, 46(44):Paper No. 445203, 40 pp., 2013.

\bibitem[Wal24]{walton2024symmetries}
Chelsea Walton.
\newblock {\em Symmetries of algebras}, volume~1.
\newblock 619 Wreath Publishing, Oklahoma City, OK, 2024.

\bibitem[Woo10]{Wood}
Simon Wood.
\newblock Fusion rules of the {$\mathcal{W}_{p,q}$} triplet models.
\newblock {\em J. Phys. A}, 43(4):Paper No. 045212, 18 pp., 2010.

\bibitem[Xu96]{Xu}
Xiaoping Xu.
\newblock Twisted modules of coloured lattice vertex operators superalgebras.
\newblock {\em Quart. J. Math. Oxford Ser. (2)}, 47(186):233--259, 1996.

\end{thebibliography}

\end{document}